\documentclass{amsart}
\usepackage{rotating,amsmath,amssymb,amscd}

\newtheorem{theorem}{Theorem}[section]
\newtheorem{corollary}[theorem]{Corollary}
\newtheorem{lemma}[theorem]{Lemma}
\newtheorem{proposition}[theorem]{Proposition}
\newtheorem{problem}[theorem]{Problem}

\theoremstyle{definition}
\newtheorem{definition}[theorem]{Definition}
\newtheorem{example}[theorem]{Example}

\theoremstyle{remark}
\newtheorem{remark}[theorem]{Remark}

\numberwithin{equation}{section}
\numberwithin{table}{section}

\DeclareMathOperator*{\ric}{Ric}
\DeclareMathOperator*{\dist}{dist}
\DeclareMathOperator*{\diam}{diam}
\renewcommand{\phi}{\varphi}
\renewcommand{\epsilon}{\varepsilon}
\newcommand{\N}{\mathbb{N}}
\newcommand{\Z}{\mathbb{Z}}
\newcommand{\R}{\mathbb{R}}
\newcommand{\C}{\mathbb{C}}
\renewcommand{\P}{\mathbb{P}}

\newcommand{\vol}{\textup{vol}}

\usepackage{slashed}
\newcommand{\av}{\declareslashed{}{-}{-0.1}{0}{\int_B}\slashed{\int_B}}

\begin{document}

\title{Complete Calabi-Yau metrics from $\P^2 \# \,9 \bar{\P}^2$}
\author{Hans-Joachim Hein}
\address{Department of Mathematics, Princeton University, Princeton, New Jersey 08544}
\email{hhein@math.princeton.edu}
\date{August 27, 2010}

\begin{abstract}
Let $X$ denote the complex projective plane, blown up at the nine base points of a pencil of cubics,
and let $D$ be any fiber of the resulting elliptic fibration on $X$.
Using ansatz metrics inspired by work of Gross-Wilson 
and a PDE method due to Tian-Yau, we prove
that $X \setminus D$ admits complete Ricci-flat K{\"a}hler metrics in most de Rham
cohomology classes.
If $D$ is smooth, the metrics converge to split flat 
cylinders $\R^+ \times S^1 \times D$ at an exponential rate.
In this case, we also 
obtain a partial uniqueness result and a local description of the Einstein moduli space, 
which contains cylindrical metrics whose cross-section does not split off a circle.
If $D$ is singular but
of finite monodromy, they converge at least quadratically to flat $T^2$-submersions over flat $2$-dimensional 
cones which need not be quotients of $\R^2$.
If $D$ is singular of infinite monodromy,
their volume growth rates are $4/3$ and $2$ for the 
Kodaira types ${\rm I}_b$ and ${{\rm I}_b}^*$,
their injectivity radii decay like
$r^{-1/3}$ and $(\log r)^{-1/2}$, and their curvature tensors decay like $r^{-2}$ and $r^{-2}(\log r)^{-1}$.
In particular, the ${\rm I}_b$ examples show that the curvature estimate from Cheeger-Tian \cite{ct-einstein} 
cannot be improved in general.
\end{abstract}

\maketitle

\tableofcontents

\section{Introduction}

\noindent \emph{Background.} This paper is motivated by questions of compactness and singularity formation 
in sequences of Einstein metrics on $4$-manifolds. A favorite and classical toy model which has inspired much of the 
research in this area and related fields is
the Kummer construction of K3 surfaces: Take the standard lattice
$\Z^4 \subset \C^2$, divide $\C^2/\Z^4$ by the involution $(z,w) \mapsto -(z,w)$
in standard coordinates,
and resolve the 16 orbifold singularities of the quotient by realizing them via $(z,w) \mapsto (z^2,zw,w^2)$ as 
a neighborhood of the origin in the affine
surface $Y^2 - XZ = 0$ and then blowing up the origin in $\C^3$. 
What results is a K3, i.e.~a simply-connected
smooth complex surface with a nowhere vanishing holomorphic $(2,0)$-form $\Omega$, lifted from $dz \wedge dw$ in this case.
From \cite{k3}, the moduli space of K3 surfaces is a connected
$20$-dimensional complex analytic space, 
and every K3 has a $20$-dimensional cone of K{\"a}hler classes,
each of which contains a unique Ricci-flat K{\"a}hler ($=$ Calabi-Yau $=$ CY)
metric \cite{yau3}. All Ricci-flat metrics on the underlying smooth manifold
arise in this way \cite{hit1}, with a small
ambiguity from hyperk{\"a}hler rotation,
and hence the moduli space of Ricci-flat metrics on the K3 manifold up to scaling
and diffeomorphism is $57$-dimensional.
Using the Kummer picture, one can give a fairly accurate approximate description of these metrics in 
certain extremal limits, cf.~\cite{bm, dona,pope}:
The flat orbifold metric 
on the Kummer surface is obviously CY
with volume form $\Omega \wedge \bar{\Omega}$.
On the other hand, the minimal resolution of
$Y^2 - XZ = 0$ supports an explicit family of CY metrics (Eguchi-Hanson) with volume form $\Omega \wedge \bar{\Omega}$
which are uniformly asymptotic to the flat cone $\C^2/\Z_2$, and moreover converge to that 
cone in the Gromov-Hausdorff sense as the family parameter tends to zero. 
Thus, replacing neighborhoods of the orbifold points by truncated Eguchi-Hanson spaces produces
increasingly Ricci-flat metrics on K3.
It then turns out that there
exist honest CY metrics nearby which
converge to the flat Kummer orbifold. Viewing $|{\rm Rm}|^2$ as an energy density whose total mass
is fixed by the Chern-Gau{\ss}-Bonnet formula $\int |{\rm Rm}|^2 = 8\pi^2 \chi$ for closed Einstein
$4$-manifolds, the energy of these metrics forms Dirac masses 
around the singularities,
and zooming in
at the local maxima of $|{\rm Rm}|$ recovers the Eguchi-Hanson spaces as \textquotedblleft bubbles\textquotedblright$\;$in the limit,
each extracting $1/16$ of
the scale-invariant total energy.

This picture lies at the foundation of a vast body of research on
the structure 
of Gromov-Hausdorff limits of Einstein manifolds
and other classes of metrics solving
canonical elliptic equations, and on the bubbles that may form in the process. 
See \cite{clw, joyce1, tian}
for some fairly spectacular applications of this line of work 
to existence questions.
The results are most complete in the $4$-dimensional Einstein case 
with a uniform lower volume bound, cf.~\cite{and}.
The bubbles are then complete and Ricci-flat of maximal volume growth, and such spaces
have been almost completely classified. More precisely, from  
\cite{cheesurv, kron1, kron2, cct}, a complete, simply-connected, K{\"a}hler, Ricci-flat $4$-manifold
with $|B(x_0,r)| > \epsilon r^4$
for some $x_0 \in M$ and $\epsilon > 0$ and all $r > 0$ 
must be asymptotic to a flat cone $\C^2/\Gamma$, $\Gamma < {\rm SU}(2)$, at a rate of $r^{-4}$
and thus belong to one of a number of 
finite-dimensional families of \textquotedblleft ALE\textquotedblright$\;$spaces, all of which can be constructed
algebraically
as resolutions or deformations of the singularities
$\C^2/\Gamma$. The K{\"a}hler condition may well turn out to be unnecessary here.

The striking analogy of all this with the more classical study of Yang-Mills moduli spaces
is largely due to the assumed lower volume bound. The stronger
nonlinearity of the Einstein equations compared to Yang-Mills really only makes itself felt when this bound
is dropped, thus allowing for the possibility of \emph{collapsing}.
In the world of bounded 
sectional rather
than Ricci curvature, this phenomenon is 
understood
to be caused by a controlled sort of contraction
along
the orbits of
local actions of nilpotent Lie groups \cite{cfg}. As for a first impression in the 
Einstein case, we can again look at the Kummer picture.
In \cite{hitchin}, Hitchin suggested 
resolving the
flat orbifolds
$(\R^k \times T^{4-k})/\Z_2$, $k = 1,2,3$,
to construct new complete
CY metrics 
with $r^k$ volume growth.
With a little good will, one can then imagine these hypothetical spaces to form as bubbles in
degenerating Ricci-flat metrics on K3
when collapsing the flat orbifold in the Kummer construction along $4-k$ directions 
onto $T^k/\Z_2$. A related folklore conjecture 
stipulates that \textquotedblleft most\textquotedblright$\;$complete CY surfaces, or at least those
that actually do occur as bubbles, 
are locally asymptotic to a product $\R^k \times T^{4-k}$. The cases with $k = 3,2,1$ have been
dubbed ALF, ALG, and ALH.

Crepant resolution and other constructions have produced examples for each of these types 
\cite{atiyah-hitchin, bm, cherkis-hitchin, hitchin, lebrun, santoro, ty1},
and also there is work of Biquard-Minerbe \cite{bm, minerbe}
\cite{minerbe2} towards
a possible classification of ALF spaces, the basic conjecture
being that such metrics must live on resolutions or deformations
of $\C^2/\Gamma$ with $\Gamma < {\rm SU}(2)$ cyclic or dihedral, except for 
finitely many sporadic examples of a slightly different nature. However, one major issue with the ALF, G, H picture
is that, unlike the ALE case, there seems to be no clear understanding as to what tameness conditions 
have to be imposed in order for a complete Ricci-flat $4$-manifold of less than maximal volume growth
to belong to one of these categories. In terms of counterexamples, 
K{\"a}hler is not immediate \cite{riem-schw}, K{\"a}hler does not imply finite topology \cite{kron0},
and it is possible to have K{\"a}hler and finite topology,
but ${\rm inj} = 0$ and fractional volume growth \cite{ty1},
although these last examples are not of finite energy, by Santoro \cite{santoro1}. On a related note,
there does not seem to be a single known case of collapsing of Einstein metrics in any dimension
where the bubbles have been identified rigorously.

In this paper, we produce new families of complete
CY surfaces which may help to shed light on some of these issues.
Removing anticanonical curves from rational elliptic surfaces, 
we construct sets of ALG and ALH spaces
which we conjecture to be exhaustive up to certain deformations,
but also two series of CY manifolds
with $r^{4/3}$ and $r^2$ volume growth and ${\rm inj}$ decaying like $r^{-1/3}$ and $(\log r)^{-1/2}$, respectively,
yet with $|{\rm Rm}|$ comparable to 
$r^{-2}$ and $r^{-2}(\log r)^{-1}$, and hence of finite energy. 
See Theorem \ref{main} for existence and Theorem \ref{main1} for the asymptotics.
In the ALH case, we prove a partial uniqueness result (Theorem \ref{alh}),
and we find and integrate
all infinitesimal ALH Ricci-flat deformations (Theorem \ref{alh2}). In this last part,
 we rely heavily on Kovalev \cite{kovalev2}, but we aim to clarify certain issues.  \medskip\

\noindent \emph{Statement of results.} We need two bits of terminology for the existence result. 

\begin{definition}\label{ratell}
A \emph{rational elliptic surface} is the blow-up of $\P^2$ in the base points of a pencil of cubics, i.e.~a
family $sF + t G = 0$, $(s:t) \in \P^1$, where $F,G$ are smooth cubics intersecting in $9$ points with 
multiplicity. Blowing up these points, if needed repeatedly, produces 
an elliptic fibration $f: X \to \P^1$, $X = \P^2 \#\,9\bar{\P}^2$.
\end{definition}

In fact, $f = |{-K_X}|$, so for each fiber $D = f^{-1}(p)$, to be understood as a divisor, there exists precisely
one meromorphic $2$-form $\Omega$ up to
scale with ${\rm div}(\Omega) = -D$, and all meromorphic $2$-forms on $X$ arise in this way.
Cf.~Section 4.1 for a more thorough discussion of the complex geometry of such surfaces, 
in particular, of the structure of the singular fibers of $f$. Put $M := X \setminus D$ and 
fix a small disk $\Delta = \{|z| < 1\} \subset \P^1$ with $z(p) = 0$ such that all fibers of $f$ over $\Delta^*$ are smooth.

\begin{definition}\label{bad} A \emph{bad 2-cycle} in $M$ is one that arises from the following process 
up to orientation and isotopy. Consider the topological monodromy representation of $\pi_1(\Delta^*) = \Z$ 
in the mapping class group of any fiber $F$ over $\Delta^*$.
Take a simple loop $\gamma \subset F$ such that $[\gamma] \in H_1(F,\Z)$
is indivisible and invariant under the monodromy, and 
move $\gamma$ around the puncture by lifting a simple 
loop $\gamma' \subset \Delta^*$ up to every point in $\gamma$ such that the union of the 
translates of $\gamma$ is a $T^2$ imbedded in $f^{-1}(\gamma')$. \end{definition}

While proving Proposition \ref{extend}, we will see that $H_2(X|_{\Delta^*},\Z)$ is generated by $[F]$ and the classes
of the bad cycles. Up to orientation, there are two bad cycles if the monodromy is trivial, 
one if the generator of the monodromy is conjugate to 
$(\begin{smallmatrix} 1 & b \\ 0 & 1\end{smallmatrix})$, $b \in \N$, and none at all otherwise. Moreover,
bad cycles are boundaries in $X$.

\begin{theorem}[Existence]\label{main} Let $\omega$ be any K{\"a}hler metric on $M$
such that $\int_M \omega^2 < \infty$ and $\int_C \omega = 0$ for all bad $2$-cycles $C$. For instance,
$\omega$ could be the restriction to $M$
of a K{\"a}hler metric on $X$.
Then there exists $\alpha_0 > 0$ possibly depending on $X$, $\Omega$, $\omega$
such that for all $\alpha > \alpha_0$ there exists a complete
Calabi-Yau metric $\omega_{\rm CY}$ on $M$ with
$\omega_{\rm CY}^2 = \alpha\Omega \wedge \bar{\Omega}$ and 
 such that $\omega_{\rm CY} - \omega$ is $d$-exact on $M$.
\end{theorem}

The existence of $\omega_{\rm CY}$ follows from a method due to Tian-Yau \cite{ty1, ty2} for solving
the Monge-Amp{\`e}re equation $(\omega_0 + i\partial\bar{\partial}u)^2 = \alpha\Omega \wedge \bar{\Omega}$
with a complete background metric $\omega_0$ on $M$, which has to satisfy
$\omega_0^2 \approx \alpha\Omega\wedge\bar{\Omega}$ in a certain sense.
We construct $\omega_0$ by exploiting the elliptic fibration structure in a neighborhood of
$D$, modifying an idea from Gross-Wilson \cite{gw}. We refer to the end of this introduction for details.
From a more careful analysis of $\omega_0$ and the PDE, one can then deduce asymptotics for $\omega_{\rm CY} = \omega_0 + i\partial\bar{\partial}u$.
To state the result, we need some more terminology.
  
\begin{definition}\label{alhdef} Let 
$\epsilon, \delta, \ell > 0$, $\theta \in (0,1]$, $\tau \in \mathfrak{H}/{\rm PSl}(2,\Z)$, where $\mathfrak{H}$ is the upper half plane. 
Let
$g_{\epsilon,\tau}$ denote the unique flat metric of area $\epsilon$ and modulus
$\tau$ on $T^2$. Let $g$ be a complete Riemannian metric on an open $4$-manifold $N$.

(i) The metric $g$ is called ALG($\delta,[\theta,\epsilon,\tau]$) if there exist $r_0 > 0$, a compact subset $K \subset N$, 
and an imbedding $\Phi: S(\theta,r_0) \times T^2 \hookrightarrow N \setminus K$ with a dense image, where
$S(\theta,r_0) := \{z \in \C: |z| > r_0, \, 0 < {\rm arg}\, z < 2\pi \theta\}$, such that 
\begin{equation}|\nabla_{g_{\rm flat}}^k(\Phi^*g - g_{\rm flat})|_{g_{\rm flat}} \leq 
C(k) |z|^{-\delta - k}\end{equation} for all $k \in \N_0$, where $g_{\rm flat} := g_\C \oplus g_{\epsilon,\tau}$.

(ii) The metric $g$ is called ALH if there exist $\delta > 0$, a compact subset $K \subset N$,
and a diffeomorphism
$\Phi: \R^+ \times T^3 \to N \setminus K$ such that
\begin{equation}|\nabla_{g_{\rm flat}}^k(\Phi^*g - g_{\rm flat})|_{g_{\rm flat}} \leq C(k)e^{-\delta t}\end{equation}
for all $k \in \N_0$, where $g_{\rm flat} := dt^2 \oplus h$ for some flat metric $h$ on $T^3$. 
More specifically, we say that $g$ is ALH$(\ell,\epsilon,\tau)$ 
if $g$ is ALH with $h = \ell^2 d\phi^2 \oplus g_{\epsilon,\tau}$ with respect to some topological splitting $T^3 = S^1 \times T^2$,
with $\phi \in S^1 = \R/2\pi \Z$ and $g_{\epsilon,\tau}$ as above.
\end{definition}

Thus an ALG space converges to a twisted product of a flat $2$-cone and a flat $2$-torus, where 
the cone is \emph{not} restricted to be a quotient of $\R^2$.
As for ALH spaces,
the ones coming from Theorem \ref{main} all have an isometrically split cross-section $S^1 \times T^2$,
the length $2\pi\ell$ of the circle being an intrinsic parameter of the ALH metric. On the other hand, 
we will see in Theorem \ref{alh2}(ii) that such spaces admit small Ricci-flat ALH deformations whose 
cross-sections do no longer split.

The \emph{monodromy} referred to in the next theorem is the topological monodromy as in Definition \ref{bad},
viewed as a conjugacy class in ${\rm Sl}(2,\Z)$ and labelled according to the so-called 
\emph{Kodaira types}, which in turn are in one-to-one correspondence with the possible 
topological types of the deleted fibers $D$, cf.~Section 4.1.

\begin{theorem}[Asymptotics]\label{main1} Let $X$, $f$, $D$, $M$, $\Omega$, $\Delta$, $\omega$, $\alpha$, $\omega_{\rm CY}$ be 
as above and let $\epsilon$ denote the area of the fibers of $f$ with respect to either $\omega$
or $\omega_{\rm CY}$.

\textup{(i)} If $D$ is
smooth, then $\omega_{\rm CY}$ is \textup{ALH}$(\ell,\epsilon,\tau)$, $\ell^2 \epsilon = i \alpha \int_D R \wedge \bar{R}$,
where $R$ is the residue of $\Omega$ along $D$, normalized as $z \Omega = dz \wedge R$, and $\tau$ is the modulus of $D$.

\textup{(ii)} If $D$ is singular with finite order monodromy, then $\omega_{\rm CY}$ is \textup{ALG}$(2,[\theta,\epsilon,\tau])$ 
with $(\theta,\tau)$ depending on the Kodaira monodromy type as follows, 
$\zeta_3 = \exp(\frac{\textup{\begin{tiny}$2\pi i$\end{tiny}}}{\textup{\begin{tiny}$3$\end{tiny}}})$:\medskip\

\begin{center}
\begin{tabular}{c c c c c c c c}
\hline
\textup{type} & ${\rm I}_0^*$ & ${\rm II}$ & ${\rm II}^*$ & ${\rm III}$ & ${\rm III}^*$ & ${\rm IV}$ & ${\rm IV}^*$ \medskip\ \\ 
$\theta$ & $\frac{1}{2}$ & $\frac{1}{6}$ & $\frac{5}{6}$ & $\frac{1}{4}$ & $\frac{3}{4}$ & $\frac{1}{3}$ & $\frac{2}{3}$ \medskip\ \\ 
$\tau$ & \textup{any} & $\zeta_3$ & $\zeta_3$ & $i$ & $i$ & $\zeta_3$ & $\zeta_3$ \\ \hline 
\end{tabular}
\end{center}\smallskip\

\textup{(iii)} If $D$ is singular with infinite monodromy, then there are two possibilities.

$\bullet$ Kodaira type ${\rm I}_b$, monodromy conjugate to $(\begin{smallmatrix} 1 & b \\ 0 & 1\end{smallmatrix})$, $1 \leq b \leq 9$:
For any fixed $x_0$, $\omega_{\rm CY}$ satisfies $|B(x_0,s)| \sim s^{4/3}$
for $s \gg 1$, ${\rm inj}(x) \sim |B(x,1)| \sim r(x)^{-1/3}$, $r(x)$ $:=$ $1 + {\rm dist}(x_0,x)$, $|{\rm Rm}| \sim r^{-2}$,
and $|\nabla^k{\rm Rm}|\lesssim_k r^{-2-k}$ for all $k \in \N$. Moreover, $\omega_{\rm CY}$ has
a unique tangent cone at infinity given by the half-line $\R^+$.

$\bullet$ Kodaira type ${\rm I}_b^*$, monodromy conjugate to $-(\begin{smallmatrix} 1 & b \\ 0 & 1 \end{smallmatrix})$, $1 \leq b \leq 4$:
$\omega_{\rm CY}$ satisfies 
$|B(x_0,s)| \sim s^{2}$ for $s \gg 1$,
${\rm inj}(x) \sim |B(x,1)| \sim (\log r(x))^{-1/2}$, $|{\rm Rm}| \sim r^{-2}(\log r)^{-1}$,
and $|\nabla^k{\rm Rm}| \lesssim_k r^{-2-k}(\log r)^{-1}$ for all $k \in \N$. Moreover, $\omega_{\rm CY}$ has
a unique tangent cone at infinity given by $\R^2/\Z_2$.
\end{theorem}

\begin{remark}\label{sect1rem} (i) In the ALG and ${\rm I}_b^*$ cases, the parameter $\alpha$ from Theorem \ref{main},
while corresponding to an infinitesimal
Einstein deformation, has very likely no geometric significance and can be changed by flowing along the 
scaling vector field $r \partial_r$ on the
tangent cone at infinity. On the other hand, in the ALH and 
${\rm I}_b$ cases, $\alpha$ corresponds to an intrinsic length scale of $\omega_{\rm CY}$.

(ii) Elliptic fibrations are called \emph{isotrivial} if all smooth fibers are isomorphic, 
cf. Example \ref{isoex}
for a careful discussion of isotrivial rational elliptic surfaces. Isotrivial ALG spaces with a type I$_0^*$, II, III, or IV fiber at infinity
can also be obtained from crepant resolution
constructions, as in \cite{bm, santoro}, or heuristically in \cite{stringy, hitchin}. 
The cones on the base are then honest quotients of $\R^2$ by cyclic group actions.
 
(iii) The I$_b$ manifolds appear to be the first examples of complete Ricci-flat four-manifolds to 
saturate the $r^{-2}$ curvature bound from Cheeger-Tian \cite{ct-einstein}.

(iv) For certain $(X,D)$ with $D$ smooth, Biquard-Minerbe \cite{bm} found cyclic groups
of (anti-)automorphisms of $X$ acting freely on $M$ and isometrically with respect to
$g_{\rm CY}$ for proper choices of $\omega$. This yields asymptotically cylindrical Ricci-flat spaces
with any one of the five non-trivial orientable quotients of $T^3$ as cross-section.\hfill $\Box$\end{remark}
 
\begin{remark} Recall that $\omega_{\rm CY} = \omega_0 + i\partial\bar{\partial}u$ with a complete background metric $\omega_0$
and a well-behaved potential $u$, where $\omega_0$ coincides with a Gross-Wilson style \cite{gw} semi-flat metric at infinity.
Thus, $\omega_0$ is Ricci-flat with $\omega_0^2 = \alpha\Omega\wedge\bar{\Omega}$ outside some compact subset of $M$,
and exactly flat when restricted to smooth fibers there.

(i) Cherkis-Kapustin \cite{cherkis} 
gave arguments
to the effect that certain moduli spaces of instantons appearing in physics are rational elliptic surfaces with an I$_b^*$ 
removed, $0 \leq b \leq 4$, the natural $L^2$ metric being hyperk{\"a}hler and semi-flat at infinity.
 
(ii) In the ALG and ALH cases, the deviation of $\omega_0$ from being flat is typically at least as large as
$\omega_{\rm CY} - \omega_0$. However, if the elliptic fibration is isotrivial, then $\omega_0$ is indeed flat and
we can replace ALG$(2)$ by ALG$(2 + \frac{1}{\theta} -\delta)$ for any $\delta > 0$.\hfill $\Box$
\end{remark}

We now discuss our partial uniqueness result for ALH metrics. What we do here 
is similar
to work of Joyce \cite{joyce} on crepant resolutions of $\C^m/\Gamma$, see also \cite{vanc, cvc2, cvc3}.
In these papers uniqueness is proved assuming that the coordinate system at infinity is \emph{fixed},
i.e.~that the two complete CY metrics on a given complex manifold 
which one would like to prove to be equal are already asymptotic \emph{as tensor fields}. We make 
a small elementary improvement on this if the cross-section splits as a product.

\begin{definition}\label{ed} Let $(N,g)$ be a complete Riemannian four-manifold which is ALH as in
Definition \ref{alhdef}(ii). We say that a tensor field $T$ on $N$ is \emph{exponentially decaying},
or ED for short,
if $|\nabla^k T| \leq C(k) e^{-\epsilon t}$ for some $\epsilon > 0$ and all $k \in \N_0$.
\end{definition}

\begin{definition}\label{cpt} Let $(N,g,J)$ be complete K{\"a}hler and ALH$(\ell,\epsilon,\tau)$
with respect to $\Phi: \R^+ \times S^1 \times T^2 \rightarrow N \setminus K$.
We say that $(N,J)$ is $\Phi$-\emph{compactifiable} if there exist a compact manifold $Y$, a complex structure $\bar{J}$ on $Y$, a $\bar{J}$-holomorphic $2$-torus $E \subset Y$, and a diffeomorphism
$\iota: N \to Y\setminus E$ with $\iota^*\bar{J} = J$ such that $\iota \circ \Phi \circ \mu$ extends
as an imbedding of $\Delta \times T^2$ into $Y$, where $\mu(\exp(-\frac{t}{\ell} + i\phi),x) := (t,\phi,x)$.
\end{definition}

Nordstr{\"o}m \cite{no2} has proved that every simply-connected
ALH$(\ell,\epsilon,\tau)$ Calabi-Yau surface
becomes $\Phi$-compactifiable after hyperk{\"a}hler rotation if necessary.

\begin{theorem}[ALH Partial Uniqueness]\label{alh} 
Let $(M_i, g_i)$, $i = 1,2$, be Ricci-flat ${\rm ALH}$ Riemannian $4$-manifolds with
parallel orthogonal almost-complex structures $J_i$ and corresponding K{\"a}hler forms $\omega_i$.
Let $\Psi: M_1 \to M_2$ be a diffeomorphism 
such that $\Psi^*J_2 = J_1$ and $\Psi^*[\omega_2] = [\omega_1]$ in the sense of de Rham cohomology.

{\rm (i)} If the $M_i$ are diffeomorphic to the complement of a smooth fiber in a rational elliptic surface 
and if $\Psi^*g_2 - g_1 = {\rm ED}$, then $\Psi^*g_2 = g_1$.

{\rm (ii)} If the $(M_i, g_i)$ are ${\rm ALH}(\ell,\epsilon,\tau)$, 
the $(M_i, J_i)$ are $\Phi_i$-compactifiable with respect to coordinate systems $\Phi_i$, and if $\mu^{-1} \circ 
\Phi_2^{-1} \circ \Psi \circ \Phi_1 \circ \mu$ 
extends as a diffeomorphism from $\Delta \times T^2$ to itself with $\mu$ as in Definition \ref{cpt}, then $\Psi^* g_2 - g_1 = {\rm ED}$.
\end{theorem}

Part (i) is a corollary of cylindrical Hodge theory while Part (ii) 
simply expresses the fact that, under the given conditions, $\Psi$ must converge
to a tame automorphism of $\C^* \times \C/(\Z + \Z\tau)$, hence 
to an isometry of the flat metric 
$\ell^2 |d\log z|^2 \oplus g_{\epsilon,\tau}$.

\begin{theorem}[ALH Moduli]\label{alh2} Let $M$ be diffeomorphic to the complement of some smooth fiber in a
rational elliptic surface, let $g$ be ${\rm ALH}$ hyperk{\"a}hler on $M$, and define
$\mathcal{H}$ $:=$ $\{h \in C^\infty({\rm Sym}^2T^*M): h$ $=$ parallel $+$ {\rm ED} outside a compact set$\}$.

{\rm (i)} The kernel inside $\mathcal{H}$ of the linearization at $g$ of the Ricci curvature operator
is of dimension $29$
modulo arbitrary Lie derivatives and multiples of $g$. The subspace of ${\rm ED}$ bilinear forms has dimension $24$,
and the quotient is naturally isomorphic to the deformation space $\mathfrak{sl}(3,\R)/\mathfrak{so}(3)$ of the flat metric on the cross-section.

{\rm (ii)} If $(X,J)$ is a generic rational elliptic surface and if $(M,g,J)$ with $M = X \setminus D$
is one of the {\rm ALH} Calabi-Yau spaces with an isometrically split cross-section $T^3 = S^1 \times T^2$ constructed in Tian-Yau \cite[Theorem 5.2]{ty1}, 
then all kernel elements from Part {\rm (i)} are tangent to some
curve of {\rm ALH} hyperk{\"a}hler metrics on $M$.
\end{theorem}

Koiso \cite{koiso} showed that CY metrics have unobstructed deformations on compact manifolds.
Kovalev \cite{kovalev2} sketched
an extension of this theory to the asymptotically cylindrical case, but there are certain difficulties in applying this work unmodified, largely due to the fact that
$D$ is flat and correspondingly $\pi_1(D) \neq 0$. For example, it may not
be clear from the outset that the complex structure stays compactifiable up to hyperk{\"a}hler rotation, although this does turn out to be the case. On the other hand, 
the cross-section \emph{can} be deformed so as to break the isometric splitting.

In the setting of Part (ii), the deformation space can be understood as follows:
There are complex moduli from varying the pencil ($\dim_\C = 8$, ED) or $D$ ($\dim_\C = 1$, not ED). Among the remaining moduli of the metric, those that keep
the splitting of the cross-section correspond
to varying the K{\"a}hler class ($\dim_\R = 8$, ED) or the length of the $S^1$ factor inside a fixed K{\"a}hler class ($\dim_\R = 1$, not ED). The non-split moduli
($\dim_\R = 2$, not ED) are still realized inside each fixed K{\"a}hler class.

The restriction to Tian-Yau type metrics in Part (ii) is for technical convenience only. In fact, there are various ways of proving 
that ALH spaces have unobstructed deformations in general, e.g.~by reducing the problem to
Nordstr{\"o}m's work \cite{nordstrom} on asymptotically cylindrical $G_2$-manifolds by taking products with a flat $3$-torus.\medskip\

\noindent \emph{Plan of the paper.} As mentioned above, we will apply 
a non-compact version due
to Tian-Yau \cite{ty1, ty2} of Yau's solution \cite{yau3} of the Calabi conjecture.
The set-up
needed for this is an open complex manifold $M$
together with a 
complete K{\"a}hler metric $\omega_0$ whose 
Ricci
form satisfies $\rho(\omega_0) =  i\partial\bar{\partial} f$
with $f \to 0$ at infinity. 
A convenient way 
to fulfill this condition is to look for complex manifolds $M$ 
with global holomorphic 
volume forms $\Omega$ and complete K{\"a}hler metrics $\omega_0$ 
such that $e^f := (i\textup{\begin{footnotesize}$^{m^2}$\end{footnotesize}}\Omega \wedge \bar{\Omega})/\omega_0^m$  
converges to $1$ at infinity, where $m = \dim_\C M$.
If this is the case, then $\Omega$
obviously has some sort of a pole at infinity.
We may then hope to find a well-behaved solution 
$u$ for the Monge-Amp{\`e}re equation $(\omega_0 + i\partial\bar{\partial}u)^m = e^f \omega_0^m$,
so that $\omega_0 + i\partial\bar{\partial}u$ would be complete and Calabi-Yau 
with asymptotic geometry comparable to $\omega_0$.

From a PDE point of view, the difficult part in extending Yau's work is to provide 
Sobolev inequalities for the $L^\infty$ estimate,
and sometimes (in practice: if the volume growth of $\omega_0$ is $r^2$
or less), 
bounded solutions may not even exist unless one imposes an integrability
condition, $\int (e^f - 1)\omega_0^m = 0$.
Strictly speaking, the 
existence result
from Tian-Yau \cite{ty1} suffices
for our applications; however,
we offer a self-contained proof 
of some new, sharp weighted Sobolev inequalities (Proposition \ref{weight-sob}),
hence a sharp existence theorem (Proposition \ref{solve}),
which hold in some generality and may be of independent interest.
Also, while the solutions produced by this 
method are bounded in all $C^k$ norms,
we need some decay for the asymptotics in Theorem \ref{main1}. 
This will be proved in Section 3.4 by elementary arguments.

This leaves us with writing down background metrics $\omega_0$ on the complement of 
a given
fiber in a rational elliptic fibration. To understand the nature of this problem,
it is helpful to recall the general approach in Tian-Yau \cite{ty1,ty2}.
Thus, $M = X \setminus D$, where $X$ is a compact
K{\"a}hler manifold and $D$ is a reduced divisor such that there exists a meromorphic volume form 
$\Omega$ on $X$ with ${\rm div}(\Omega) = -\alpha D$ for some pole order $\alpha \in \N$. 
Our goal is 
to find a complete K{\"a}hler metric $\omega_0$ such
that 
$\omega_0^m$ 
converges to $i\textup{\begin{footnotesize}$^{m^2}$\end{footnotesize}}\Omega \wedge \bar{\Omega}$ when approaching $D$.
Note that there can be no \emph{canonical} choice for $\omega_0$, not even if $X = \P^2$, $D = \P^1$
($\alpha = 3$), as both the Euclidean metric
and Taub-NUT \cite{lebrun} are CY on $\C^2$ with 
identical volume forms. The strategy in Tian-Yau then is to assume $D$ is \emph{smooth} and
build $\omega_0$ by separation of variables, $\omega_0 = i\partial\bar{\partial}F(|\Omega|^2) + \omega$,
where $\omega$ is a K{\"a}hler metric on $X$, 
$F: \R^+ \to \R^+$, and $|\Omega|^2 = h(\Omega,\Omega)$ for a hermitian metric $h$
on $K_X$, all three to be determined from the equation.

One finds that if $D$ is ample, then $\omega$ is irrelevant
to leading order and $\omega_0$ will be the complex cone over a positive K{\"a}hler-Einstein metric 
on $D$ if $\alpha > 1$ and if such a metric exists,
and a sort of degenerate cone over a CY metric on $D$ if $\alpha = 1$, with the circles normal to $D$ 
collapsing and the directions tangent to $D$ blowing up.
On complex surfaces, the $\alpha > 1$ case only occurs for $D = \P^1$ in $X = \P^2$, 
producing the flat metric on $\C^2$
(to obtain nontrivial ALE spaces, one has to consider orbifolds), while $\alpha = 1$ occurs
for every del Pezzo surface $X = \P^2 \# \,k\bar{\P}^2$, $0 \leq k \leq 8$, 
and $D$ the strict transform of a smooth cubic in $\P^2$.
On the other hand, if $D$ is not ample but still moves in a pencil, then $\alpha = 1$,
$\omega$ does play a role to leading order,
and $\omega_0$ limits
to a Ricci-flat cylinder $\R^+ \times S^1 \times (D,\omega|_D)$.
If $m = 2$, this 
happens precisely when $X$ is rational elliptic, 
so then the Tian-Yau ansatz yields an ALH space.

In this paper, we drop the separation of variables idea altogether, and focus on the elliptic fibration structure
rather than the smoothness of $D$. Our background metrics $\omega_0$
will at first only be defined in a neighborhood of $D$, and
will be obtained as Riemannian submersions with flat metrics
on the fibers, as in Gross-Wilson \cite{gw}. 
It is interesting to note that even though the neighborhood boundaries of an I$_b$ 
fiber
and of a smooth cubic in $\P^2 \#\, (9-b)\bar{\P}^2$
are homeomorphic, the Tian-Yau ansatz
in that case looks quite different:
Here, each $T^2$ fiber carries a monodromy invariant $S^1$ which collapses, 
while there, the normal $S^1$ bundle collapses onto $T^2$.

We construct our background metrics $\omega_0$, study their geometric properties, and thus 
conclude the proof of Theorems \ref{main} and \ref{main1}, in Section 4.
Section 4.1 provides a detailed review of rational elliptic surfaces,
with examples for the isotrivial case, and of Kodaira's theory of singular fibers.
Section 4.2 constructs semi-flat metrics
explicitly in terms of the periods. We treat fibers of any dimension, and 
show that the induced metric $h$ on the base satisfies
${\rm Ric}(h) = g_{\rm WP}$, the Weil-Petersson metric being pulled back from the Siegel upper half plane in this case.
Section 4.3 studies the geometry of the semi-flat metrics in the surface case,
extracting formulas for the periods from Kodaira's work.
Section 4.4 shows how to glue the semi-flat metrics, which live in a neighborhood of the deleted fiber, 
with given metrics on $M$ so that the Monge-Amp{\`e}re integrability condition
$\int (\omega_0^2 - \alpha\Omega \wedge \bar{\Omega}) = 0$ is satisfied.

Section 5 gives the proofs of Theorems \ref{alh} and \ref{alh2}, after briefly reviewing the necessary Hodge theoretic
machinery on asymptotically cylindrical manifolds, and in Section 6 we discuss some open problems 
and a potential application of the new Calabi-Yau metrics to the collapsing/bubbling problem on K3 (Problem \ref{bubble}). \medskip\

\noindent \emph{Notation.} Generic constants are $\geq 1$ and $\sim$ means comparable
up to generic factors. $A = A(x,r,s) := B(x,s) - \bar{B}(x,r)$, $\mu A := A(x,\mu^{-1}r, \mu s)$, $\mu > 1$. 
For $\delta \geq 0$, $v_{\textup{\begin{tiny}$\delta$\end{tiny}}}(r)$ [$a_{\textup{\begin{tiny}$\delta$\end{tiny}}}(r)$] denotes the volume [area] of a ball [sphere] of radius $r$ in the 
model space of constant curvature $-\delta^{\textup{\begin{tiny}$2$\end{tiny}}}$.
For functions $\phi > 0$ and $u$ on a set $X$,
$u_{\textup{\begin{tiny}$X,\phi$\end{tiny}}}$ and $\|u\|_{{\textup{\begin{tiny}$X,\phi,p$\end{tiny}}}}$
denote the average and $L^{\textup{\begin{tiny}$p$\end{tiny}}}$ norm of $u$ over $X$
with respect to $\phi\,d\vol$. If $1 \leq p < n$, then $\alpha_{\textup{\begin{tiny}$np$\end{tiny}}} := n/(n-p)$. 
Our notation for K{\"a}hler manifolds is $(M,\omega)$, $n = {\rm dim}_{\textup{\begin{tiny}$\R$\end{tiny}}} M = 2m$,
$\omega(X,Y) = \langle JX,Y\rangle$, $\omega = i g_{\textup{\begin{tiny}$j\bar{k}$\end{tiny}}}
dz^{\textup{\begin{tiny}$j$\end{tiny}}} \wedge d\bar{z}^{\textup{\begin{tiny}$k$\end{tiny}}}$, hence 
$g_{\textup{\begin{tiny}$j\bar{k}$\end{tiny}}} = g(\partial_{z^j},\partial_{\bar{z}^k})
= \frac{\textup{\begin{tiny}$1$\end{tiny}}}{\textup{\begin{tiny}$2$\end{tiny}}}\delta_{\textup{\begin{tiny}$jk$\end{tiny}}}$ on $\C^{\textup{\begin{tiny}$m$\end{tiny}}}$, 
and we agree that if $g$ is a $J$-invariant inner product on an $\R$-vector space $V$, then it induces
a hermitian metric $h(X,Y) := g(\bar{X},Y)$ on $V^{1,0}$, so identifying $V^{1,0} = V$ as usual, 
$h(X,Y) = g(\frac{\textup{\begin{tiny}$1$\end{tiny}}}{\textup{\begin{tiny}$2$\end{tiny}}}(X + iJX), \frac{\textup{\begin{tiny}$1$\end{tiny}}}{\textup{\begin{tiny}$2$\end{tiny}}}(Y - iJY)) 
= \frac{\textup{\begin{tiny}$1$\end{tiny}}}{\textup{\begin{tiny}$2$\end{tiny}}}(g(X,Y) + i \omega(X,Y))$ if $X,Y \in V$.
The Ricci form, $\rho = \rho(\omega)$, $\rho(X,Y) = {\rm Ric}(JX,Y)$, $[\rho] = 2\pi c_{\textup{\begin{tiny}$1$\end{tiny}}}(M)$, 
and the K{\"a}hler differential operators, $d^{\textup{\begin{tiny}$c$\end{tiny}}} := i(\bar{\partial} - \partial)$,
$dd^{\textup{\begin{tiny}$c$\end{tiny}}} = 2i\partial\bar{\partial}$,
$\Delta := 2g^{\textup{\begin{tiny}$j\bar{k}$\end{tiny}}} \partial_{z^j}\partial_{\bar{z}^k}$.\medskip\

\noindent \emph{Acknowledgments.} I am deeply grateful to my advisor, Gang Tian, for introducing me to this circle of questions
and for his guidance. I thank Richard Bamler, Ronan Conlon, Aaron Naber, Johannes Nordstr{\"o}m and Yikuan Yan
for many illuminating discussions, John Lott 
for conversations leading to the formulation of Problem \ref{bubble}, an anonymous referee for indicating a shortcut which made Section 2.2 much more readable, and the 
BICMR at Peking University for its support and great hospitality during completion of a preliminary version of this paper in Spring 2009.\newpage

\section{Weighted Sobolev inequalities}

\subsection{Overview} Our goal here is to prove Proposition \ref{weight-sob}, which provides
weighted Sobolev inequalities with almost sharp weights on a class of complete Riemannian manifolds 
$(M^n,g)$ 
with
$\ric \geq -Cr^{-2}$ at infinity and homogeneous volume growth 
of some order $\beta > 0$. More precisely, we are assuming the following condition:

\begin{definition}\label{sob-beta} $(M^n,g)$ is called SOB$(\beta)$, $\beta > 0$, if there exist
$x_0 \in M$ and $C \geq 1$ such that $A(x_0, r, r + s)$ is connected for all 
$r \geq 100$ and $s > 0$,  $|B(x_0, s)| \leq C s^\beta$ for all $s \geq 1$, 
and $|B(x,\frac{1}{2}r(x))| \geq \frac{1}{C}r(x)^\beta$ and ${\rm Ric}(x) \geq -Cr(x)^{-2}$ if $r(x) \geq 1$.
\end{definition}

The connectedness hypothesis implies that all far-out annuli satisfy
a Neumann-type Poincar{\'e} inequality, and follows from $M$ having one end and $b_1(M) < \infty$. We claim that if SOB$(\beta)$ 
holds for some $\beta > 0$, then the following is true:

\begin{proposition}\label{weight-sob} \textup{(i)} For all $\epsilon > 0$ there exists a step function $\psi_\epsilon: M \to \R^+$,
with $\psi_\epsilon \sim (1+r)^{-\max\{\beta,2\} - \epsilon}$,
such that for all $\alpha \in [1,\alpha_{n2}]$ and all $u \in C^\infty_0(M)$,
\begin{equation}\label{ns}\left(\int_M |u - u_\epsilon|^{2\alpha} \, (1 + r)^{\alpha(\min\{\beta-2,0\} - \epsilon) - \beta} \, d\vol\right)^{\frac{1}{\alpha}}
\leq C(\epsilon) \int_M |\nabla u|^2 \, d\vol,\end{equation}
where $u_\epsilon$ denotes the average of $u$ with respect to the finite measure $\psi_\epsilon \, d\vol$.

\textup{(ii)} If $\beta > 2$, then, for all $\alpha \in [1,\alpha_{n2}]$ and all $u \in C^\infty_0(M)$,
\begin{equation}\label{ds} \left(\int_M |u|^{2\alpha} \, (1 + r)^{\alpha(\beta - 2) - \beta} d\vol\right)^{\frac{1}{\alpha}}
\leq C \int_M |\nabla u|^2 \, d\vol.\end{equation}
\end{proposition}

\begin{corollary}
Let $(M^n,g)$ be such that there exist a compact set $K \subset M$, a closed manifold
$(N^{n-1},h)$, a diffeomorphism $\Phi: (1,\infty) \times N \to M \setminus K$, and $C \geq 1$
such that $\frac{1}{C} g_{\rm cone} \leq \Phi^*g \leq C g_{\rm cone}$ for the metric $g_{\rm cone} = dt^2 \oplus t^2 h$
on $\R^+ \times N$. Then
\begin{equation}\label{conesob} \int_M |u|^{2\alpha} (1 + r)^{\alpha(n-2)-n} \, d\vol \leq C \int_M |\nabla u|^2 \, d\vol\end{equation}
for all $u \in C^\infty_0(M)$ and $\alpha \in [1,\alpha_{n2}]$. \hfill $\Box$\end{corollary}

\begin{remark}\label{rem2} (i) Tian-Yau \cite{ty1} prove (\ref{ns}) in a
rather more general setting but with a weaker weight, and (\ref{ds})
intersects with results of Minerbe \cite{min} for ${\rm Ric} \geq 0$, some of whose methods, originally from
Grigor{'}yan and Saloff-Coste \cite{gsc},
we borrow. In \cite{ty2}, using fairly special methods, Tian-Yau prove (\ref{conesob})
for $\alpha = \alpha_{n2}$ if $N$ is a finite cover of a minimal submanifold
of a round sphere. Very recent, independent work of van Coevering \cite{cvc3} contains 
a general proof of (\ref{conesob}) along similar lines.

(ii) SOB$(\beta)$ and (\ref{ds}) are closely related to non-parabolicity, i.e.~the existence of a 
positive Green's function: SOB$(\beta)$ implies Li-Tam's condition (VC) from \cite{li-tam}, so
according to their results, $M$ is non-parabolic if and only if $\beta > 2$, in which case all Green's functions are 
$\sim r^{2-\beta}$ at infinity. Also, by Carron \cite{carron}, a complete $M$ 
is non-parabolic if and only if there exists $\phi \in C^\infty(M)$, $\phi > 0$, with
$\int u^2\phi \leq C \int |\nabla u|^2$ for all $u$ $\in$ $C^\infty_0(M)$, which corresponds to
(\ref{ds}) for $\alpha = 1$. Third, assuming SOB$(\beta)$ with $\beta > 2$ and a mild extra condition, 
we will apply (\ref{ds}) in Section 3.4 to solve
$\Delta u = f$ if $|f| \leq C r^{-\mu}$, $\mu \in (2,\beta)$, with $|u| \leq C(\delta)r^{2-\mu + \delta}$ for all $\delta > 0$. \hfill $\Box$ \end{remark}

The key to the proof is a
volume comparison method introduced by Gromov \cite{gro1}.
Buser \cite{bus}, Maheux and Saloff-Coste \cite{msc},
and Cheeger-Colding \cite{cc} have applied this method to 
derive various Neumann-type $L^p$ Poincar{\'e} and Sobolev inequalities, 
for the most part on geodesic balls, and there is a Dirichlet-type $L^1$ Sobolev inequality
for geodesic balls due to Anderson \cite{and} based on the same principle.
In Sections 2.2 and 2.3, we show that the same idea
also yields Dirichlet-type inequalities on other domains than balls, especially
on annuli (Corollary \ref{anncor}). This already gives a global Sobolev inequality
under lower Ricci bounds
(Corollary \ref{varo}) which implies Gallot's inequality \cite{gallot} in the compact case.
In Section 2.4, we then combine Corollary \ref{anncor} with patching methods
from Grigor'yan and Saloff-Coste \cite{gsc} and Minerbe \cite{minerbe}, and with the
Cheeger-Colding segment inequality \cite{cc}, to prove Proposition \ref{weight-sob}.

\subsection{Isoperimetric estimates \`a la Gromov and Anderson}
Let $M$ denote an $n$-dimensional Riemannian manifold without boundary, which may be incomplete.
The goal of this section is to prove a general estimate (Lemma \ref{cover})
for the volume
 of a domain in $M$ in terms of the area of its boundary, 
based on an idea of Gromov
 \cite[\S 6.C, Appendix C]{gro1} and Anderson
\cite[Section 4]{and}. 

\begin{definition}\label{good}
Let $X$ be a metric-measure space, $Y \subset X$. Let 
$\mathcal{B}$ be a covering of $Y$ by metric balls.
We say $\mathcal{B}$ is \emph{$(\epsilon,r_0)$-good}, $0 < \epsilon \leq \frac{1}{2}$, $r_0 > 0$,
if ${\rm center}(B) \in Y$, $\min\{|B \setminus Y|,|B \cap Y|\} \geq \epsilon |B|$, and ${\rm radius}(B) \leq r_0$ for all $B \in \mathcal{B}$.
\end{definition}

\begin{lemma}\label{cover}
Let $\Omega \subset M$ be open with a smooth boundary.
Let $\mathcal{B}$ be an $(\epsilon,r_0)$-good covering of $\Omega$.
Let $N$ be the open $5r_0$-neighborhood of $\Omega$, and assume $\bar{N}$ is compact. If ${\rm Ric} \geq -(n-1)\delta^2$ on $N$ with $0 \leq r_0\delta \leq \Lambda$, and if $\alpha \geq 1$,
then 
\begin{equation}\label{iso}\frac{|\Omega|^{\frac{1}{\alpha}}}{|\partial \Omega|}
\leq C(n, \epsilon, \Lambda)
\sup_{B \in \mathcal{B}}|B|^{\frac{1}{\alpha} - 1} {\rm radius}(B).\end{equation}
\end{lemma}

\begin{proof}
Choose finitely many balls $B_i = B(x_i, r_i) \in \mathcal{B}$,
$i = 1, ..., k$, such that the $2B_i$ are pairwise disjoint but still $\Omega \subset \bigcup 5B_i$. This can be
achieved through a standard Vitali type procedure:
Since $\bar{\Omega}$ is compact and $B \setminus \Omega \neq \emptyset$ for all $B \in \mathcal{B}$, 
there exists a finite sub-collection $\mathcal{B}'$ $\subset$ $\mathcal{B}$ that still covers $\Omega$.
We choose a $B_1 \in \mathcal{B}'$ of maximal radius, and we take $B_{i+1} \in \mathcal{B}'$
to be of maximal radius among all those $B \in \mathcal{B}'$ for which
$2B$ is disjoint from $2B_1, ..., 2B_i$. Then, indeed, $\bigcup 5B_i \supset \bigcup \mathcal{B}' \supset \Omega$.\smallskip\

\noindent \emph{Key estimate:} We have $|B_i| \leq C(n,\epsilon,\Lambda)r_i|\partial \Omega \cap 2 B_i|$ for all $i = 1, ..., k$.\smallskip\

If this is true, then (\ref{iso}) follows immediately, by noting that
$$|\Omega|^{\frac{1}{\alpha}} \leq C(n,\Lambda)\sum |B_i|^{\frac{1}{\alpha}}
\leq C(n,\epsilon,\Lambda) \sum |B_i|^{\frac{1}{\alpha} - 1} r_i |\partial\Omega \cap 2 B_i|.$$
 
\noindent \emph{Proof of the key estimate.} The basic idea is as follows. Fix $z \in B_i \cap \Omega$ and project $B_i \setminus \Omega$ onto $\partial\Omega$ along minimal geodesics emanating from $z$. Define $\Sigma^{\rm first}$ $\subset$ $\partial\Omega \cap 2B_i$ to consist of the first points of entry into $B_i \setminus \Omega$ of the \textquotedblleft light rays\textquotedblright$\;$involved in
this projection. By integrating the infinitesimal version of the Bishop-Gromov volume comparison inequality
along the maximal sub-segment with endpoints in $\bar{B}_i\setminus \Omega$ of 
each such light ray, and then integrating across $\Sigma^{\rm first}$, one eventually finds that
$$\epsilon |B_i| \leq |B_i\setminus\Omega| \leq \int_{\Sigma^{\rm first}} 
\frac{v_\delta(2r_i) - v_\delta({\rm dist}(z,y))}{a_\delta({\rm dist}(z,y))} \; d{\rm area}(y).$$
This proves the key estimate if $r_i \leq C(n,\epsilon,\Lambda)\, {\rm dist}(z, \Sigma^{\rm first})$. But if this inequality \emph{fails} for \emph{all} $z \in B_i \cap \Omega$, then intuitively one should be able to find points $z' \in B_i \setminus \Omega$ such that the argument does go through with $z,B_i \setminus \Omega$ replaced by $z',B_i \cap \Omega$.

We now work out the details. Specifically, we show that $\min\{|B_i \setminus \Omega|, |B_i \cap \Omega|\} \leq 2 (v_\delta(4r_i)/a_\delta(2r_i))
|\partial \Omega \cap 2B_i|$, and this then suffices by $(\epsilon,r_0)$-goodness.

Define $X_1, X_2 \subset (B_i \cap \Omega) \times (B_i \setminus \Omega)$ as follows: $X_1 := \{(z, z'):$
\emph{there is a unique minimal geodesic} $\gamma$ \emph{from} $z$ \emph{to} $z'$, \emph{and this has the following
properties: it intersects} $\partial\Omega$ \emph{only transversely, and if} $y$ \emph{denotes
the first point along} $\gamma$, \emph{counted from} $z$, \emph{where} $\gamma$ \emph{intersects} $\partial\Omega$,
\emph{then} ${\rm dist}(z,y) \geq {\rm dist}(y,z')\}$, and almost verbatim for $X_2$, with the only difference that the inequality is now reversed. 
Then $X_1 \cup X_2$
has full measure in
 $(B_i \cap \Omega) \times (B_i \setminus \Omega)$, and so one of $X_1, X_2$ must
have at least half measure.

$\bullet$ $X_1$ \emph{has at least half measure.} By Fubini, there must be a $z \in B_i \cap \Omega$ such
that $Z' := \{z' \in B_i \setminus \Omega: (z, z') \in X_1\}$ has at least half measure in $B_i \setminus \Omega$.
We now bound $|Z'|$ above by projecting onto $\partial \Omega$ along
minimal geodesics from $z$, and integrating the infinitesimal Bishop-Gromov inequality along these geodesics.

Let $\Sigma^{\rm first}$ be the set of all $y \in \partial \Omega \cap 2B_i$ which occur as the first intersection with $\partial\Omega$ of the geodesic $\gamma$ from $z$ to some $z' \in Z'$ as in the definition of $X_1$. Thus, for all $y \in \Sigma^{\rm first}$ there exists a unique minimal geodesic $\gamma_y$ from $z$ to $y$, 
and we can write $\gamma_y(t) = \exp_{z}(v_yt)$, where $v_y \in T_z M$ is a uniquely determined unit vector.

Define $d_1, d_2: \Sigma^{\rm first} \to \R^+$ by $d_1(y) := {\rm dist}(z,y)$ and
$d_2(y) := \min\{\sup\{t > 0: \gamma_y(d_1(y) + t) \in Z'\}$, $\sup\{t > 0: \gamma_y$ \emph{is minimal on} 
$[0, d_1(y) + t]\}\}$. 

For every $z' \in Z'$ then, there exists
a unique $y \in \Sigma^{\rm first}$ such that $z' = \gamma_y(t)$ for some $d_1(y) < t \leq d_1(y) + d_2(y)$. 
Thus, if we define an imbedding $\Phi: U \hookrightarrow M$ with
$U$ $:=$ $\{(y,t) \in \Sigma^{\rm first} \times \R^+: 
d_1(y) < t < d_1(y) + d_2(y)\}$ and $\Phi(y,t) := \gamma_y(t)$, then ${\rm clos}(\Phi(U))$
contains $Z'$. On the other hand, a fairly standard calculation yields
$$\Phi^*(d{\rm vol}_M)|_{(y,t)}
= \frac{J(tv_y)}{J(d_1(y)v_y)} \cos \alpha_y\, d{\rm area}_{\partial\Omega} \wedge dt,$$
where $J(w) := |w|^{n-1} \det d\,{\exp_z}|_w$ for all $w \in T_{z}M$, 
and where $\alpha_y$ denotes the angle between $\dot{\gamma}_y(d_1(y))$ and the
exterior unit normal to $\partial\Omega$
at $y$.

We integrate over $U$ and apply relative volume comparison:
\begin{align*} 
|\Phi(U)| &\leq \int_{\Sigma^{\rm first}} \int_{d_1(y)}^{d_1(y) + d_2(y)} \frac{J(tv_y)}{J(d_1(y)v_y)} \; 
dt\; d{\rm area}(y)\\
&\leq \int_{\Sigma^{\rm first}} \int_{d_1(y)}^{d_1(y) + d_2(y)} \frac{a_\delta(t)}{a_\delta(d_1(y))} \; dt \; d{\rm area}
(y)\\
&= \int_{\Sigma^{\rm first}} \frac{v_\delta(d_1(y) + d_2(y)) - v_\delta(d_1(y))}{a_\delta(d_1(y))} \; d{\rm area}(y).
\end{align*}
We now estimate the integrand as follows:
$$\frac{v_\delta(d_1 + d_2) - v_\delta(d_1)}{a_\delta(d_1)} \leq \frac{v_\delta(d_1+d_2)}{a_\delta(d_1)} \leq  \frac{v_\delta(2d_1)}{a_\delta(d_1)} \leq \frac{v_\delta(4r_i)}{a_\delta(2r_i)},$$ 
because $d_2 \leq d_1 \leq 2r_i$ and because $s \mapsto v_\delta(2s)/a_\delta(s)$ is non-decreasing.
Altogether then, $|B_i \setminus \Omega| \leq 2|Z'| \leq 2|\Phi(U)| \leq 2(v_\delta(4r_i)/a_\delta(2r_i))|\Sigma^{\rm first}|$, as needed.

$\bullet$ $X_2$ \emph{has at least half measure.} By Fubini again, there now exists a $z' \in B_i \setminus \Omega$ such that $Z := \{z \in B_i \cap \Omega: (z, z') \in X_2\}$ has at least half measure in $B_i \cap \Omega$. Let $\Sigma^{\rm last}$ denote the set of all $y \in \partial\Omega \cap 2B_i$ which occur 
as the \emph{last} intersection of $\gamma^{-1}$ with $\partial\Omega$, where
$\gamma$ is the geodesic from some $z \in Z$ to $z'$ as in the definition of $X_2$.
For each $y \in \Sigma^{\rm last}$, there then exists a unique minimal geodesic $\gamma_y(t) = \exp_{z'}(tv_y)$ from $z'$ to $y$, and the rest of the argument will be the same as above up to replacing 
$z,Z',\Sigma^{\rm first}$ by $z',Z,\Sigma^{\rm last}$ and switching the interior and exterior normals of $\Omega$.
\end{proof}

\begin{remark} A similar sort of reasoning yields the following result: If
$B = B(x,r) \subset M$ is such that $3r < {\rm diam}(M)$, $4\bar{B}$ is compact,
and ${\rm Ric} \geq -\Lambda r^{-2}$ on $4B$ with $\Lambda \geq 0$, then $\frac{1}{C}r|\partial B| \leq |B| \leq Cr|\partial B|$ with a uniform
$C = C(n,\Lambda)$ which in particular does not depend on the collapsedness of $B$.
Here, the upper bound follows as before, by projecting $B$ onto $\partial B$ along minimal geodesics
issuing from a point on $\partial(3B)$. For
the lower bound, we sweep out a subset $B^* \subset B$ by joining $x$ to all smooth points of $\partial B$, 
express $|B^*|$ in polar coordinates, and then use
Bishop-Gromov in the form $J(tv)/J(rv) \geq a_\delta(t)/a_\delta(r)$, $v \in T_xM$, $|v| = 1$, $t \leq r$, to estimate
 from below.
\end{remark}

\subsection{Dirichlet-type Sobolev inequalities on balls and annuli} 
Assuming lower Ricci bounds, subsets of geodesic balls or annuli admit $(\epsilon,r_0)$-good ball coverings
for controlled values of 
$\epsilon,r_0$ (Lemma \ref{annlem}). By Lemma \ref{cover}, this implies a 
Dirichlet-isoperimetric, hence a Dirichlet-Sobolev inequality (Corollary \ref{anncor}). As a corollary, 
we obtain a global Gallot- or Varopoulos-type inequality (Corollary \ref{varo}).

Note that $M$ is still not required to be complete. In the following lemma,
we fix a point $x_0 \in M$ and put $B(r) := B(x_0,r)$, $A(r_1, r_2) := A(x_0, r_1, r_2)$.

\begin{lemma}\label{annlem}
{\rm (i)} If $s < \frac{1}{4}\diam(M)$, $B := B(s)$, $9\bar{B}$ is compact, $\ric$ $\geq$ $-(n-1)\delta^2$
on $9B$,
then for all $\Lambda \geq 0$ there exists $\epsilon = \epsilon(n,\Lambda) > 0$ such that 
if $0 \leq \delta s \leq \Lambda$, then for all $x \in \Omega \subset B$
there exists $0 < r_{x,\Omega} \leq 4s$ with $|B(x,r_{x,\Omega}) \setminus \Omega| = \epsilon |B(x,r_{x,\Omega})|$.

{\rm (ii)} If $r > 6t$, $A$ $:=$ $A(r, r + s)$,
$\bar{B}(r + 2s + 2t)$ is compact, $\ric$ $\geq$ $-(n-1)\delta^2$ on 
$A(r-6t, r + 2s + 2t)$, then for all $\Lambda, N \geq 0$ there exists an $\epsilon = \epsilon(n,\Lambda,N) > 0$ such
that if $0$ $\leq$ $\delta t$ $\leq$ $\Lambda$, $s \leq Nt$,
then for all $x \in \Omega \subset A$ there exists $0 < r_{x,\Omega} \leq r_x^* := \dist(x_0,x) - r + 2t$ 
with $|B(x,r_{x,\Omega}) \setminus \Omega| = \epsilon|B(x,r_{x,\Omega})|$. 
\end{lemma}

\begin{proof} (i) Fix a minimal geodesic $\gamma$ from $x$ to some point $x'' \in \partial B(x,4s)$
and put $x' := \gamma(3s)$. Then, by volume comparison, 
$|B(x,4s) \setminus B| \geq |B(x',s)| \geq \epsilon |B(x,4s)|$
for some definite $\epsilon = \epsilon(n,\Lambda) > 0$. The claim then follows by continuity.

(ii) Fix a minimal geodesic $\gamma$ from $x_0$ to $x$.
Let $x_i := \gamma(r + (2i - 3)t)$,
$B_i$ $:=$ $B(x_i, t)$ for $i = 1, ..., k$, where $k$ is maximal with $r + (2k - 3)t \leq \dist(x_0,x)$;
notice $k \leq \frac{1}{2}(N + 3)$. Then $B(x, r_x^*) \setminus A \supset B_1$ and $3B_i \supset B_{i+1}$ for 
$i = 1, ..., k - 1$.
Volume comparison shows $|B_i| \geq \epsilon |3B_i|$ for all $i$,
so $|B(x, r_x^*) \setminus A| \geq  \epsilon |B_k|$ by induction. By
volume comparison again,
$|B_k| \geq \epsilon |B(x_k, 2t + r_x^*)|$,
and $B(x_k, 2t + r_x^*) \supset B(x,r_x^*)$ by maximality of $k$.
Thus, $|B(x,r_x^*) \setminus A| \geq \epsilon|B(x,r_x^*)|$, so we conclude as in (i). \end{proof}

Thus, for all $\Omega \subset B, A$ in ${\rm (i)},{\rm (ii)}$, the covering $\{B(x,r_{x,\Omega}): x \in \Omega\}$ is $(\epsilon,r_0)$-good
with $\epsilon = \epsilon(n,\Lambda)$, $r_0 = 4s$, and $\epsilon = \epsilon(n,\Lambda,N)$, 
$r_0 = s + 2t$, respectively, so 
then Lemma \ref{cover} provides a uniform isoperimetric estimate for all such $\Omega$,
which
by the following well-known result implies a Dirichlet-type Sobolev inequality for $B,A$.
 
\begin{lemma}[{\cite[Theorem 9.1]{li}}]\label{fed}
For all $\Omega_0 \subset M$ open and precompact, $\alpha \geq 1$,
$$\sup\left\{\frac{\|u\|_\alpha}{\|\nabla u\|_1}: 
u \in C^\infty_0(\Omega_0), \; u \neq 0\right\}
= \sup\left\{\frac{|\Omega|^{\frac{1}{\alpha}}}{|\partial\Omega|}: \Omega
\subset \Omega_0 \;\text{open},\;\partial\Omega\;\text{smooth}\right\}.$$
\end{lemma} 

\noindent Also, recall that by H{\"o}lder's inequality, for all $p \geq 1$ and $\alpha \geq 1$,
\begin{equation}\label{lp}
{\rm DS}(\Omega_0, p, \alpha) := \sup \frac{\|u\|_{\alpha p}}{\|\nabla u\|_p} \leq \frac{\alpha p}{\alpha'} \,
\sup \frac{\|u\|_{\alpha'}}{\|\nabla u\|_1}
\;\;\text{if}\;\,1 - \frac{1}{\alpha'} = \frac{1}{p}\left(1 - \frac{1}{\alpha}\right),
\end{equation}
where the suprema are over all $u \in C^\infty_0(\Omega_0)$ with $u \neq 0$.

\begin{corollary}\label{anncor} {\rm (i)} If $20 s < {\rm diam}(M)$, $B(20s)$ is precompact, ${\rm Ric} \geq -(n-1)\delta^2$ on $B(20s)$
with $0 \leq \delta s \leq \Lambda$, then, for all $p \in [1,n)$ and $\alpha \in [1,\alpha_{np}]$,
\begin{equation}\label{andball}
{\rm DS}(B(s), p, \alpha) \leq C(n,p,\Lambda) v_\delta(s)^{\frac{1}{n}}|B(s)|^{\frac{1}{p}(\frac{1}{\alpha}-1)}.
\end{equation}

{\rm (ii)} If $20s < \min\{r, {\rm diam}(M)\}$, $B(r + 20s)$ is precompact,
and $\ric$ $\geq$ $-(n-1)\delta^2$ on $A(r - 20s, r + 20s)$ with
$0 \leq \delta s \leq \Lambda$, then, for all $p \in [1,n)$ and $\alpha \in [1,\alpha_{np}]$,
\begin{equation}\label{sob-p}
{\rm DS}(A(r,r+s), p, \alpha) \leq C(n,p,\Lambda) v_\delta(s)^{\frac{1}{n}} \sup_{x \in A(r,r+s)} |B(x,s)|^{\frac{1}{p}(\frac{1}{\alpha}-1)}.
\end{equation}
\end{corollary}

\begin{proof} (i) Combining Lemmas \ref{cover}, \ref{annlem} with Lemma \ref{fed}, (\ref{lp}),
\begin{align*}{\rm DS}(B(s), p, \alpha) \leq C(n,p,\alpha,\Lambda)\sup
\left\{v_\delta(r)^{\frac{1}{n}} |B(x,r)|^{\frac{1}{p}(\frac{1}{\alpha}-1)} : x \in B(s), r \in (0,4s]\right\}
\end{align*}
for all $p \geq 1$, $\alpha \geq 1$. The right-hand side is finite iff either $p \in [1,n)$ and $\alpha \in 
[1,\alpha_{np}]$, or $p \geq n$. The second case is covered by the first, and in the first case, (\ref{andball})
follows by volume comparison. The proof of (ii) is similar.\end{proof}

\begin{remark}\label{rem1}{\rm
Corollary \ref{anncor}(i) recovers Anderson \cite[Theorem 4.1]{and}, which has its roots in Gromov \cite{gro1}. We have streamlined
Anderson's proof, and tweaked it so as to apply to more general domains. Note that (\ref{andball}) simply means that if
$B = B(s)$, then, for all $\alpha \in [1,\alpha_{np}]$ and all $u \in C^\infty_0(B)$,
\begin{equation}\label{andsobsimp}\left(\av\, |u|^{\alpha p} \right)^{\frac{1}{\alpha p}} \leq C(n,p,\Lambda) s \left(\av\, |\nabla u|^p\right)^{\frac{1}{p}}.\end{equation}
Croke's sharp isoperimetric inequality in \cite{croke} would imply (\ref{andsobsimp}) with an additional collapsedness factor of $s^n|B|^{-1}$ on the right; the improvement afforded 
by (\ref{andsobsimp}) was crucial for the main $\epsilon$-regularity theorem proved in \cite{and}.
By Maheux and Saloff-Coste 
\cite[Th{\'e}or{\`e}me 1.1]{msc}, (\ref{andsobsimp}) holds as well for all 
$u \in C^\infty(B)$ with mean value zero.}
\end{remark}

We conclude with a global Sobolev inequality
similar to the ones given in Hebey \cite[Theorem 3.14, Proposition 3.22]{heb} 
which follows immediately from Corollary \ref{anncor}.
In the compact case, this implies
a familiar result of Gallot \cite[Th{\'e}or{\`e}me 6.16]{gallot}.

\begin{corollary}\label{varo}
Let $M^n$ be complete without boundary.
If $|B(x,r)| \geq \epsilon > 0$ for 
all $x \in M$, $100r < \diam(M)$, 
$\ric \geq -\Lambda r^{-2}$, $\Lambda \geq 0$, and $p \in [1,n)$, $\alpha \in [1,\alpha_{np}]$,
\begin{equation}\|u\|_{\alpha p} \leq C(n,p,\Lambda)\epsilon^{\frac{1}{p}(\frac{1}{\alpha}-1)}(r\|\nabla u\|_p + 
\|u\|_p),\end{equation}
for all $u \in C^\infty_0(M)$ if $M$ is open, and for all $u \in C^\infty(M)$ if $M$ is closed.
\end{corollary}

\begin{proof} Fix $x_0 \in M$ and define $r_m := (1 + \frac{m}{100})r$, $A_m := A(x_0, r_m, r_{m+1})$ for $m \in \N_0$.
Take $\chi_0 \in C^\infty_0(B(x_0, r_2))$ with $0 \leq \chi_0 \leq 1$, 
$\chi_0 \equiv 1$ on $B(x_0, r_1)$, $|\nabla \chi_0| \leq 200 r^{-1}$. 
For $m \geq 1$ take $\chi_m$ $\in$ $C^\infty_0(A_{m-1} \cup \bar{A}_m \cup A_{m+1})$
with $0 \leq \chi_m \leq 1$, $\chi_m \equiv 1$ on $A_m$, $|\nabla \chi_m| \leq 200 r^{-1}$.
By Corollary \ref{anncor}, for all $m \in \N_0$,
\begin{align*}\begin{split}
\|u\chi_m\|_{\alpha p} \leq C(n,p,\Lambda) r \epsilon^{\frac{1}{p}(\frac{1}{\alpha}-1)}\|\nabla(u\chi_m)\|_p.
\end{split}\end{align*}
Take $p$-th powers, sum over $m$, and take $p$-th roots. \end{proof}

\subsection{Proof of the weighted Sobolev inequalities} This section concludes the proof of 
Proposition \ref{weight-sob} by combining Corollary \ref{anncor} with some analysis
on graphs, similar to what was developed in \cite{gsc, minerbe} in much greater generality.
Assume that $M^n$ is complete noncompact without boundary and
SOB$(\beta)$ holds for some $\beta > 0$. 
Generic constants are allowed to depend on $n$, $\beta$, and the $C$ from SOB$(\beta)$.

The key step in passing from the Dirichlet-Sobolev bounds in Corollary \ref{anncor} 
to the global Sobolev bounds in Proposition \ref{weight-sob} is a certain Neumann-type Poincar{\'e}
inequality. We need a continuous and a discrete version. The continuous
one follows from a special case of the Cheeger-Colding segment inequality \cite[Theorem 2.11]{cc}:

\begin{lemma}[Cheeger-Colding]\label{segment} Let $B = B(x_0,r) \subset (M^n,g)$. If $2B$ is precompact and
${\rm Ric} \geq -\Lambda r^{-2}$ on $2B$ with $\Lambda \geq 0$, then, for all $u \in C^\infty(B)$,
\begin{equation}\label{seg}\int_B |u-u_B|^2 \leq C(n,\Lambda)r^2\int_{2B}|\nabla u|^2.\end{equation}
\end{lemma}

Passing from $B$ to $2B$ is inevitable in their proof because a segment
between two points in $B$ will usually only be contained in $2B$.
Buser \cite[Lemma 5.1]{bus} gives an $L^p$ Neumann-type Poincar{\'e} inequality for every $p \geq 1$
which does not require doubling the radius, but (\ref{seg}) is all we really need, and surprisingly simple to show. Also, (\ref{seg}) has a useful discrete counterpart with a closely related proof.

If $V$ is a countable set and $E$ is a set of $2$-element subsets of $V$, we call $G = (V,E)$ a \emph{graph}.
We say $G$ is \emph{connected} if any two vertices $x,y \in V$, $x \neq y$,
can be joined by a \emph{path},
i.e.~a set $\gamma \subset V$ which can be listed as $\gamma = \{\gamma_0, ..., \gamma_m\}$ such that
$\{\gamma_0, \gamma_m\} = \{x,y\}$ 
and $\{\gamma_{i-1},\gamma_i\} \in E$ for $i = 1, ..., m$. 
If $u: V \to \C$, then we write $u_x := u(x)$,
and we define $|\nabla u|^2: V \to \R$ by setting $|\nabla u|^2_x := \sum_{y \in V: \{x,y\} \in E} |u_x-u_y|^2$.

\begin{lemma}\label{disc} Let $G$ be connected and let $w: V \to \R^+$ satisfy $\sum_{x \in V} w_x = 1$. For all 
$x,y \in V$, $x \neq y$, fix a path $\gamma = \gamma_{xy} = \gamma_{yx}$ as above 
and write $m = m_{xy} = m_{yx}$.  Define $\bar{w}: V \to \R^+$ by 
$\bar{w}_z := \sum_{\{x,y\}: z \in \gamma_{xy}} m_{xy} w_{x} w_{y}$.
Then, for all $u: V \to \C$,
\begin{equation}\label{discp}\sum_{x \in V} w_x u_x = 0 \Longrightarrow \sum_{x \in V} w_x |u_x|^2 \leq \sum_{x \in V} \bar{w}_x |\nabla u|_x^2. \end{equation}
\end{lemma}

\noindent To prove this, multiply the lhs by $\sum w_x = 1$ and use $\sum w_xu_x = 0$ to obtain 
\begin{align*}\sum_{x} w_x|u_x|^2 &= \sum_{\{x,y\}} w_xw_y|u_x - u_y|^2 
\leq \sum_{\{x,y\}} m_{xy}w_xw_y\sum_{i = 1}^{m_{xy}} |u_{\gamma_{{xy},i}} - u_{\gamma_{xy,i-1}}|^2. \end{align*}

\noindent \emph{Proof of Proposition \ref{weight-sob}.} 
Fix $\alpha \in [1,\alpha_{n2}]$ and $\eta := 1.001$. Corollary \ref{anncor} yields
\begin{eqnarray}\label{ballsob}&&B := B(x_0,1000) \Longrightarrow {\rm DS}(\eta B,2,\alpha) \leq C,\\
\label{annsob}&&A := A(x_0,r,\eta r),\, r \geq 1000 \Longrightarrow {\rm DS}(\eta A,2,\alpha) \leq C r^{1 + \frac{\beta}{2}(\frac{1}{\alpha}-1)},\end{eqnarray}
where we recall that $\mu A(x_0,r,s) := A(x_0,\mu^{-1}r, \mu s)$ if $\mu \geq 1$, $r < s$.
The remainder of the proof is in three steps. 
In Step 0, we use Lemmas \ref{segment}, \ref{disc} 
to establish weak Neumann-Poincar{\'e} inequalities, (\ref{ball-neu-poinc}), (\ref{ann-neu-poinc}),
on slightly larger domains $\eta B_\kappa$, $\eta A_\kappa$, which
together with (\ref{ballsob}), (\ref{annsob}) imply weak Neumann-Sobolev
inequalities (\ref{ball-neu-sob}), (\ref{ann-neu-sob}) for $\eta B$, $\eta A$.
In Steps 1 and 2, we apply these four Neumann-type inequalities from Step 0, 
and Lemma \ref{disc} again, to prove (\ref{ns}) and (\ref{ds}), respectively.\medskip\

\noindent \emph{Step 0: Weak Neumann-type Poincar{\'e} and Sobolev on certain balls and annuli.} 
For $\kappa \geq 1$ define $B_\kappa := B(x_0,1000\kappa)$ and $A_\kappa := A(x_0,r,\eta\kappa r)$.
We first show that
\begin{eqnarray}\label{ball-neu-poinc}&&\|u-u_{\eta B_\kappa}\|_{\eta B_\kappa,2} \leq C(\kappa)\|\nabla u\|_{\eta^2  B_\kappa, 2},\\
\label{ann-neu-poinc}&&\|u-u_{\eta A_\kappa}\|_{\eta A_\kappa, 2}\leq C(\kappa) r\|\nabla u\|_{\eta^2 A_\kappa, 2}.\end{eqnarray}
We write out a detailed argument for the annulus case (\ref{ann-neu-poinc}) only. 
Pick a maximal $r/2000C_0$-separated set $x_1,...,x_m$ in $\eta A_\kappa$, so that the
$B_i := B(x_i, r/1000C_0)$ cover
$\eta A_\kappa$, but the $\frac{1}{2}B_i$ are disjoint.
Notice
$|B_i|$ $\sim$ $r^\beta$ from SOB$(\beta)$, so $m \leq C(\kappa)$.
Then, from SOB$(\beta)$ and the segment inequality (\ref{seg}), for all $i$ and all $1 \leq \lambda \leq 10$,
\begin{equation}\label{ball-poinc}\int_{\lambda B_i} |u-u_{\lambda B_i}|^2 \leq C r^2 \int_{2\lambda B_i} |\nabla u|^2.\end{equation}
For any $\mu \in \R$ then,
\begin{equation}\label{dagger}\int_{\eta A_\kappa} |u - u_{\eta A_\kappa}|^2 \leq \int_{\eta A_\kappa} |u - \mu|^2 \leq 2\sum \int_{B_i} |u - u_{B_i}|^2 +
2\sum |B_i| |u_{B_i} - \mu|^2.\end{equation}
The first sum can be bounded by using (\ref{ball-poinc}) with $\lambda = 1$. For the second, we 
apply Lemma \ref{disc}, 
as follows. Construct a graph $G = (V,E)$ by taking
$V := \{1,...,m\}$ and for $i \neq j$, $\{i,j\} \in E$ if and only if $B_i \cap B_j \neq \emptyset$.
Then $G$ is connected because $\eta A_\kappa$ is.
Let $w \equiv \frac{\textup{\begin{tiny}$1$\end{tiny}}}{\textup{\begin{tiny}$m$\end{tiny}}}$, and for $i \neq j$, let $\gamma_{ij} = \gamma_{ji}$ be any path joining $i$ and $j$.
Thus, by Lemma \ref{disc}, 
if $\mu = \frac{\textup{\begin{tiny}$1$\end{tiny}}}{\textup{\begin{tiny}$m$\end{tiny}}}\sum u_{B_i}$, then the second sum in (\ref{dagger}) is bounded by
$$\sum |B_i| |u_{B_i} - \mu|^2 \leq C r^\beta \sum_i \sum_{j: B_i \cap B_j \neq \emptyset} |u_{B_i} - u_{B_j}|^2. $$
Next, for any constant $\nu \in \R$, by Cauchy-Schwarz,
\begin{align*} |u_{B_i} - u_{B_j}|^2 &\leq \frac{1}{|B_i||B_j|} \int_{B_i \times B_j} |u(x) - u(y)|^2 \, dx \, dy 
\\&\leq 4\frac{|B_i \cup B_j|}{|B_i||B_j|} \int_{B_i \cup B_j} |u(x) - \nu|^2 \, dx. \end{align*}
Now $B_i \cup B_j \subset 3B_i$ since $B_i \cap B_j \neq \emptyset$, so put $\nu = u_{3B_i}$ and apply
(\ref{ball-poinc}), $\lambda = 3$:
$$|u_{B_i} - u_{B_j}|^2 \leq C r^{2-\beta} \int_{6B_i} |\nabla u|^2.$$ 
This bounds the second sum in (\ref{dagger}), concluding the proof of (\ref{ann-neu-poinc}). The proof
of (\ref{ball-neu-poinc}) is entirely similar, dropping the $r$-dependence everywhere. 
Notice one could apply the segment inequality to $\eta B_\kappa$
directly in that case, but it will be convenient later to be integrating over $\eta^2 B_\kappa$ on the rhs of 
(\ref{ball-neu-poinc}) rather than $2\eta B_\kappa$.

To conclude, we deduce weak Neumann-type Sobolev inequalities for $\eta B$ and $\eta A$. 
Construct cut-off functions $\chi_B \in C^\infty_0(\eta B)$, $0 \leq \chi_B \leq 1$, $\chi_B \equiv 1$ on $B$,
$|\nabla \chi_B| \leq C$, and $\chi_A \in C^\infty_0(\eta A)$, 
$0 \leq \chi_A \leq 1$, $\chi_A \equiv 1$ on $A$, $|\nabla \chi_A| \leq C r^{-1}$. 
Then, setting $\kappa = 1$, (\ref{ballsob}), (\ref{ball-neu-poinc}) and (\ref{annsob}), (\ref{ann-neu-poinc}) easily imply
\begin{eqnarray}\label{ball-neu-sob}&&\|u-u_{\eta B}\|_{B,2\alpha} \leq \|\chi_B(u-u_{\eta B})\|_{\eta B, 2\alpha}
\leq C \| \nabla u\|_{\eta^2B,2} ,\\
\label{ann-neu-sob}&&\|u-u_{\eta A}\|_{A,2\alpha} \leq \|\chi_A(u - u_{\eta A})\|_{\eta A, 2\alpha} 
\leq C r^{1 + \frac{\beta}{2}(\frac{1}{\alpha} - 1)} \|\nabla u\|_{\eta^2 A, 2}.\end{eqnarray}
Together with (\ref{ball-neu-poinc}), (\ref{ann-neu-poinc}), these are what we need for Steps 1 and 2 below.  \hfill $\Box$\medskip\

Put $r_i := 1000\eta^{i}$ ($i \in \N_0$), $A_0 := B = B(x_0,r_0)$, 
$A_i$ $:=$ $A(x_0, r_{i-1}, r_{i})$ ($i \in \N$). 
Fix $\phi \in C^\infty(M)$, $\phi > 0$, to be determined, and
recall that $\|..\|_{X,\phi,p}$ denotes the $L^p$ norm on $X$ with respect to $\phi\,d\vol$. Let
$u_i := u_{\eta A_i}$ and $\phi_i$ $:=$ $\sup_{A_i} \phi$ ($i \in \N_0$).\medskip\

\noindent \emph{Step 1: Proof of \textup{(\ref{ns})}.}
For $\mu \in \R$ to be determined, consider (all sums over $\N_0$)
\begin{equation}\label{daggerdagger} 
\|u - \mu\|_{M,\phi,2\alpha}^2 \leq 2\sum \|u - u_{i}\|_{A_i,\phi,2\alpha}^2
+ 2\sum (\phi_i |A_i|)^{\frac{1}{\alpha}} |u_{i} - \mu|^2.\end{equation}
The first sum can be bounded by (\ref{ball-neu-sob}) for $i = 0$ and (\ref{ann-neu-sob}), $A = A_i$, for 
$i \geq 1$:
$$\sum \|u - u_{i}\|_{A_i,\phi,2\alpha}^2 \leq C \sum \phi_i^{\frac{1}{\alpha}} r_i^{2 + \beta(\frac{1}{\alpha} - 1)}\|\nabla u\|_{\eta^2A_i, 2}^2.$$
For the second sum in (\ref{daggerdagger}), we apply Lemma \ref{disc} to the graph
$A_0-A_1-A_2-\;\cdots$ with weights
$w_i$ $:=$ $\tilde{w}_i/\tilde{w}$, where $\tilde{w}_i := (\phi_i |A_i|)^{\textup{\begin{tiny}$1/\alpha$\end{tiny}}}$ and
$\tilde{w} := \sum \tilde{w}_i$, assuming this series converges. Thus, if 
we choose
\begin{equation}\label{average}\mu := \sum w_i u_{i}  = \int_M u\psi\, d\vol, \; \psi := \sum \frac{w_i}{|\eta A_i|}\chi_{\eta A_i}, \int_M \psi \, d\vol = 1,\end{equation}
then we can continue to estimate the second sum in (\ref{daggerdagger}) as follows:
$$\sum (\phi_i |A_i|)^{\frac{1}{\alpha}} |u_{i} - \mu|^2
\leq \frac{C}{\tilde{w}} \sum (\bar{w}_k + \bar{w}_{k+1}) |u_{k} - u_{k+1}|^2, \;\bar{w}_k := \sum_{i \leq k \leq j} (j-i) \tilde{w}_i \tilde{w}_j.$$
As before, for all $k \geq 0$ and any $\nu \in \R$, by Cauchy-Schwarz,
\begin{equation}\label{intermed}\begin{split}|u_k - u_{k+1}|^2 &\leq \frac{1}{|\eta A_k||\eta A_{k+1}|} \int_{\eta A_k \times \eta A_{k+1}} |u(x) - u(y)|^2 \, dx \, dy \\
&\leq 4 \frac{|\eta A_k \cup \eta A_{k+1}|}{|\eta A_k| |\eta A_{k+1}|}\int_{\eta A_k \cup \,\eta A_{k+1}} |u(x) - \nu|^2 \, dx. \end{split}\end{equation}
For $k = 0$, apply (\ref{ball-neu-poinc}) with $\kappa = \eta$, and 
for $k \geq 1$, (\ref{ann-neu-poinc}) with $r = r_{k-1}$, $\kappa = \eta^2$, choosing $\nu$ to be the average of $u$ over 
the appropriate domain in each case. Thus, $\|u - \mu\|_{M,\phi,2\alpha} \leq C\|\nabla u\|_{M,2}$ with $\mu$ as in (\ref{average}), provided
$\phi: M \to \R^+$ satisfies
$$\sup_{i \in \N_0} \phi_i^{\frac{1}{\alpha}} r_i^{2 + \beta(\frac{1}{\alpha} - 1)} \leq C, \; \frac{1}{C} \leq \tilde{w} \leq C, \; \sup_{k \in \N_0} (\bar{w}_k + \bar{w}_{k+1}) r_k^{2-\beta} \leq C.$$
The first condition checks if $\phi \leq C(1 + r)^{\alpha(\beta - 2) -\beta}$, the second,
if $\phi|_{A_0} \geq C^{-1}$ and $\phi \leq C(1 + r)^{-\beta - \epsilon}$ $(\epsilon > 0)$, and the third,
if $\phi \leq C(1 + r)^{\alpha(\beta-2)-\beta - \epsilon}$ $(\epsilon > 0)$. \hfill $\Box$\medskip\

\noindent \emph{Step 2: Proof of \textup{(\ref{ds})}.} We begin as in Step 1, splitting (and summing over $\N_0$)
$$\|u\|_{M,\phi,2\alpha}^2 \leq 2\sum \|u - u_i\|_{A_i,\phi,2\alpha}^2 +
2\sum (\phi_i |A_i|)^{\frac{1}{\alpha}}|u_i|^2.$$
The first sum can be bounded by (\ref{ball-neu-sob}) for $i = 0$ and (\ref{ann-neu-sob}), $A = A_i$, for $i \geq 1$:
\begin{equation}\label{ddagger}\sum \|u - u_i\|_{A_i,\phi,2\alpha}^2 \leq C\sum \phi_i^{\frac{1}{\alpha}} r_i^{2 + \beta(\frac{1}{\alpha}-1)}
\|\nabla u\|_{\eta^2A_i, 2}^2.\end{equation}
To estimate the second sum, fix $K \in \N$ and consider
\begin{equation}\label{intermed2}\sum (\phi_i |A_i|)^{\frac{1}{\alpha}}|u_i|^2 \leq
2 \sum (\phi_i |A_i|)^{\frac{1}{\alpha}}|u_i - u_{i+K}|^2 + 2\sum (\phi_i |A_i|)^{\frac{1}{\alpha}} |u_{i+K}|^2.\end{equation}
The first term here can again be bounded in terms of the Dirichlet energy:
\begin{equation}\label{ddaggerddagger}\sum (\phi_i |A_i|)^{\frac{1}{\alpha}}|u_i - u_{i+K}|^2 \leq CK \sum (\phi_i |A_i|)^{\frac{1}{\alpha}} r_i^{2-\beta} \int_{\eta^2 A_i \cup\,...\cup\, \eta^2 A_{i+K}} |\nabla u|^2.\end{equation}
To see this, write $u_i - u_{i+K}$ as a telescope sum, use Cauchy-Schwarz, and estimate each term as in (\ref{intermed}).
As a result, if $\phi \sim (1+r)^{\textup{\begin{tiny}$\alpha(\beta$\end{tiny}$-$\begin{tiny}$2)$\end{tiny}$-$\begin{tiny}$\beta$\end{tiny}}}$, then both (\ref{ddagger}), (\ref{ddaggerddagger}) 
are bounded by $CK^2 \int |\nabla u|^2$. The remaining term in (\ref{intermed2}) can be absorbed:
\begin{align*}\sum (\phi_i |A_i|)^{\frac{1}{\alpha}} |u_{i+K}|^2 &= \sum \frac{(\phi_i |A_i|)^{\frac{1}{\alpha}}}{(\phi_{i+K} |A_{i+K}|)^{\frac{1}{\alpha}}} (\phi_{i+K} |A_{i+K}|)^{\frac{1}{\alpha}} |u_{i+K}|^2 \\
&\leq C\eta^{(2-\beta)K} \sum (\phi_i |A_i|)^{\frac{1}{\alpha}}|u_i|^2,\end{align*}
so it suffices to make $K$ sufficiently large, depending only on $\eta$, $\beta$, and $C$. \hfill $\Box$

\section{A complex Monge-Amp{\`e}re equation}

We construct bounded solutions for the complex Monge-Amp{\`e}re equation  ($\C$MA) $(\omega_0 + i \partial\bar{\partial} u)^m
= e^f \omega_0^m$, under suitable decay assumptions for $f$, on certain complete K{\"a}hler manifolds $(M,\omega_0)$,
and then look at methods for improving the decay of $u$.
From the point of view of geometry,
we care about boundedness or decay for $i\partial\bar{\partial} u$ rather than $u$, but
Yau's method can only produce potentials $u$ which are bounded themselves, although
$i\partial\bar{\partial}u$ may then be expected to decay particularly fast. 

If $(M,\omega_0)$ is \emph{parabolic}, this means that we must impose an integrability condition,
$\int (e^f - 1)\omega_0^m = 0$, to kill off unbounded contributions from the leading term
in the asymptotic expansion of the Green's function,
even though geometrically reasonable solutions $u$ 
may well exist for \emph{arbitrary} $f$ of sufficient decay.
In our main application, integrability will come from replacing
$\omega_0$ by $\omega_0 + \beta$ if necessary, where $\beta \in C^\infty_0$ and
$\beta = i\partial\bar{\partial}\phi$ with $\phi$ asymptotic to a Green's function, cf.~Remark \ref{growsol}(i).

To state the existence result, let $(M,\omega_0)$ denote a complete noncompact K{\"a}hler manifold 
with a $C^{3,\alpha}$ quasi-atlas (cf.~Definition \ref{quasi} and also (\ref{holder}) for the associated H{\"o}lder spaces)
which satisfies ${\rm SOB}(\beta)$ from Definition \ref{sob-beta} for some $\beta > 0$.
 
\begin{proposition}\label{solve}
Let $f \in C^{2,\alpha}(M)$ satisfy $|f| \leq C r^{-\mu}$ for some $\mu > 2$, 
and in addition $\int (e^f - 1)\omega_0^m = 0$ if $\beta \leq 2$.
Then there exists $u \in C^{4,\bar{\alpha}}(M)$, $\bar{\alpha} \in (0,\alpha]$,
with $\int |\nabla u|^2 < \infty$ if $\beta \leq 2$,
such that $(\omega_0 + i\partial\bar{\partial} u)^m = e^f \omega_0^m$.
If moreover $f \in C^{\textup{\begin{tiny}$k,\bar{\alpha}$\end{tiny}}}_{\textup{\begin{tiny}loc\end{tiny}}}(M)$ for some $k \geq 3$,
then every such $u$ is in $C^{\textup{\begin{tiny}$k$\end{tiny}$+$\begin{tiny}$2,\bar{\alpha}$\end{tiny}}}_{\textup{\begin{tiny}loc\end{tiny}}}(M)$.
\end{proposition}

We will prove this in Sections 3.1--3.3 by applying Tian-Yau's method from \cite{ty1}, but using our
sharp Sobolev inequalities from Section 2.
In fact, the existence result from \cite{ty1} would already be sufficient for all applications in this paper
because in our setting, $f \in C^\infty_0(M)$ and $\beta \leq 2$. However,
Proposition \ref{solve} appears to be essentially sharp under the given conditions
and may thus be of independent interest.

Theorem \ref{main1} requires some amount of decay for $u$ rather than just boundedness. The analysis for this
will be carried out in Section 3.4 and is self-contained in that we do not need pseudodifferential methods
(we stay \textquotedblleft above the first indicial root\textquotedblright). To get pointwise decay, we 
employ an energy argument if $\beta \leq 2$ and Moser iteration with weights if $\beta > 2$. Derivative 
decay then follows by bootstrapping. We refer to Section 3.4 for the precise statements, which are a little
unwieldy.

\subsection{Set-up, structure of the proof}

\begin{definition}\label{quasi} Let $(M, \omega_0)$ be a complete K{\"a}hler manifold. A $C^{k,\alpha}$ 
\emph{quasi-atlas} for $(M,\omega_0)$ is a collection $\{\Phi_x: x \in A\}$, $A \subset M$,
of holomorphic local diffeomorphisms $\Phi_x: B \to M$, $\Phi_x(0) = x$, from $B = B(0,1) \subset \C^m$
into $M$ which extend
smoothly to the closure $\bar{B}$, and such that there exists a constant $C \geq 1$ such that
${\rm inj}(\Phi_x^*g_0) \geq \frac{\textup{\begin{tiny}$1$\end{tiny}}}{\textup{\begin{tiny}$C$\end{tiny}}}$, 
$\frac{\textup{\begin{tiny}$1$\end{tiny}}}{\textup{\begin{tiny}$C$\end{tiny}}}g_{\C^m} \leq \Phi_x^*g_0 \leq Cg_{\C^m}$, 
and $\|\Phi_x^*g_0\|_{C^{k,\alpha}(B)} \leq C$ for all $x \in A$, and such that for all $y \in M$ 
there exists $x \in A$ with $y \in \Phi_x(B)$ and ${\rm dist}_{g_0}(y,\partial\Phi_x(B)) \geq 
\frac{\textup{\begin{tiny}$1$\end{tiny}}}{\textup{\begin{tiny}$C$\end{tiny}}}$.
\end{definition}

Given a $C^{k,\alpha}$ quasi-atlas, we can define \emph{global} H{\"o}lder spaces of functions
$C^{l,\gamma}(M)$ for all $l \in \N_0$ and $0 \leq \gamma < 1$ by setting
\begin{equation}\label{holder}\|u\|_{C^{l,\gamma}(M)} := \sup \{\|u \circ \Phi_x\|_{C^{l,\gamma}(B)}: x \in A\}.
\end{equation}
There is an analogous definition of $C^{l,\gamma}(U)$ for any open set $U \subset M$.
We also work with \emph{local} H{\"o}lder spaces, whose definition only requires an ordinary atlas
and is in fact independent of such a choice. For a second order elliptic differential operator $L$
with coefficients in $C^{l,\gamma}(M)$, we then have
$\|u\|_{C^{l+2,\gamma}(M)} \leq C(\|Lu\|_{C^{l,\gamma}(M)} + \|u\|_\infty)$ for all $u \in C^\infty(M)$,
provided that $l \leq k-1$ and $0 < \gamma \leq \alpha$. This follows from elliptic estimates, 
together with the fact that the euclidean $C^{l,\gamma}$ norm on $B$
is comparable with the intrinsic $C^{l,\gamma}$ norm associated to $\Phi_x^*g_0$ if $l \leq k+1$ and $\gamma \leq \alpha$.
    
Tian-Yau \cite[Proposition 1.2]{ty1} provide a simple criterion for $(M,\omega_0)$ to admit a $C^{k,\alpha}$ quasi-atlas.
We restate it here and write out a slightly more detailed proof.

\begin{lemma}\label{ty} If $|{\rm Rm}| \leq C$, then there exists a quasi-atlas
which is $C^{1,\alpha}$ for all $\alpha$. If moreover
$|\nabla {\rm Scal}| + ... + |\nabla^k{\rm Scal}| \leq C$, then this quasi-atlas is even $C^{k+1,\alpha}$. 
\end{lemma}

\noindent \emph{Sketch of proof.} Fix $x_0 \in M$. 
By pulling back the K{\"a}hler structure under $\exp_{x_0}$ and rescaling,
we may assume that
${\rm inj}(x_0) \geq \frac{\pi}{2}$. Fix an orthonormal frame
$\{X_j, Y_j\}$ at $x_0$, $Y_j = JX_j$,
and extend it to $B(x_0,1)$ by parallel transport along radial geodesics.
From Jost-Karcher \cite[Satz 2.1]{joka},
there exist coordinates $\{x_j, y_j\}$ on $B(x_0,1)$ 
with $|\nabla x_j - X_j| \leq C\rho$ and $|\nabla y_j - Y_j| \leq C\rho$,
where 
$\rho := {\rm dist}(x_0,-)$.
By \cite[Satz 5.1]{joka}, there then exists $r_0 > 0$ small but definite such that
for all $0 < r \leq r_0$ there exist harmonic coordinates $\{\tilde{x}_j, \tilde{y}_j\}$ on $B(x_0,r)$ 
with $|\nabla \tilde{x}_j - \nabla x_j| \leq C r^2$. These satisfy the $C^{1,\alpha}$ requirement, but may not be holomorphic.

However, if $\tilde{z}_j := \tilde{x}_j + i {\tilde{y}}_j$, then $|\bar{\partial}\tilde{z}_j| \leq Cr$.
Solving the $\bar{\partial}$-Neumann 
problem on $B(x_0,r)$ with $r$ small but still definite,
we can construct a function $w_j$ on $B(x_0,r)$ such that
$\bar{\partial}w_j = \bar{\partial}\tilde{z}_j$ and 
$\|w_j\|_2 \leq Cr \|\bar{\partial} w_j\|_2 \leq C r^{m + 2}$.
To see this, we apply the version of H{\"o}rmander's $L^2$ estimate in Demailly \cite[Corollary 8.10]{dem},
using that $\rho^2$ is smooth and uniformly convex on $B(x_0,r)$ when $r$ is small enough.

Now $\Delta w_j = \Delta \tilde{z}_j = 0$. With respect to $\{\tilde{x}_j, \tilde{y}_j\}$ 
(even with respect to $\{x_j, y_j\}$), $\Delta$ is uniformly elliptic with bounded coefficients, so
\cite[Theorem 8.15]{gt} applies and yields
$\|w_j\|_{L^\infty(B(x_0,r/2))} \leq C r^{-m}\|w_j\|_{L^2(B(x_0,r))} \leq C r^2$. Finally, since $\Delta$ even has $C^{0,\alpha}$
coefficients in $\{\tilde{x}_j, \tilde{y}_j\}$,
we obtain $\|\nabla w_j\|_{\infty} \leq Cr$, $\|w_j\|_{C^{2,\alpha}} \leq Cr^{-\alpha}$.
Thus, making $r$ smaller if needed, the functions $\tilde{z}_j - w_j$ define 
holomorphic coordinates on $B(x_0,r)$ because their gradients are still linearly independent,
and moreover still satisfy the $C^{1,\alpha}$ condition. For $C^{k+1,\alpha}$ one
uses the local formulas for the Ricci and scalar curvatures of a K{\"a}hler metric, as in the proof of \cite[Proposition 1.2]{ty1}.
\hfill $\Box$\medskip\

We now outline the proof of Proposition \ref{solve}. Some preliminaries: If
\begin{equation}\label{ma}(\omega_0 + i\partial\bar{\partial} u)^m = e^f \omega_0^m\end{equation}
with $u \in C^4(M)$ and $f \in C^2(M)$, 
then $\omega := \omega_0 + i \partial
\bar{\partial} u$ is again positive, thus even uniformly equivalent to $\omega_0$ because
$\det_{\omega_0}\omega = e^f$ and ${\rm tr}_{\omega_0}\omega = m + \frac{1}{2}\Delta_{\omega_0}u$.
To see this, it suffices to prove $\omega > 0$ at \emph{one point} because $\det_{\omega_0}\omega$ never vanishes.
Now $u$ need not attain its infimum but we can use the following version of Yau's maximum principle,
which is a simple modification of Yau \cite[Theorem 1]{yau2}:

\begin{lemma}[Yau]\label{yaumax}
Let $(M, g)$ be complete with sectional
curvature bounded below, and let $x_0 \in M$.
If $u \in C^2_{{\rm loc}}(M)$ with $\sup |u| + \sup |\nabla u| + \sup |\nabla^2 u| < \infty$ 
does not attain its supremum on $M$,
then there exists $\{x_k\}\subset M$ such that
$\dist(x_0,x_k) \to \infty$ and
$u(x_k) \to \sup u$, $|(\nabla u)(x_k)| \to 0$, and $\limsup_k \max\{{\rm spec}((\nabla^2 u)(x_k))\} \leq 0$. \hfill $\Box$
\end{lemma}

As usual, the following identity is all-important when studying (\ref{ma}):
\begin{equation}\label{binomial}\omega^m - \omega_0^m = \frac{1}{2}dd^c u \wedge \sum_{k = 0}^{m-1} (\omega^k \wedge \omega_0^{m-1-k}). \end{equation}
This first of all shows that the linearization of (\ref{ma}) at $\omega_0$ is a Poisson equation, 
\begin{equation}\label{pois} \dot{f} \omega_0^m = \frac{m}{2}dd^c \dot{u} \wedge \omega_0^{m-1} = \frac{1}{2}(\Delta_{\omega_0}\dot{u})\omega_0^m,
\end{equation}
but more substantially, it expresses the \emph{divergence structure} of $\C$MA, allowing one to treat (\ref{ma}) much in analogy
with (\ref{pois}) even when $u$ is not infinitesimal.

In particular, $\omega^m - \omega_0^m$ is exact, so
$\int (e^f - 1) \omega_0^m = 0$ if $M$ is compact. By Yau's famous theorem \cite{yau3}, this is then the only 
constraint that a given $f$ must satisfy in order for (\ref{ma}) to admit a solution $u$, which then is unique
up to a constant.
 
If $M$ is noncompact, then (\ref{ma}) does \emph{not} necessarily 
imply $\int (e^f - 1) \omega_0^m = 0$, not even if $f \in C^2_0(M)$, because $|\nabla u|$
may not decay fast enough to permit integration by parts. Indeed, 
if $\beta > 2$, we will prove existence of bounded solutions $u$ for any $f \in C^{\infty}_0(M)$.
On the other hand, if $\beta \leq 2$, then presumably $\int (e^f - 1)\omega_0^m = 0$ is not just technically necessary
to produce a bounded $u$, but we may expect solutions for general $f \in C^\infty_0(M)$
to grow like a Green's function, cf.~Remark \ref{growsol}(i).\medskip\

As for details, if one attempts to generalize Yau's proof in the compact case,
one must eventually prove
a-priori estimates for \emph{both} (\ref{ma}) and its linearization.
Given Yau's work, the remaining difficulties for the linear and the non-linear
equation are fairly similar.
The derivative estimates are either local anyway, or can be salvaged using Lemma \ref{yaumax}. The $L^\infty$ estimate
relies on iteration, thus on Sobolev inequalities. 
For $\beta > 2$, the iteration formally works just like in the compact case,
and for $\beta \leq 2$, one exploits $\int (e^f - 1)\omega_0^m = 0$ to compensate for the \textquotedblleft subtraction\textquotedblright$\;$in (\ref{ns}).

The \emph{only issue} with this strategy is that
we need to be able to integrate by parts in order to perform Moser iteration, but $|u|$ and $|\nabla u|$ might not decay fast enough.
For the linear problem, one could resort to solving
Dirichlet problems on larger and larger domains, but this is not an option for $\C$MA.

The key idea in Tian-Yau \cite{ty1} is to introduce a small parameter $\epsilon > 0$:
\begin{equation}\label{pert-ma-1} (\omega_0 + i\partial\bar{\partial}u_\epsilon)^m = e^{f + \epsilon u_\epsilon}\omega_0^m.\end{equation}
This is then fairly simple to solve by the continuity method
because both (\ref{pert-ma-1}) and its linearization enjoy a \emph{trivial} $L^\infty$ bound from the maximum principle.
It remains to establish a \emph{uniform} $L^\infty$ estimate for (\ref{pert-ma-1}) as $\epsilon \to 0$. For intuition,
observe that the Green's function of $\Delta - \epsilon$ in $\R^\beta$ decays exponentially even if $\beta = 2$. Indeed, it turns
out that solutions to (\ref{pert-ma-1}) are highly integrable,
so Moser iteration does go through now, although the extra $\epsilon$ term is not easy to deal with when $\beta \leq 2$.

In Section 3.2, we recall how to solve (\ref{pert-ma-1}). This
is due to Cheng-Yau \cite{cy}. We mainly wish to clarify what assumptions on $(M,\omega_0)$ and $f$ are really used.

In Section 3.3, we then establish high integrability of the solutions of (\ref{pert-ma-1}), and show 
how to prove an $L^\infty$ estimate for (\ref{pert-ma-1}) that is uniform in $\epsilon$. Both parts of the argument
are based on iteration, thus on the divergence structure in (\ref{binomial}).

\subsection{Solution of the $\epsilon$-perturbed equation} Here we only require that 
$(M,\omega_0)$ is a complete K{\"a}hler manifold which admits a $C^{3,\alpha}$ quasi-atlas 
for some $\alpha \in (0,1)$, and $f \in C^{2,\alpha}(M)$ in the sense of (\ref{holder}).
For a fixed $\epsilon > 0$, we will use the continuity method to prove that (\ref{pert-ma-1}) 
has a solution $u_\epsilon \in C^{4,\bar{\alpha}}(M)$ for some $\bar{\alpha} \in (0,\alpha]$ 
which may possibly depend on $(M,\omega_0)$, $\alpha$, $f$, but not on $\epsilon$ if $\epsilon \leq 1$.

Replacing $f$ by $tf$, $t \in [0,1]$,
in (\ref{pert-ma-1}), we obtain an equation (\ref{pert-ma-1}.$t$). 
For $\gamma \in (0,\alpha]$ define
$\mathcal{T}_{\gamma}$ $:=$ $\{t \in [0,1]:$ (\ref{pert-ma-1}.$t$) has a solution $u_{\epsilon,t} \in C^{4,\gamma}(M)\}$. 
Trivially, $0 \in \mathcal{T}_{\gamma}$ for all $\gamma$, and the gist of the proof is to establish
existence of at least one $\bar{\alpha} \in (0,\alpha]$
such that $\mathcal{T}_{\bar{\alpha}}$ is both open and closed in $[0,1]$.\medskip\

\noindent {\bf Claim 1.} For all $\gamma \in (0,\alpha]$ and  $t \in \mathcal{T}_\gamma$ there
exists a bounded linear operator $G:$ $C^{2,\gamma}(M) \to C^{4,\gamma}(M)$
with $(\frac{1}{2}\Delta_\omega - \epsilon) \circ G = {\rm id}$,
where $\omega := \omega_0 + i\partial\bar{\partial}u_{\epsilon,t}$.  \medskip\

Claim 1 implies openness of $\mathcal{T}_{\gamma}$ in $[0,1]$. Notice
that $\omega_0 + i\partial\bar{\partial} u > 0$ is indeed 
an open condition on $u \in C^{4,\gamma}(M)$ because of the definition, (\ref{holder}).\medskip\ 

\noindent \emph{Proof of Claim 1.} By Schauder theory and because $\epsilon > 0$,
for every bounded and sufficiently smooth domain $\Omega \subset M$ 
there exists a unique solution $u = u_\Omega \in C^{4,\gamma}(\Omega)$
for the Dirichlet problem $(\frac{1}{2}\Delta_\omega - \epsilon)u = f|_\Omega$ in $\Omega$, $u = 0$ on $\partial \Omega$,
and there are local estimates (with respect to the charts of the chosen quasi-atlas) 
for $\|u\|_{C^{4,\gamma}}$ in terms of 
$\|f\|_{C^{2,\gamma}}$ and $\|u\|_\infty$.
By the maximum principle, $\|u\|_\infty \leq \frac{1}{\epsilon}\|f\|_\infty$.
Thus we can apply Arzel{\`a}-Ascoli to the sequence $\{u_{\Omega_k}\}$ for an
exhaustion $M = \bigcup \Omega_k$.\hfill $\Box$\medskip\

\noindent {\bf Claim 2.} There exists $\bar{\alpha} \in (0,\alpha]$ 
such that if $u = u_{\epsilon,t}\in C^4(M)$ solves (\ref{pert-ma-1}.$t$) for $t \in [0,1]$, $\epsilon \in (0,1]$, then 
$\|u\|_{C^{2,\bar{\alpha}}(M)} \leq C(M,\omega_0,\|f\|_{C^{2,\alpha}(M)})$.\medskip\

Then, writing out (\ref{pert-ma-1}.$t$) in a local quasi-chart and taking partials
with respect to coordinates, we find that $v := D_iu$ satisfies an 
elliptic equation ${\rm tr}(A \cdot D^2v) = B\ast D^2u + t D_if + \epsilon v$,
where $A$ is controlled in $C^{2,\bar{\alpha}}$ and $B$ is controlled in $C^{1,\bar{\alpha}}$,
independent of the chart, by Claim 2. 
Thus, $u$ is controlled in the global H{\"o}lder space $C^{4,\bar{\alpha}}(M)$.
Altogether this yields closedness of $\mathcal{T}_{\bar{\alpha}}$ in $[0,1]$.\medskip\

\noindent \emph{Proof of Claim 2.} First, by Lemma \ref{yaumax},
$\|u\|_\infty \leq \frac{1}{\epsilon}\|f\|_\infty$. Next, 
\begin{equation}\label{c2}0 \leq {\rm tr}_{\omega_0}\omega = m + \frac{1}{2}\Delta_{\omega_0}u 
\leq C \exp(C(u - \inf u)), \end{equation}
with $C = C(M,\omega_0, \|f\|_{C^2(M)})$ if $\epsilon \leq 1$. This is a simple variant,
using Lemma \ref{yaumax}, of Yau's classical $C^{\textup{\begin{tiny}$2$\end{tiny}}}$ estimate. We recommend B{\l}ocki \cite{blocki} for a readable exposition.
(\ref{c2}) has two important implications:
First, $\|dd^c u\|_\infty \leq C$ since $\det_{\omega_0} \omega$ $=$ $e^{f + \epsilon u}$, so in particular
(\ref{pert-ma-1}) is uniformly elliptic at $u$. Second,
$\|u\|_{C^{1,\gamma}(M)}$ $\leq$ $C(\gamma)$ for all $\gamma \in (0,1)$ by linear theory. 
Finally, exploiting $\|dd^c u\|_\infty \leq C$ and uniform ellipticity, well-known
local estimates yield the desired $C^{2,\bar{\alpha}}$ bound, see again \cite{blocki}. \hfill $\Box$

\subsection{Passage to the limit as $\epsilon \to 0$} From our previous discussion,
only the $L^\infty$ estimate needs to be made uniform in $\epsilon$, and this involves two steps:
(1) Solutions of (\ref{pert-ma-1}) are highly integrable, although the \emph{values} of their $L^p$ norms may blow
up as $\epsilon \to 0$, and (2) Moser iteration, which crucially relies on being able to integrate by parts,
which in turn is justified by (1). Both steps involve multiplying (\ref{binomial}) by a host of test functions
and integrating by parts. We impose conditions on $(M,\omega_0)$, $u$, and $f$ as they become necessary.

Here is the basic calculation. For $u = u_\epsilon \in C^2_{\rm loc}(M)$ solving (\ref{pert-ma-1}),
$\zeta \in C^\infty_0(M)$, $p > 1$, multiply both sides of (\ref{binomial}) by
$\zeta u|u|^{p-2}$ to obtain
\begin{equation}\zeta u|u|^{p-2}(e^{f + \epsilon u} - 1)\omega_0^m = \frac{1}{2}\zeta u|u|^{p-2} dd^c u \wedge T \end{equation}
with $T := \omega_0^{m-1} + \omega_0^{m-2} \wedge \omega + ... + \omega^{m-1}$. Then integrate by parts, which yields
\begin{equation}\label{ibp}\begin{split}&\int \zeta |\nabla|u|^{\frac{p}{2}}|^2 \omega_0^m
+ \frac{mp^2}{2(p-1)}\int \zeta u|u|^{p-2}(e^{\epsilon u} - 1) e^f \omega_0^m 
\leq \\
&{-\frac{mp^2}{2(p-1)}}\left[\int \zeta u|u|^{p-2}(e^f-1)\omega_0^m +
\frac{1}{2}\int u|u|^{p-2}d\zeta \wedge d^cu \wedge T\right].\end{split}\end{equation}
Notice the structure of the good term with $\epsilon$ 
on the left: 
For all $C < \infty$ there exists $\delta = \delta(C) > 0$ such that if $|x| \leq C$, then $x(e^x - 1) \geq \delta x^2$,
so $u(e^{\epsilon u} - 1) \geq \delta \epsilon u^2$ 
for a constant $\delta > 0$ that only depends on an upper bound for $\|f\|_\infty$.\medskip\

(1) \emph{High integrability.} By \cite[Theorem I.4.2]{sy}, if Ric$(\omega_0)$ is bounded below, then
there exists a smooth function $\rho: M \to \R$
such that $\rho \sim 1 + {\rm dist}(x_0,-)$ for every fixed $x_0 \in M$.
Fix a smooth cut-off function $\chi: \R^+ \to \R^+$ with $\chi(t) = 1$ for $t < 1$, $\chi(t) = 0$ for $t > 2$,
apply (\ref{ibp}) with $\zeta = (\chi \circ \frac{\rho}{R})\rho^k$, $k \in \R$, $R > 0$, and let $R \to \infty$.
Assuming $\|u\|_{C^2(M)} < \infty$ (so in particular $\omega$ and $\omega_0$ are uniformly equivalent)
and $\|f\|_{\infty} < \infty$, this readily implies
\begin{equation}\label{ibp2}\int \rho^k|u|^{p-2}|\nabla u|^2 + 
\epsilon \int \rho^k |u|^p \leq C(\int \rho^k |u|^{p-1}|e^f-1| + \int \rho^{k-2}|u|^p),\end{equation}
where all integrals are against $\omega_0^m$, and $C$ may depend in an unspecified 
way on $M$, $\omega_0$, $f$, $u$, $\rho$, $p$, and $k$.
Now assume $(M,\omega_0)$ has polynomial volume growth and\medskip\

\noindent {\bf Temporary Assumption:} $f \in C^{\infty}_0(M)$.\medskip\

\noindent Then fix $k_0 \in \Z$ sufficiently negative and iterate (\ref{ibp2}) from $k = k_0$ to obtain
\begin{equation}\label{qualit}\int \rho^k |u|^{p-2}|\nabla u|^2 + \int \rho^k |u|^p < \infty \;\;(k \in \N_0, p > 1). \end{equation}
As mentioned above, this qualitative high integrability (at least for $k = 0$) of $u = u_\epsilon$, $\epsilon > 0$, 
is what allows us to perform Moser iteration in Step (2) and thus derive an $L^\infty$ bound for $u$ which is uniform
in $\epsilon$. We also need (\ref{qualit}), but for \emph{all} $k \in \N_0$, 
in Section 3.4 to prove pointwise decay for $u$ if $M$ satisfies SOB$(\beta)$ with $\beta > 2$.

It may help to stress at this point that the Temporary Assumption
is stronger than needed for (\ref{qualit}), and the full strength of (\ref{qualit}) is not
needed for everything that follows. However, we will eventually find that the size of the support of $f$ does not
influence the final quantitative bounds
for $u$, which will then allow us to relax $f \in C^\infty_0(M)$ to the condition stated in Proposition \ref{solve}.\medskip\

(2) \emph{Uniform $L^\infty$ estimate.} Applying (\ref{ibp}) with 
$\zeta = \chi \circ \frac{\rho}{R}$, $R \to \infty$, any $p > 1$,
and using (\ref{qualit}) with $k = 0$ to get rid of the boundary term, we obtain   
\begin{equation}\label{moser}\int |\nabla|u|^{\frac{p}{2}}|^2
+ \frac{mp^2}{2(p-1)}\int u|u|^{p-2}(e^{\epsilon u} - 1) \leq -\frac{mp^2}{2(p-1)} \int u|u|^{p-2}(e^f-1).\end{equation}
We now assume all the hypotheses of Proposition \ref{solve}. If $\beta > 2$, then, by (\ref{ds}), 
\begin{equation}\label{ds-2}\left(\int |U|^{2\alpha} \rho^{\alpha(\beta - 2) - \beta} \right)^{\frac{1}{\alpha}}
\leq C\int |\nabla U|^2 \;\;(U \in C^\infty_0(M), \alpha \in [1, \alpha_{n2}]). \end{equation}
For $\mu > 2$ as in the hypothesis, let $C_0 := \sup \rho^\mu |f| < \infty$. 
Then from (\ref{moser}), (\ref{ds-2}), and H{\"o}lder's inequality,
we can deduce the following two facts: First, 
\begin{equation}\label{start}p > \frac{\beta-2}{\mu - 2} \Longrightarrow \left(\int |u|^{\alpha p} \rho^{\alpha(\beta-2)-\beta}\right)^{\frac{1}{\alpha}} \leq \frac{C_0 C p^2}{p-1} \left(\int |u|^{\alpha p} \rho^{\alpha(\beta - 2) - \beta} \right)^{\frac{p-1}{\alpha p}},\end{equation}
so the corresponding weighted $L^{\alpha p}$ norm of $u$ is \emph{finite} for all such $p$
and we may \emph{start} the Moser iteration process at $\alpha p$.
Second, if $q := \mu + \alpha(\beta-2)-\beta > 0$, which can be arranged for by a suitable choice of $\alpha > 1$,
then
$$p > \frac{\alpha}{q}(\beta-2)\Longrightarrow
\left(\int |u|^{\alpha p} \rho^{\alpha(\beta-2)-\beta}\right)^{\frac{1}{\alpha}}
\leq \frac{C_0 C p^2}{p-1}\left(\int |u|^p \rho^{\alpha(\beta-2)-\beta}\right)^{\frac{p-1}{p}},$$
so we can keep iterating from any such $p$. This proves Proposition \ref{solve} if $\beta > 2$;
in particular, we can now trivially get rid of the
Temporary Assumption above.

The parabolic case, $\beta \leq 2$, is considerably more delicate. By (\ref{ns}),
\begin{equation}\label{ns-2} \left(\int_M |U - U_\psi|^{2\alpha} \phi \right)^\frac{1}{\alpha} \leq C \int |\nabla U|^2 \;\;(U \in C^\infty_0(M), \alpha \in [1,\alpha_{n2}]),\end{equation}
where we can take $\phi = \rho^{\alpha(\beta-2-\delta)-\beta}$, $\psi \sim \rho^{-2-\delta}$
for any $\delta > 0$,
and $U_\psi$ denotes the weighted average of $U$ with respect to $\psi$. If $v := u - u_\psi$ \emph{was} still a solution of
(\ref{pert-ma-1}), then from (\ref{moser}) applied to $v$, and (\ref{ns-2}), we would get
\begin{eqnarray}\label{par1}\left(\int |v|^{2\alpha}\phi\right)^{\frac{1}{\alpha}} \leq C \int |v||e^f - 1|,\\
\label{par2}\left(\int \left||v|^{\frac{p}{2}} - (|v|^{\frac{p}{2}})_\psi\right|^{2\alpha}\phi\right)^{\frac{1}{\alpha}} \leq  \frac{Cp^2}{p-1} \int |v|^{p-1}|e^f-1|.\end{eqnarray}
Thus, if $|f| \leq C \rho^{-\mu}$ with $\mu > 1 + \frac{\beta}{2}$, then (\ref{par1}) would 
imply $\|v\|_{\phi,2\alpha} < \infty$, and using the general inequality
\begin{equation}\label{aux}\|(|U|^{\frac{p}{2}})_\psi\|_{\phi, 2\alpha} 
\leq \left(\int \psi\right)^{-1}\left(\int \frac{\psi^2}{\phi}\right)^{\frac{1}{2}} \left(\int \phi\right)^{\frac{1}{2\alpha}} \|U\|_{\phi,p}^{\frac{p}{2}}\end{equation}
(where all three factors on the right-hand side are finite provided $\alpha < 2$) we could then iterate (\ref{par2}),
assuming in addition that $\mu > 2$.

The integrability condition $\int (e^f - 1)\omega_0^m = 0$ will be used precisely to circumvent
the difficulty that $v := u - u_\psi$ does not in general solve (\ref{pert-ma-1}) anymore.
For this we now entirely follow Tian-Yau. The first thing to observe is that from (\ref{moser}) applied to $u$, (\ref{ns-2}), 
and $\int (e^f - 1)\omega_0^m = 0$, the inequality (\ref{par1}) \emph{still holds} for $v$ even though
$v$ itself may not solve (\ref{pert-ma-1}). However, (\ref{moser}) is not enough to iterate. 
For this one has to go back to (\ref{binomial}) and multiply by other test functions than $\zeta u|u|^{p-2}$,
e.g.~by $\zeta \max\{0, u - u_\psi\}^{p-1}$ if $u_\psi \geq 0$, cf. \cite[Lemma 3.5]{ty1}. 
As a result, $-C \leq u \leq u_\psi + C$ if $u_\psi \geq 0$, and $u_\psi - C \leq u \leq C$ if $u_\psi \leq 0$,
assuming $|f| \leq C \rho^{-\mu}$, $\mu > 2$.

Integrating by parts in (\ref{pert-ma-1}) gives $\int (e^{f + \epsilon u} - 1)\omega_0^m = 0$, so
$\int e^f(e^{\epsilon u} - 1)\omega_0^m = 0$ from integrability, thus
$\inf u < 0 < \sup u$. If $\sup u$ was attained at $x_{\rm max}$ $\notin$ ${\rm supp}(f)$ or $u(x_k) \to \sup u$,
$x_k \to \infty$, 
then $(\omega_0 + i \partial\bar{\partial} u)^m = e^{\epsilon u}\omega_0^m$ at $x_{\max}$ or at $x_k$,
$k \gg 0$, a contradiction. We conclude that $\inf u, \sup u$ are attained at
$x_{\rm min}, x_{\rm max} \in {\rm supp}(f)$. 
If $u_\psi \geq 0$, then $-C \leq u \leq u_\psi + C$ from before,
and $u(x_{\rm min}) < 0$. Let $B$ denote the connected component of $x_{\rm min}$ in $\{u < 1 + \min u\}$. Then
$\Delta u \leq C$ on $B$ from (\ref{c2}), which yields $\int_B \psi \geq \delta \int_M \psi$
for some definite $\delta > 0$, which may however depend on $\rho(x_{\rm min})$,
cf.~\cite[Lemma 3.7]{ty1}. Thus, 
$\max u \leq u_\psi$ $+$ $C$ $\leq$ $(1-\delta)(\max u) + C$, and so $\|u\|_\infty$ is bounded in terms of 
$\sup \rho^\mu |f|$, for some $\mu > 2$, and the size of ${\rm supp}(f)$.
A similar argument works if $u_\psi \leq 0$, with
$\Delta u \geq -n$ applied near $x_{\max}$. 

Taking the limit as $\epsilon \to 0$, we obtain $u \in C^{4,\bar{\alpha}}(M)$ with
$(\omega_0 + i\partial\bar{\partial} u)^m = e^f \omega_0^m$. However, 
unlike before, our $C^{4,\bar{\alpha}}$ bound for $u$ now depends on the size
of ${\rm supp}(f)$. To overcome this issue, we iterate again, but
with $v := u - u_{\psi}$, which \emph{does} 
satisfy the same equation as $u$. The only issue now is the boundary term in (\ref{moser}),
but we have $\int |\nabla u|^2 < \infty$ as a by-product of the Moser iterations used in constructing $u$,
which fortuitously suffices because the volume growth $\beta \leq 2$, cf.~\cite[p.596]{ty1}.
\hfill $\Box$

\subsection{Decay of the solutions} Our final goal here
is to exhibit improved decay for bounded solutions $u$ of the
$\C$MA (\ref{ma}) under various hypotheses on the decay of 
$f$ and on the geometry of the complete K{\"a}hler manifold $(M,\omega_0)$.

Proposition \ref{decay} collects all the results that we need for proving the asymptotics in Theorem \ref{main1}. The main underlying idea only works for quadratic volume growth or less, and under the additional assumption that $\int |\nabla u|^2 < \infty$, but is fairly simple and universal. We write 
$M$ as a union of annuli $A_i = A(x_0,r_i,r_{i+1})$ and try to use integration
by parts to conclude
$\int_{A_i} |\nabla u|^2 \leq Ce^{-\delta i}$ for some small $\delta = \delta(u,f) > 0$. 
This works whenever the $A_i$ have their first Neumann eigenvalues
bounded below by $\frac{1}{C}(r_{i+1} - r_i)^{-2}$, and there are basically two different settings where
we can make sure this is the case: (1) $A_i \subset B(x_i,\rho_i)$ with $x_i \in A_i$ and
$\rho_i \sim r_{i+1} - r_i = o(r_i)$,  and (2) $r_{i+1} - r_{i} \sim r_i$, assuming appropriate lower Ricci bounds in each case. The conditions needed for (1) are detailed in Definition \ref{asympt1},
and they lead to exponentially decaying bounds, in terms of $r(x)$, for the integral of $|\nabla u|^2$ over $B(x,1)$. For (2) it is enough to assume SOB$(\beta)$ (Definition \ref{sob-beta}), and we obtain polynomial energy decay. 

In each case, if the injectivity radius does not decay too fast, we can apply
Moser iteration to get pointwise decay. Higher derivative bounds then follow from (scaled) Schauder estimates, in the setting described in Definition \ref{asympt2} and Lemma \ref{hmgcrit}.

We make two remarks on this approach before going into details. First, the decay is qualitatively optimal in each case (cf.~Remark \ref{growsol}(ii)), but the actual decay \emph{rates} obtained will be unreasonably small. Second, we need to assume parabolic volume
growth when integrating by parts,
but non-parabolic manifolds ought to be \emph{better} behaved.
We will briefly address both issues after proving Proposition \ref{decay}.

\begin{definition}\label{asympt1} A complete Riemannian manifold $M$ with one end is CYL$(\beta,\gamma)$,
$\beta > 0$, $\gamma \in [0,1)$, if there exist $x_0 \in M$ and $C \geq 1$ such that 
$\frac{\textup{\begin{tiny}$1$\end{tiny}}}{\textup{\begin{tiny}$C$\end{tiny}}}s^\beta \leq |B(x_0,s)| \leq C s^\beta$ for all $s \geq C$, and such that $A(x_0, r(x)-\frac{\textup{\begin{tiny}$1$\end{tiny}}}{\textup{\begin{tiny}$C$\end{tiny}}}r(x)^\gamma, r(x)+\frac{\textup{\begin{tiny}$1$\end{tiny}}}{\textup{\begin{tiny}$C$\end{tiny}}}r(x)^\gamma) \subset B(x,Cr(x)^\gamma)$ 
and ${\rm Ric}(x) \geq -Cr(x)^{-2\gamma}$ for all $x \in M$ with $r(x) \geq C$.
\end{definition}

\begin{definition}\label{asympt2} We say that $(M^n,g)$ is HMG$(\lambda, k,\alpha)$, $\lambda \in [0,1]$, $k \in \N_0$, 
$\alpha \in (0,1)$, if there exist $x_0 \in M$ and $C \geq 1$ such that for all $x \in M$ with $r(x) \geq C$ 
there exists a local diffeomorphism $\Phi_x$ from the unit ball $B \subset \R^n$ into $M$ such that $\Phi_x(0) = x$,
$\Phi_x(B) \supset B(x,\frac{1}{C}r(x)^\lambda)$, ${\rm inj}(\Phi_x^*g) \geq \frac{1}{C}r(x)^\lambda$, $\frac{1}{C}g_{\rm euc} \leq r(x)^{-2\lambda}\Phi_x^*g \leq C g_{\rm euc}$,
and
\begin{equation}\|r(x)^{-2\lambda}\Phi_x^*g - g_{\rm euc}\|_{C^{k,\alpha}(B,\, g_{\rm euc})} \leq C.\end{equation}
Then, given any smooth weight function $\phi: [0,\infty) \to [1,\infty)$, we define
\begin{equation}\label{wt}\|u\|_{\phi,l,\gamma} := 
\|u\|_{C^{l,\gamma}(B(x_0,2C),g)} + \sup \{\phi(r(x))\|u \circ \Phi_x\|_{C^{l,\gamma}(B,\,g_{\rm euc})} : r(x) \geq C\}\end{equation}
for $l\in \N_0$, $0 \leq \gamma < 1$, $\epsilon \geq 0$. E.g., $\|u\|_{\phi,1,0} \sim \sup_M \phi(r)(|u| + (1+r)^\lambda |\nabla u|)$.
\end{definition}

The following criterion will show that all the manifolds that appear when proving Theorem \ref{main1} are in fact HMG$(1)$.
This is similar to Lemma \ref{ty}.

\begin{lemma}\label{hmgcrit} A complete K{\"a}hler manifold with $|{\rm Rm}| + \sum_{i = 1}^k r^{i\lambda}|\nabla^i{\rm Scal}| \leq Cr^{-2\lambda}$
for some $k \in \N_0$ and $\lambda \in [0,1]$ is ${\rm HMG}(\lambda,k+1,\alpha)$ for every $\alpha \in (0,1)$. \hfill $\Box$
\end{lemma}

\begin{proposition}\label{decay}
Let $(M^n,\omega_0)$ be complete K{\"a}hler and let $u, f \in C^\infty(M)$ be such that
$\sup |\nabla^i u| +\sup |\nabla^i f| < \infty$ for all $i \in \N_0$, and 
$(\omega_0 + i\partial\bar{\partial}u)^m = e^f \omega_0^m$. 
   
{\rm (ia)} Assume {\rm CYL}$(\beta,\gamma)$, $\beta \in (0,2]$, $\gamma \in [0,1)$, and
$\exp(\kappa r(x)^{1-\gamma})|B(x,1)|\to\infty$ as $r(x) \to \infty$ for every fixed $\kappa > 0$. If $\int |\nabla u|^2 < \infty$ and 
$|f| \leq C r^{1-\beta - \gamma}\exp(-\epsilon r^{1-\gamma})$ for some $\epsilon > 0$, 
then $\sup_{B(x,1)} |u - u_{B(x,1)}| \leq C\exp(-\delta r(x)^{1-\gamma})$ for some $\delta > 0$.

{\rm (ib)} Assume ${\rm SOB}(\beta)$, $\beta \in (0,2]$, and for every fixed $\kappa > 0$, $r(x)^\kappa |B(x,1)| \to \infty$ as $r(x) \to \infty$.
If $\int |\nabla u|^2 < \infty$ and $|f| \leq C r^{-\epsilon}$ for some $\epsilon > 0$, then there exists a
$\delta > 0$ such that for all $x \in M$, $\sup_{B(x,1)} |u-u_{B(x,1)}| \leq C r(x)^{-\delta}$.

{\rm (ii)} Assume {\rm HMG}$(\lambda, k+1,\alpha)$, and that $M$ has at most linear diameter growth in the sense that for all $s \geq C$ and $x_1,x_2 \in M$ with $r(x_1) = r(x_2) = s$, there exists 
a path $\gamma$ from $x_1$ to $x_2$ with $s-\frac{1}{C}s^\lambda \leq r \circ \gamma \leq Cs$ and ${\rm length}(\gamma) \leq Cs$.
Let $\phi$ be a 
weight with $\phi(t)(1+t)^{-\textup{\begin{tiny}$\delta$\end{tiny}}}$ non-decreasing for some $\delta > 0$
and $\phi(t -\frac{\textup{\begin{tiny}$1$\end{tiny}}}{\textup{\begin{tiny}$2$\end{tiny}}}t^{\textup{\begin{tiny}$\lambda$\end{tiny}}})$
$\geq$ $\frac{\textup{\begin{tiny}$1$\end{tiny}}}{\textup{\begin{tiny}$C$\end{tiny}}}\phi(t)$.
Then, if $\sup_x \phi(r(x))r(x)^{\textup{\begin{tiny}$1$\end{tiny}}-\textup{\begin{tiny}$\lambda$\end{tiny}}}\sup_{B(x,1)} |u - u_{B(x,1)}| < \infty$ and $\|f\|_{\psi, k, \alpha} < \infty$, where $\psi(t)$ $:=$ $(1+t)^{\textup{\begin{tiny}$1$\end{tiny}}+\textup{\begin{tiny}$\lambda$\end{tiny}}}\phi(t)$, then
 $\|u-\bar{u}\|_{\phi,k+2,\alpha} < \infty$ for some constant $\bar{u} \in \R$.
\end{proposition}

\begin{proof} (i) For any sequence $C \leq r_1 < r_2 < ...$, define
annuli $A_i := \{r_i < r < r_{i+1}\}$ and ends $E_i := \{r > r_i\}$.
Construct cut-off functions $\zeta_i$ on
$M$ with $\zeta_i \equiv 1$ on $E_{i+1}$, ${\rm supp}(\zeta_i) \subset E_i$, and $(r_{i+1}-r_i)|\nabla \zeta_i| \leq C$, uniformly in $i$.
Take a cut-off function $\chi$ 
with $\chi(t) = 1$ for $t < 1$, $\chi(t) = 0$ for $t > 2$,
multiply (\ref{binomial}) by $(\chi \circ \frac{r}{R}) \zeta_i u$, 
integrate by parts, and let $R \to \infty$. The boundary term with $\chi'$ goes away because $|u| \leq C$, $\int |\nabla u|^2< \infty$, 
and $\beta \leq 2$. Thus, for all $i$, and again because $|u| \leq C$, 
\begin{equation}\label{aux2}\begin{split}&\sum_{j = i+1}^\infty \|\nabla u\|_{L^2(A_j)}^2 \leq \int \zeta_i |\nabla u|^2 
\leq C \int |\nabla \zeta_i| |u| |\nabla u| + C \int \zeta_i |u| |f|\\
&\leq C(r_{i+1}-r_i)^{-1}\|u\|_{L^2(A_i)} \|\nabla u\|_{L^2(A_i)} + C\sum_{j = i}^\infty (r_{j+1}^\beta - r_{j}^\beta)\|f\|_{L^\infty(A_j)}.\end{split}\end{equation}
The key observation is that 
(\ref{aux2}) still holds if $u$ is replaced by $u - \mu$ for any $\mu \in \R$ with $|\mu| \leq C$.
If we can choose $\mu$ such that $\|u-\mu\|_{L^2(A_i)} \leq C(r_{i+1}-r_i)\|\nabla u\|_{L^2(A_i)}$
and if $f \in C^\infty_0(M)$, then (\ref{aux2}) implies 
$\|\nabla u\|_{L^2(A_i)} \leq C\exp(-\delta i)$ for some $\delta > 0$ by 
induction.
In the settings described in (ia) and (ib), such 
a choice for $\mu$ is essentially always possible, given an appropriately chosen sequence of radii.

(ia) Pick $x_i \in A_i$ such that $\rho_i := r(x_i) = \frac{1}{2}(r_i + r_{i+1})$. By assumption then, 
$$A_i' := A(x_0, \rho_i - \frac{1}{C}\rho_i^\gamma, \rho_i + \frac{1}{C}\rho_i^\gamma) \subset B_i := B(x_i, \rho_i^\gamma).$$
Suppose the $r_i$ are such that $A_i \subset A_i'$ and $2B_i \subset A_{i}'' := A_{i-K} \cup A_{i-K+1} \cup ... \cup A_{i+K}$ for some $K \in \N$, $K \leq C$, and all $i \geq C$.
Then, if $\mu := u_{B_i}$, Lemma \ref{segment} yields
\begin{equation}\label{aux3}\int_{A_i} |u - \mu|^2 \leq \int_{B_i} |u-\mu|^2 \leq C\rho_i^{2\gamma} \int_{2B_i} |\nabla u|^2 \leq C\rho_i^{2\gamma}\int_{A_{i}''} |\nabla u|^2.\end{equation}
The essentially unique choice of radii which makes this work is 
$r_i := \frac{1}{C}i^{\frac{1}{1-\gamma}}$.
Define $Q_i := \int_{E_i} |\nabla u|^2$. Then (\ref{aux2}) and (\ref{aux3}) imply that for all $i \geq C$,
$$Q_{i+K} \leq Q_{i+1} \leq C(Q_{i-K} - Q_{i+K}) + Ce^{-\epsilon i},$$
since $\rho_i^\gamma \leq C(r_{i+1}-r_i)$ and by the decay assumption for $f$.
Thus, $Q_{2Ki} \leq Ce^{-\delta i}$ for some $\delta > 0$ by induction, and in particular,
$$\int_{B(x,1)} |\nabla u|^2 \leq C\exp(-\delta r(x)^{1-\gamma})$$
for all $x \in M$. Using the Neumann-type Sobolev inequality from 
Maheux and Saloff-Coste \cite[Th{\'e}or{\`e}me 1.1]{msc}, and the assumed lower
bound on $|B(x,1)|$ as $r(x)\to\infty$, we can now apply Moser iteration
on $B(x,1)$ to conclude the proof.

(ib) The argument here is entirely analogous, replacing (\ref{aux3}) by the Neumann-Poincar{\'e} inequality for annuli from  (\ref{ann-neu-poinc}): If $r_i = 1000\eta^i$,
$\eta = 1.001$, $\mu = u_{\eta A_i}$, then $\|u-\mu\|_{L^2(A_i)} \leq C\eta^{i}\|\nabla u\|_{L^2(\eta^2A_i)}$, where we recall that $\eta A(r,s) := A(\eta^{-1}r, \eta s)$. 

(ii) As in Kovalev \cite{kovalev}, we define a linear differential operator $L_u$ by setting
$$L_u(v)\omega_0^m := i\partial\bar{\partial}v \wedge \sum_{j = 0}^{m-1} \omega_0^j \wedge (\omega_0 + i\partial\bar{\partial}u)^{m-1-j}.$$
For any $x \in M$, put $u_x := u - u_{B(x,1)}$. Then $\C$MA implies $L_u(u_x) = e^f-1$. For 
a given $\mu \in [0,1]$, consider the following statement $S(\mu)$:
\begin{align*}r(x) \geq C
\Longrightarrow \|u_x\|_{C^{k+2,\alpha}(B(x,\frac{1}{C}r(x)^{\mu}),\,r(x)^{-2\mu}\omega_0)} \leq C \phi(r(x))^{-1}r(x)^{\lambda-1}.\end{align*}
\emph{Claim:} $S(0)$ is true, 
and $S(\mu)$ implies $S(\nu)$ for all $\nu \in [\mu, \min\{\lambda,\mu + \frac{\delta}{k+2+\alpha}\}]$.\medskip\

Thus, $S(\lambda)$ is true, so $|\nabla u| \leq C\phi(r)^{-1}(1+r)^{-1}$ and thus
$|u - \bar{u}| \leq C\phi(r)^{-1}$ for some constant $\bar{u}$ since $\phi(t)(1 + t)^{-\delta}$ is non-decreasing and ${\rm diam}(\partial B(x_0,s)) \leq Cs$. Using 
$S(\lambda)$ again, this proves (ii), and in fact
slightly more if $\lambda < 1$.\medskip\

\noindent \emph{Proof of the claim.} To show $S(0)$, pull the equation $L_u(u_x) = e^f-1$ back under $\Phi_x$ and
apply Schauder theory on $\Phi_x^{-1}(B(x,\frac{1}{C}))$, using the $\Phi_x^*\omega_0$-H{\"o}lder norms. Note that the coefficients of $L_u$ 
are already bounded in this sense.

Now assume $S(\mu)$ and let $\nu$ be as given. We pull back under $\Phi_x$, aiming to
apply Schauder theory again, but on the larger ball
$B_x(\mu) := \Phi_x^{-1}(B(x,\frac{\textup{\begin{tiny}$1$\end{tiny}}}{\textup{\begin{tiny}$C$\end{tiny}}}r(x)^\mu))$,
and in the stronger norms associated with the metric 
$\omega_x(\nu) := r(x)^{\textup{\begin{tiny}$-2\nu$\end{tiny}}}\Phi_x^*\omega_0$. 
From $S(\mu)$, $\phi(t) \geq (1+t)^{\delta}$, and $(k + 2 + \alpha)\nu \leq (k + 2 + \alpha)\mu + \delta$ we conclude that the
coefficients of $L_u$ are bounded in $C^{k,\alpha}(B_x(\mu),\omega_x(\nu))$. Also,
$r(x)^{2\nu}\|f \circ \Phi_x\|_{C^{k,\alpha}(B_x(\mu),\,\omega_x(\nu))} \leq 
C \phi(r(x))^{\textup{\begin{tiny}$-1$\end{tiny}}}r(x)^{\textup{\begin{tiny}$\lambda$\end{tiny}$-$\begin{tiny}$1$\end{tiny}}}$ since $\|f\|_{\psi,k,\alpha} < \infty$ and $\nu \leq \lambda$. 
Thus, Schauder theory yields
$\|u_x \circ \Phi_x\|_{C^{k+2,\alpha}(B_x(\mu),\,\omega_x(\nu))} \leq C \phi(r(x))^{-1}r(x)^{\lambda-1}$.

It remains to prove that
$\sup\{|u_x(y)| : y \in B(x,\frac{\textup{\begin{tiny}$1$\end{tiny}}}{\textup{\begin{tiny}$C$\end{tiny}}}r(x)^\nu)\}\leq C \phi(r(x))^{-1}r(x)^{\lambda-1}$. 
By assumption, this bound holds for all
$y \in B(x,\frac{\textup{\begin{tiny}$1$\end{tiny}}}{\textup{\begin{tiny}$C$\end{tiny}}}r(x)^\mu)$. Also, since $\nabla u_x = \nabla u_y$,
the derivative estimates we just showed imply
$\sup\{ |\nabla u_x(y)|: y \in B(x,\frac{\textup{\begin{tiny}$1$\end{tiny}}}{\textup{\begin{tiny}$C$\end{tiny}}}r(x)^\nu)\} \leq C \phi(r(x))^{-1}r(x)^{\lambda-1-\nu}$
because $\phi(t-\frac{\textup{\begin{tiny}$1$\end{tiny}}}{\textup{\begin{tiny}$2$\end{tiny}}}t^\lambda) \geq \frac{\textup{\begin{tiny}$1$\end{tiny}}}{\textup{\begin{tiny}$C$\end{tiny}}}\phi(t)$.
Thus, the desired sup bound follows by integrating along radial $\omega_0$-geodesics in $M$.
\end{proof}

As we mentioned earlier, there are two obvious issues with Proposition \ref{decay}: (1) The decay rates are not quantitatively optimal. (2) 
We need to assume low volume growth, but non-parabolic manifolds ought to be rather \emph{better} behaved in terms of potential theory. 
We conclude here by briefly discussing these questions.\medskip\

(1) \emph{Optimal decay rates in Proposition \ref{decay} {\rm (see also Remark \ref{growsol}(ii))}.}

On an asymptotically cylindrical manifold, if $-\lambda^2$ $(\lambda > 0)$ is the first eigenvalue of the cross-section, then
separation of variables suggests
$|u| + |\nabla u| + ... + |\nabla^{k} u| \leq C e^{-\delta r}$ if $|f| + |\nabla f| + ... + |\nabla^k f| \leq C e^{-\epsilon r}$,
with $\delta = \epsilon$ if $\epsilon \in (0,\lambda)$, but no better than that (even if
$\epsilon \geq \lambda$)
unless we impose extra integrability conditions on $f$. Similarly, on a space asymptotic 
to $C \times T^{n-2}$  (or even a flat $T^{n-2}$-submersion over $C$)
for a flat $2$-dimensional cone $C$ of cone angle $\theta \in (0,1]$, we expect $|u| + r|\nabla u| + ... + r^{k} |\nabla^{k} u| \leq
C r^{-\delta}$ if $|f| + r|\nabla f| + ... + r^k |\nabla^k f| \leq C r^{-2-\epsilon}$, with
$\delta = \epsilon$ if $\epsilon \in (0, \frac{1}{\theta})$. In either case, Proposition \ref{decay} yields the right qualitative behavior, but 
not the best values for $\delta$. 
However, there is a relatively simple way of boosting the decay. The idea is that
if $u$ did \emph{not} decay fast enough compared to $f$, then after translating in the cylindrical,
or rescaling and unwrapping the fibers in the homogeneous setting, thanks to the 
bounds established in Proposition
\ref{decay}, $u$ would converge to a \emph{harmonic} function
on $\R^+$ $\times$ cross-section or $C^2 \times \R^{n-2}$ (constant in the $\R^{n-2}$ directions) with \emph{some} 
and hence \emph{optimal}
decay, which contradicts the assumption. This outline can be
made rigorous 
using \textquotedblleft 3-circles\textquotedblright$\;$type arguments, as in Cheeger-Tian \cite{ct-cones}. \hfill $\Box$\medskip\

(2) \emph{Non-parabolic manifolds $(\beta > 2)$.}

Suppose $(M,\omega_0)$ satisfies SOB$(\beta)$ with $\beta > 2$. Proposition \ref{solve} yields a bounded
solution $u$ if $|f| \leq C r^{-\mu}$ for some $\mu > 2$, however one
would not expect $u$ to have finite energy outside of a certain range for $\mu$. 
On the other hand, from the behavior of the Laplacian on $\R^\beta$ and since the Green's function
on $(M,\omega_0)$
decays like $r^{2-\beta}$, one might hope for $|u|$ $\leq$ $C r^{2-\mu}$
if $\mu \in (2,\beta)$. 
We now describe two settings where this can be proved, though with a small loss in the exponent.

$\bullet$ \emph{Barrier method} (Santoro \cite{santoro}). This is based on the observation that on $\R^\beta$, 
$$\Delta(r^{2-\mu}(\log r)^\epsilon) = r^{-\mu}(\log r)^{\epsilon}\left[(2-\mu)(\beta-\mu)+ \frac{\epsilon(2 + \beta -2\mu)}{ \log r} + \frac{\epsilon(\epsilon-1)}{(\log r)^2}\right],$$
which becomes negative for $r \gg 1$ if either $\mu \in (2,\beta)$, $\epsilon$ arbitrary,
or $\mu = \beta$, $\epsilon > 0$.
Thus, if $\Delta u = f$ with $|u| + |\nabla u| + |\nabla^2u| \leq C$ and $|f| \leq C r^{-\mu}(\log r)^\epsilon$, 
$\mu \in (2,\beta)$,
or $|f| \leq C r^{-\mu}(\log r)^{\epsilon - 1}$, $\mu = \beta$, $\epsilon > 0$, 
then $|u|$ $\leq$ $Cr^{2-\mu}(\log r)^\epsilon$ by Lemma \ref{yaumax}. 
This idea obviously
generalizes
to various sorts of asymptotically flat submersions over $\R^\beta$ with bounded fibers, and also works for 
$\C$MA on such spaces.

$\bullet$ \emph{Moser iteration with weights.} Let $(M,\omega_0)$ and $f$ be as in Proposition \ref{solve} and assume
there exists a smooth $\rho \sim 1 + r$ 
with $|d \rho| + \rho |dd^c \rho| \leq C$. By Section 3.2,
for all $\epsilon \in (0,1)$ there exists $u = u_\epsilon$, $\|u\|_{C^{4,\bar{\alpha}}(M)} \leq C$,
with $(\omega_0 + i\partial\bar{\partial}u)^m = e^{f + \epsilon u}\omega_0^m$. 
Moreover, assuming for the moment that $f \in C^\infty_0(M)$, by (\ref{qualit}),
\begin{equation}\label{int}\int \rho^k|u|^{p-2} |\nabla u|^2 + \int\rho^k |u|^p < \infty \;\;(k \in \N_0, p > 1).\end{equation} 
For a cut-off function $\chi$ with $\chi(t) = 1$ for $t < 1$, $\chi(t) = 0$ for $t > 2$,
multiply (\ref{binomial}) by $(\chi \circ \frac{\rho}{R})\zeta u|\zeta u|^{p-2}\zeta$, where 
$R > 0$, $\zeta := \rho^{l}$ for any $l \in \R$, and $p > 1$. Commute the spare factor of $\zeta$ past the $dd^c$,
integrate by parts, and let $R \to \infty$. The boundary term disappears by (\ref{int}).
Eventually, using
$|d \rho| + \rho|dd^c \rho| \leq C$, $\omega_0 + i\partial\bar{\partial} u \sim \omega_0$, and the sharp weighted
Sobolev inequality (\ref{ds}), 
$$\left(\int \rho^{\alpha(\beta-2)-\beta} |\zeta u|^{\alpha p}\right)^{\frac{1}{\alpha}} 
\leq \frac{C p^2}{p-1}\left(\int |\zeta u|^{p-1} \zeta |f| + 
\frac{|l|(|l|+1)p}{p-1}\int \rho^{-2} |\zeta u|^p\right),$$
for any $\alpha \in [1,\alpha_{n2}]$, with $C$ depending only on $\|f\|_{C^{2}(M)}$.
This yields
$$\left(\int \zeta_{k+1} |u|^{p_{k+1}}\right)^{\frac{1}{p_{k+1}}} \leq \left(\frac{C_0Cp_k^3}{(p_k-1)^2}\right)^{\frac{1}{p_k}}
\max\left\{1, \left(\int \zeta_k |u|^{p_k}\right)^{\frac{1}{p_k}}\right\}$$
for all $k \in \N$, where $p_k := \alpha^k p_0$ with a fixed $p_0 > 1$, and $\zeta_k := \rho^{\alpha^k(\beta-2) - \beta}$,
assuming that $|f| \leq C_0 \rho^{-\mu}$, $\mu > 2$, $p_0(\mu - 2) > \beta -2$, and $\alpha$ is close enough to $1$
depending on $p_0,\beta,\mu$. The basic step, $k = 1$, can be estimated from 
(\ref{start}), so altogether
$$p_0 > \frac{\beta-2}{\mu - 2} \Longrightarrow |u| = |u_\epsilon| \leq C(p_0, \beta,\mu, \|f\|_{C^{2}(M)}, \sup \rho^\mu |f|) \rho^{\frac{2-\beta}{p_0}}.$$
Therefore, $|u_\epsilon| \leq C(\delta) \rho^{2-\mu + \delta}$ for all $\delta > 0$ if $\mu \in (2,\beta)$, uniformly as $\epsilon \to 0$, 
and the assumption $f \in C^\infty_0(M)$ can be relaxed to $\|f\|_{C^{2,\alpha}(M)} + \sup \rho^\mu |f| < \infty$.\hfill $\Box$

\section{The new Calabi-Yau metrics}

This section leads up to the proof of Theorems \ref{main} and \ref{main1} in Section 4.4.

Section 4.1 reviews the basic properties of rational elliptic surfaces and works out some isotrivial examples,
which we expect to occur as bubbles in the collapsing of Ricci-flat metrics on K3, cf.~Problem \ref{bubble}.
We also summarize Kodaira's results from \cite{kodaira} on the structure of singular fibers in elliptic fibrations in general.

Section 4.2 describes a general construction of CY metrics on the total spaces of
certain polarized
families of complex $m$-tori. Given $\epsilon > 0$, a section of the family, and
a holomorphic volume form $\Omega$
on the total space, we obtain a semi-flat (i.e.~flat when restricted to fibers) 
CY metric $\omega_{{\rm sf},\epsilon}$, with top power proportional to $\Omega \wedge \bar{\Omega}$ 
and such that the tori have volume $\epsilon$ with respect to $\omega_{{\rm sf},\epsilon}$, which is explicit in terms of $\epsilon$, $\Omega$, and the periods of the tori.
Changing the reference section translates $\omega_{{\rm sf},\epsilon}$.
Gross-Wilson \cite{gw} have used such metrics on the complement of the singular fibers of
generic Jacobian elliptic K3 surfaces, with $\Omega = \Omega_{\rm K3}$.

Section 4.3 looks at elliptic fibrations $f: U \to \Delta$ over a disk,
with a section and such that the central fiber $D = f^{-1}(0)$ does not contain $(-1)$-curves. The fibration $f$ is
then isomorphic to an explicit canonical form that can be extracted from
\cite{kodaira}. Given a meromorphic $2$-form $\Omega$ on $U$ such that ${\rm div}(\Omega)$ is an integer
multiple of $D$,
this then allows us to write down an explicit formula for the associated CY metrics $\omega_{{\rm sf},\epsilon}$ 
on $U \setminus D$ and read off their geometric properties.
The most relevant cases are ${\rm div}(\Omega) = 0$ 
as in \cite{gw}, and ${\rm div}(\Omega) = -D$ as on a rational elliptic surface.

For a rational elliptic surface $f: X \to \mathbb{P}^1$, a fiber $D = f^{-1}(p)$, a meromorphic 
$2$-form $\Omega$ with ${\rm div}(\Omega) = -D$, and any choice of local section near $p$ we thus
obtain CY metrics $\omega_{{\rm sf},\epsilon}$ 
on $U \setminus D$, $U = f^{-1}(\Delta)$, $p \in \Delta \subset \mathbb{P}^1$,
with $\omega_{\textup{\begin{tiny}${\rm sf},\epsilon$\end{tiny}}}^{\textup{\begin{tiny}$2$\end{tiny}}} = \Omega \wedge \bar{\Omega}$ and
such that the fibers of $f$ have area $\epsilon$ with respect to $\omega_{{\rm sf},\epsilon}$ and
$D$ is at infinite distance.
What remains to be done in Section 4.4 is glue $\omega_{{\rm sf},\epsilon}$ with as many different K{\"a}hler 
metrics on $X \setminus U$ as possible to produce
complete K{\"a}hler metrics $\omega_0$ on $M = X \setminus D$ such that the existence theory from Section 3 
applies with $e^f =
(\Omega \wedge \bar{\Omega})/\omega_0^2 $.
The freedom of choosing a local section for $f$
is exactly what makes the K{\"a}hler gluing possible,
thanks to a remarkable $\partial\bar{\partial}$-type 
lemma for $U \setminus D$ from \cite{gw}.

\subsection{Rational elliptic surfaces and singular fibers} Besides Kodaira's original work \cite{kodaira},
useful references for what follows are \cite{bpv}, \cite{4mfds}, \cite{multfib}, and \cite{mirmod}.

Consider a rational elliptic surface $X$ as in Definition \ref{ratell}, i.e.~the blow-up of $\mathbb{P}^2$
in the nine base points of a pencil $sF + tG = 0$, where $(s:t) \in \mathbb{P}^1$ and $F$ and $G$ are smooth cubics. 
Then $X = \mathbb{P}^2 \,\#\, 9 \bar{\mathbb{P}}^2$ as smooth manifolds, and the Hodge numbers of $X$ are
$h^{1,0} = h^{2,0} = 0$ and $h^{1,1} = 10$. The natural elliptic fibration $f: X \to \mathbb{P}^1$ 
has an intrinsic description as the complete linear system $|{-K_X}|$, meaning that for each fiber
$D = f^{-1}(p)$ there exists a meromorphic $2$-form $\Omega$ on $X$, unique up to a scale, with
${\rm div}(\Omega) = -D$, and all meromorphic $2$-forms on $X$ are obtained in this way. Moreover, $f$ is the only
elliptic fibration on $X$. When blowing up a base point until all tangencies are resolved, 
 the last exceptional divisor, which is a $(-1)$-curve,
becomes a global section of $f$.
All fibers of $f$ are reduced, i.e.~the multiplicities
of $f$ along the irreducible 
components of any
fiber are coprime. A fiber $f^{-1}(p)$ in a
general
elliptic fibration is reduced if and only if $f$ admits a local 
section near $p$.        

Rational elliptic surfaces are characterized as those smooth projective surfaces
which admit a meromorphic $2$-form $\Omega$ and an elliptic fibration
whose fibers do not contain $(-1)$-curves such that ${\rm div}(\Omega)$ is an integer multiple of a fiber. 
Even though our construction of complete CY metrics only requires such a fibration near infinity,
this essentially shows that we do not have any other examples to work with.

As a last general point, notice that the space of cubics in $\mathbb{P}^2$ is a $\mathbb{P}^9$, so a
\emph{general} set of $9$ points in $\mathbb{P}^2$ determines a \emph{unique} cubic passing through it,
and the
space of $9$-point subsets of $\P^2$ modulo PGl$(3)$ has dimension $9 \cdot 2 - 8 = 10$.
On the other hand, if
cubics $F$ and $G$ intersect in finitely many points,
and $H$ is another cubic passing through these,
then, by the $AF + BG$ theorem, $H$ must already be a linear 
combination of 
$F$ and $G$. Thus, modulo PGl$(3)$ again, $9$-point sets with \emph{more} than one cubic passing through them 
depend on $2 \cdot (9 - 1) - 8 = 8$ free parameters only. If we blow up $\mathbb{P}^2$ in a $9$-point set which does \emph{not} determine
a pencil of cubics, the resulting $X$ still has meromorphic $2$-forms $\Omega$,
and sometimes an elliptic fibration,
but then one fiber $D$ must be of multiplicity $m > 1$ and 
${\rm div}(\Omega)$ is either $-\frac{1}{m}D$ 
or disconnected,
so our construction of complete CY metrics breaks down.

\begin{example}\label{isoex} Let $E = \C/\Lambda$ be a complex torus and let $\Gamma$ be a finite subgroup of ${\rm Aut}(E)$, 
so $\Gamma$ is generated by a root of unity and either $\Gamma = \Z_2$ (for any $\Lambda$), $\Gamma = \Z_4$ (for $\Lambda = \Z[i]$), 
or $\Gamma = \Z_3, \Z_6$ (for $\Lambda = \Z[\zeta_3]$),
$\zeta_m := \exp(2\pi i/m)$. 
Then $\Gamma$ acts
on $\P^1$ in a natural way, and the minimal resolution $X$ of the singular surface $(\P^1 \times E)/\Gamma$ 
comes with an obvious elliptic fibration $f: X \to \P^1$ with two singular fibers and all smooth fibers
isomorphic to $E$. Blowing down all $(-1)$-curves in the singular fibers makes $X$ rational elliptic. Below,
we work this out explicitly for $\Gamma = \Z_2, \Z_3$.

Elliptic surfaces of this sort are called \emph{isotrivial}, and are of special interest in our story 
for several reasons. First, they explain the mechanism of creation 
of singular fibers of finite monodromy in arbitrary elliptic fibrations,
and the occurrence of such fiber types in dual pairs T, T$^*$ with the same monodromy (${\rm T} = {\rm T}^*$ if
$\Gamma = \Z_2 \subset \R$). Second, only the T$^*$-types are associated with \emph{crepant} resolutions, so 
only in these cases (i.e.~with T$^*$ at the core and T at infinity)
does the complete semi-flat model metric $\omega_{{\rm sf},\epsilon}$ come from the
flat orbifold metric on $(\C \times E)/\Gamma$. Third, we expect the associated 
ALG spaces to occur as bubbles for collapsing CY metrics on elliptic K3 surfaces near finite monodromy 
singular fibers of the appropriate type.

(i) $\Gamma = \Z_2$. Fix a smooth elliptic curve $E: y^2 = g(x)$ where $g$ is a cubic,
and fix an affine coordinate $t$ on $\P^1$ such that the natural $\Z_2$-action on $\P^1$ becomes
$t \mapsto -t$.  
Setting $(T,X,Y) = (t^2,x,ty)$ yields a rational map $\mathbb{P}^1_t \times E \dashrightarrow \mathbb{P}^1_T \times \mathbb{P}^2_{XY}$
which factors through the globally defined $\Z_2$-action $(t,x,y) \mapsto (-t,x,-y)$ and
induces a birational map from $(\mathbb{P}^1_t \times E)/\Z_2$ to the surface $S: Y^2 = T g(X)$ 
which is in fact an isomorphism of elliptic fibrations away from $t = 0, \infty$. 
Project $S$ onto $\P^2_{XY}$ and
blow up the base points of the pencil $Y^2 = T g(X)$. 
Going back then already yields a resolution for 
$(\P^1 \times E)/\Z_2$, without having to blow up in a fiber first.

To see what this looks like, notice the base points are $P_i = (x_i,0) = (x_i:0:1)$, 
$i = 1,2,3$, where $g(x_i) = 0$, and $P_0 = (0:1:0)$.
At $P_i$, the curves in the pencil have a common tangent $L_i: X = x_i$ of multiplicity $2$, while
at $P_0$, they share the line at infinity $L_0$ as a flex tangent.
We blow up $P_i$ twice with exceptional curves $E_i^1, E_i^2$, and $P_0$ three times with exceptional curves
$E_0^1, E_0^2, E_0^3$.
The resulting elliptic surface then has four sections, given by the $(-1)$-curves $E_i^2$ and $E_0^3$, 
and two singular fibers. The one at $T = 0$ consists of the \textquotedblleft spine\textquotedblright$\;Y = 0$,
plus the four \textquotedblleft ribs\textquotedblright$\;E_i^1, L_0$, which intersect the spine transversely 
in four different points. The singular fiber at $T = \infty$ looks the same,
with $E_0^1$ as spine and $L_i, E_0^2$ as ribs. 
Both are of Kodaira
type I$_0^*$. 
In the resolution picture, the spines are 
the strict transforms of the rational curves 
$E/\Z_2$ $\subset$ $(\P^1_t \times E)/\Z_2$ over $t = 0, \infty$,
while the ribs resolve
the singularities.

It is clear what background metric one
would have to take in order to construct a complete Calabi-Yau metric when removing either one of the singular fibers.
There are flat metrics on $\C_t \times E$, $\C_s\times E$ ($s = \frac{1}{t}$) with corresponding
$2$-forms $\Omega_0 = dt \wedge dw$, $\Omega_\infty = ds \wedge dw$, where 
$dw$ denotes an invariant differential on $E$. The key
point is that $\Omega_0, \Omega_\infty$ are both
\emph{invariant under the $\Z_2$-action}, hence they 
lift to the rational elliptic surface with polar divisor the singular fiber at $T = \infty,0$, respectively.
Thus one
can use the flat orbifold $(\C \times E)/\Z_2$ as a background. Settings of this 
sort have been treated in \cite{bm, santoro} and heuristically in \cite{stringy,hitchin}, cf.~Remark \ref{sect1rem}(ii).

(ii) $\Gamma = \Z_3$. Take $E: y^2 + y = x^3$ and $(x,y) \mapsto (\zeta_3 x,y)$ as a generator for $\Gamma$. Fix 
an affine coordinate $t$ on $\P^1$ such that the natural $\Z_3$-action on $\P^1$ becomes $t \mapsto \zeta_3^2 t$. 
Then the rational map $\P^1_t \times E \dashrightarrow \P^1_T \times \P^2_{XY}$ given by $(T,X,Y) = (t^3, tx, y)$ factors
through the diagonal $\Gamma$-action and thus induces a birational map from $(\P^1_t \times E)/\Gamma$ 
to the surface $T(Y^2 + Y) = X^3$ in $\P^1_T \times \P^2_{XY}$.
The pencil of cubics obtained from this by projection 
has three base points of multiplicity $3$:
$P_1 = (0,0)$ $=$ $(0:0:1)$, $P_2 $ $=$ $(0,-1)$ $= (0:-1:1)$, and $P_3 = (0:1:0)$, 
with corresponding flex tangents $L_1: Y = 0$, $L_2: Y = -1$, and $L_3$, the line at infinity. 
Blowing up therefore produces three exceptional divisors $E_i^{1}, E_i^2, E_i^3$ at every $P_i$, where the $E_i^3$ become sections.
The resulting singular fiber over $T = 0$ has a spine-rib structure again, with 
$X = 0$ as spine and three two-component ribs $E_i^1 \cup E_i^2$ (Kodaira type IV$^*$). The fiber
over $T = \infty$ consists of $L_1, L_2, L_3$, intersecting in $(1:0:0) \in L_3$ (Kodaira type IV). To regard
this as a resolution of singularities for $(\P^1 \times E)/\Gamma$, we need to blow up the intersection 
of $L_1,L_2,L_3$ over $T = \infty$, the new exceptional divisor then becoming the strict transform of the rational curve
$E/\Gamma \subset (\P^1 \times E)/\Gamma$ over $t = \infty$.

As for the semi-flat background metrics on $X$ with one singular fiber removed,
there is a new phenomenon compared to the $\Z_2$ case. Of the
meromorphic $2$-forms $\Omega_0 = dt \wedge dw$, $\Omega_\infty = ds \wedge dw$ ($s = \frac{1}{t}$) on $\P^1_t \times E$
associated to the 
standard flat 
metrics on $\C_t \times E$, $\C_s \times E$, only $\Omega_0$ is invariant under the $\Z_3$-action we are looking at, which is 
generated by $(t,w)
\mapsto (\zeta_3^2 t, \zeta_3 w)$ or 
$(s,w) \mapsto (\zeta_3 s, \zeta_3 w)$ (cf.~the three two-component ribs $E_i^1 \cup E_i^2$ in the 
singular fiber over $T = 0$, which correspond to \emph{crepant resolutions} of
the three $A_2$-singularities of $(\P^1_t \times E)/\Gamma$ over $t = 0$). Thus, when
removing the type IV fiber over $T = \infty$, then obviously the right metric $\omega_0$ to use as a background is
the flat orbifold metric on $(\C_t \times E)/\Gamma$, which is a flat $T^2$-
submersion over a flat $2$-cone
of cone angle $1/3$. On the other hand, when removing the type IV$^*$ fiber over $T = 0$, we would
need to find $\omega_\infty$ with $\omega_\infty^{2} = \Omega \wedge \bar{\Omega}$,
where $\Omega = \frac{1}{T}\Omega_0$ is the unique meromorphic $2$-form on the rational elliptic surface with a simple
pole along that fiber. Thus it seems reasonable to take $\omega_0$ and conformally change
the flat cone metric $|T|^{-\frac{4}{3}} |dT|^2$
on the base (complete at $T = \infty$) by $|T|^{-2}$, which 
results in a flat cone metric (complete at $T = 0$) of cone angle $2/3$. It follows from
Sections 4.2--4.3 that this is indeed what the appropriate $\omega_{{\rm sf},\epsilon}$ looks like.
\hfill $\Box$\end{example}

We now turn to the possible structures of singular fibers in general. 
The principal result in Kodaira \cite{kodaira} asserts the following:
Let $f: U \to \Delta$ be an elliptic fibration with a section $\sigma$
over the unit disk such that all fibers except possibly $D = f^{-1}(0)$ are smooth
and $D$ does not contain any $(-1)$-curves (recall that
existence of a local section is equivalent to $D$ being reduced).
Then the pair
$(f,\sigma)$ is isomorphic to a canonical form $(\bar{f}, \bar{\sigma})$
whose total space $\bar{U}$ is birational, via explicit fiber-preserving maps,
to the quotient of the total space of an explicit 
elliptic fibration over $\Delta$ by
a finite group related to the \emph{monodromy} of $f$.
To explain the notion of monodromy, 
note $U|_{\Delta^*} \cong L/\Lambda$ due to the existence of $\sigma$, where $L$ is a holomorphic line bundle
on $\Delta^*$ and $\Lambda \subset L$ is a lattice bundle, but $H^1(\Delta^*, \mathcal{O}^*) = 0$, so
$L$ is trivial and $U|_{\Delta^*}$ $\cong$ $(\Delta^* \times \C)/(\Z \tau_1 + \Z \tau_2)$ for multi-valued
holomorphic functions $\tau_1, \tau_2$ on $\Delta^*$.
The pair $T = (\tau_1,\tau_2) \in \C^{2\times 1}$, assumed to be positively oriented, transforms as $T \mapsto TA$ with
$A$ $=$ $ (\begin{smallmatrix} a & b \\ c & d\end{smallmatrix}) \in {\rm Sl}(2,\Z)$ 
when going around the puncture once in the counterclockwise sense.
The conjugacy class of $A$ is then referred
to as the monodromy of $f$.

The central fibers that can occur are classified by topological type as I$_b$ ($b \geq 0$), II, III, IV, and
corresponding \textquotedblleft $\ast$-types,\textquotedblright$\;$I$_b^*$, etc.
I$_0$ is smooth, I$_1$ a node, II a cusp, 
III is two rational curves touching at a point, IV three rational curves intersecting transversely in a point, 
and I$_{\textup{\begin{tiny}$b$\end{tiny}}}$ ($b \geq 2$), I$_{\textup{\begin{tiny}$b$\end{tiny}}}^{\textup{\begin{tiny}$*$\end{tiny}}}$ ($b \geq 0$), II$^*$, III$^*$, IV$^*$ 
are graphs of rational curves which if weighted by intersection numbers
become affine Dynkin diagrams: $\tilde{A}_{b-1}$, $\tilde{D}_{b+4}$, $\tilde{E}_8$, $\tilde{E}_7$, $\tilde{E}_6$.
The topological type determines the monodromy, 
whose order in Sl$(2,\Z)$ is finite except for I$_b, {\rm I}_b^*$ ($b \geq 1$).
In Example \ref{isoex} we
have seen how the pairs $({\rm I}_0^*, {\rm I}_0^*)$, $({\rm IV}, {\rm IV}^*)$ occur 
in rational elliptic surfaces birational 
to $(\P^1 \times E)/\Gamma$ with $\Gamma = \Z_2, \Z_3$. The pairs (II, II$^*$),
(III, III$^*$) can be created in a similar way with $\Gamma = \Z_6, \Z_4$. In each case the
monodromy of both fibers is $\Gamma$.
The $T^2$-bundles over $S^1$ obtained by restricting $U$ to a loop around
$0$ are orientable. 
They are Bieberbach
manifolds (five out of six possible ones occur)
if the monodromy is trivial or finite, nil-manifolds for I$_b$, $b \geq 1$ (all possibilities occur), and infranil for I$_b^*$, $b \geq 1$.
      
If $X$ is rational elliptic, then the Euler numbers
of the singular fibers must add up to $\chi(X) = 12$. Persson \cite{miranda} has determined all 
possible configurations
of singular fibers; there are 279 in total, the generic one being $12 \times {\rm I}_1$.
All finite monodromy types can occur, as well as I$_b$, $0 \leq b \leq 9$ ($\chi = b$) 
and I$_b^*$, $0 \leq b \leq 4$ ($\chi = 6 + b$).

We need to understand how Kodaira's canonical form $(\bar{f},\bar{\sigma})$ for a given fibration $(f,\sigma)$ 
with a section is constructed.
Write $U|_{\Delta^*} = (\Delta^* \times \C)/(\Z\tau_1 + \Z\tau_2)$ as before.
Then $\tau := \tau_2/\tau_1: \Delta^* \to \mathfrak{H}$ is multi-valued holomorphic, with values in the upper half-plane, 
and transforms as $\tau \mapsto (d\tau + b)/(c\tau + a)$ under
the image $[A] \in {\rm PSl}(2,\Z)$ of the monodromy when going around the puncture in the counterclockwise sense.
Denote by $j: \mathfrak{H} \to \C$ the classical elliptic
modular function, normalized such that $j(i) = 1$ and $j(\zeta_3) = 0$. 
Recall that $j$ is ${\rm PSl}(2,\Z)$-invariant and unramified
except for branch points of order $2,3$ along the orbits of $i, \zeta_3$. It turns out that
$\mathcal{J} := j \circ \tau$ (Kodaira's \emph{functional invariant}) is a single-valued meromorphic function on $\Delta^*$.

$\bullet$ If $\mathcal{J}(0) \in \C \setminus \{0,1\}$, then $\tau$ is single-valued and extends to a regular
function on $\Delta$ with 
$\tau(0)$ not in the ${\rm PSl}(2,\Z)$-orbits of $i$ or $\zeta_3$. Thus in particular the stabilizer of $\tau(0)$ in 
${\rm PSl}(2,\Z)$ is trivial,
so $A = \pm 1$. If $A = 1$, then $\tau_1, \tau_2$ extend as single-valued
functions on $\Delta$ and $U \cong (\Delta \times \C)/(\Z \tau_1 + \Z \tau_2)$ with $D$ smooth (I$_0$).
If $A = -1$, we make a base change $z = u^2$ over $\Delta_z^*$. The lifted fibration 
$U' \to \Delta_u^*$ extends across $0$
with a smooth central fiber $D'$. The free $\Z_2$-action on $U'$ over $\Delta_u^*$
also extends but with four fixed points on $D'$.
The quotient is then an
elliptic fibration over $\Delta_z$,
isomorphic to $U$ over $\Delta_z^*$, but with four surface
singularities of type $A_1$ over $z = 0$.
Resolving these singularities
gives the desired normal form $\bar{U} \to \Delta_z$ 
with a type I$_0^*$ central fiber, in analogy with
Example \ref{isoex}(i).

$\bullet$ If $\mathcal{J}(0) = 0$, then $\tau$ is possibly multi-valued but still regular at $0$
with $\tau(0)$ in the orbit of $\zeta_3$. $[A]$ fixes $\tau(0)$, 
which yields six possibilities total for $A \in {\rm Sl}(2,\Z)$ up to conjugation, including $A = \pm 1$.
Each case can be treated as before: Make a base change $z = u^{{\rm ord}\,A}$ to obtain an elliptic fibration 
$U' \to \Delta_u^*$ which extends 
to $\Delta_u$ with a smooth central fiber $D'$.
The group $\langle A\rangle$ acts on $U'$ with fixed points 
on $D'$, and the desired normal form $\bar{U}$ is obtained by minimally resolving the singularities of the quotient
and blowing down $(-1)$-curves in the central fiber if necessary.
The case $\mathcal{J}(0) = 1$ is the same. Other than I$_0$, I$_0^*$ for $A = \pm 1$, 
the fiber types that can occur are II, III, IV, and their corresponding $\ast$-types, as in Example \ref{isoex}(ii).

$\bullet$ If $\mathcal{J}$ has a pole of order $b \geq 1$ at $0$, then $A$ is conjugate to $\pm A_b
= \pm (\begin{smallmatrix} 1 & b \\ 0 & 1\end{smallmatrix})$.
Also, it is possible to make $\tau_1 \equiv 1$ and $\tau_2 = \tau = b(\log z)/(2\pi i)$.
If $b = 1$ and $A = A_1$, one can compactify 
$U|_{\Delta^*} \cong (\Delta^* \times \C)/(\Z + \Z \tau)$ explicitly by embedding it
into $\Delta \times \P^2$ using suitable $\wp$-functions and taking the closure. This yields the normal
form $\bar{U}$ with central fiber a node (I$_1$). 
The case $A = A_b$, $b > 1$, with central fiber a cycle of $b$ rational curves (I$_b$),
reduces to $b = 1$ by a fiberwise $b$-fold cover.
$A = -A_b$ implies $A^2 = A_{2b}$, and $\bar{U}$ is obtained by setting
$z = u^2$, which yields an I$_{2b}$ fibration over $\Delta_u$,
then dividing by $\Z_2$ and resolving the four resulting $A_1$-singularities (I$_b^*$).

\subsection{Ricci-flat metrics on complex torus bundles} Let $f: X \to S$ be a holomorphic submersion
such that all fibers $X_s = f^{-1}(s)$ are complex $m$-tori.
Assume $f$ admits a holomorphic section $\sigma$ and a \emph{constant polarization} $\omega$, 
meaning that $\omega$ is a real $2$-form on $X$ which restricts to a K{\"a}hler form on every $X_s$
and there exist $c_i \in \R$ and
$\xi_i(s) \in H^2(X_s,\Z)$
such that $[\omega|_{X_s}] = \sum c_i \xi_i(s)$
for all $s \in S$.
For $\epsilon > 0, s \in S$ we denote by
$g_{s,\epsilon}$ the unique 
flat K{\"a}hler metric on $X_s$ with ${\rm vol}(X_s, g_{s,\epsilon}) = \epsilon$
whose K{\"a}hler class
is a scalar multiple of $[\omega|_{X_s}]$.
Restricting $g_{s,\epsilon}$ to $T_{\sigma(s)}X_s$ then induces a hermitian fiber metric 
$h_\epsilon$ on the holomorphic
vector bundle $E := \sigma^*T_{\textup{\begin{tiny}$X/S$\end{tiny}}}$ over $S$, according to
our convention that $g$ induces $h = \frac{\textup{\begin{tiny}$1$\end{tiny}}}{\textup{\begin{tiny}$2$\end{tiny}}}(g + i\omega)$.

$\bullet$ \emph{Hyperk{\"a}hler setting.} Let $\dim S = m$ and assume that
$X$ admits a holomorphic symplectic form $\Upsilon$
such that all $X_s$ are complex Lagrangians.
Then there exists a unique Riemannian metric $g_{S,\epsilon}$ on $S$ compatible with the complex 
structure on $S$ such that
the faithful pairing $E \otimes T^{1,0}S \to \C$ defined by $\Upsilon$
is isometric with respect to $h_\epsilon$ and the hermitian metric induced by $g_{S,\epsilon}$.

$\bullet$ \emph{CY setting.} Let $\dim S = 1$ and assume that $X$ admits a holomorphic volume form $\Omega$.
Then $\Omega$ defines a faithful pairing $\wedge^m E$  $\otimes$ $ T^{1,0}S \to \C$. Using the hermitian
metric on $\wedge^m E$ induced by $h_\epsilon$, we obtain a
Riemannian metric $g_{S,\epsilon}$ as before.

Using $g_{S,\epsilon}$ and the family of flat metrics $\{g_{s,\epsilon}\}$ on the fibers of $f$, we now wish to
build a submersion
metric on $X$, thus we need to specify a horizontal distribution.
Choose a fiber-preserving
biholomorphism $X \cong E/\Lambda$ for some holomorphic lattice bundle $\Lambda \subset E$. Then
$\Lambda$ induces a flat $\R$-linear connection on $E$, hence an integrable horizontal distribution
$\mathcal{H}$ on $X$, and this is independent of how we made $X \cong E/\Lambda$
since any two such isomorphisms only differ by a holomorphic section $S \to {\rm Aut}(E)$.
Then $g_{{\rm sf},\epsilon}(u,v) :=$ 
$g_{S,\epsilon}(f_*u, f_*v) + g_{s,\epsilon}(Pu, Pv)$ for $u,v \in T_x X$, $s = f(x)$, 
and $P$ $=$ projection along $\mathcal{H}_x$, defines a 
\textquotedblleft semi-flat\textquotedblright$\;$hermitian metric $g_{{\rm sf},\epsilon}$ on $X$. 
We will see that $g_{\textup{\begin{tiny}$S,\epsilon$\end{tiny}}},
g_{\textup{\begin{tiny}${\rm sf},\epsilon$\end{tiny}}}$ are in fact K{\"a}hler and
the volume form of $g_{{\rm sf},\epsilon}$ is a constant multiple of either 
$\Upsilon^{\textup{\begin{tiny}$m$\end{tiny}}} \wedge \bar{\Upsilon}^{\textup{\begin{tiny}$m$\end{tiny}}}$ or $\Omega \wedge \bar{\Omega}$,
so that $g_{{\rm sf},\epsilon}$ is either hyperk{\"a}hler or at least CY.

\begin{remark}\label{sfrem} (i) Existence of a section is necessary
to construct $g_{{\rm sf},\epsilon}$, and $g_{{\rm sf},\epsilon}$ as a tensor field
depends on the particular choice of a section $\sigma$. However, $\sigma$ also defines a
complex Lie group structure on each fiber $X_s$ with unit $\sigma(s)$,
so for every smooth section $\sigma'$ there
is an associated vertical translation map $T(x) = x + \sigma'(f(x))$,
and $g_{{\rm sf},\epsilon}(\sigma)
= T^* g_{{\rm sf},\epsilon}(\sigma')$ if $\sigma'$ is holomorphic. 
This freedom of gauge will be crucial in Section 4.4, when 
we glue $g_{{\rm sf},\epsilon}$ with other K{\"a}hler metrics.
For
a smooth section $\sigma'$, $T$ \emph{preserves} $g_{{\rm sf},\epsilon}$ if and only if
$\sigma'$ is tangent to $\mathcal{H}$ (which implies holomorphic). We thus obtain
a locally constant sheaf 
of Lie groups on $S$ with stalk $T^{2m}$, acting on $(X,g_{{\rm sf},\epsilon})$ by isometries, which
can be viewed as 
an $N$- or $N^*$-structure, cf.~\cite{nt}.

(ii) The hermitian metric $h_\epsilon$ on $E$ induces the Weil-Petersson fiber metric on the first Hodge bundle
$\wedge^mE^*$ up to scale, cf.~\cite[p.618]{song-tian}. 
Taking for granted that
$g_{S,\epsilon}$ 
is K{\"a}hler, Ric$(g_{S,\epsilon})$ must therefore be equal to the 
pull-back of an invariant metric 
on the Siegel upper half plane $\mathfrak{H}_m = \{Z \in \C^{m \times m}: Z^{\rm tr} = Z, \, {\rm Im}\,Z > 0\}$ under the
period map associated to $(f,\sigma,\omega)$. 
Below, we give a careful definition of the period map, and a pedestrian proof of this claim.

(iii) Semi-flat metrics of a similar flavor occur throughout the literature \cite{freed, stringy, gw, hit-clag, loftin, pe-poon, song-tian}, 
but we were unable to find a reference where the construction 
is carried out exactly the way we need it here. We will 
discuss some of these relations
with other work at the end of this section. \hfill $\Box$\end{remark}

We now derive an explicit formula for $g_{{\rm sf},\epsilon}$ in local coordinates
in terms of $\Upsilon,\Omega$ and the periods of the tori
$X_s$ with respect to $\omega$. One can deduce from this formula 
that $g_{{\rm sf},\epsilon}$ is K{\"a}hler and identify the Ricci tensor of the
metric $g_{S,\epsilon}$ on the base with a natural Weil-Petersson metric as mentioned in Remark \ref{sfrem}(ii). 
More importantly, we need this formula in Section 4.3, together with Kodaira's work, to read off
the geometry of the semi-flat metrics explicitly in the elliptic surface case.

Fix isomorphisms $X|_U \cong E|_U/\Lambda \cong (U \times \C^m)/\Lambda$ over a domain $U \subset \C^m, \C$,
with coordinates $z$ on $U$ and $w$ along the fibers. Fix an oriented basis $(\tau_1, ..., \tau_{2m})$ for $\Lambda$ at
one point and extend it to a tuple of multi-valued functions $\tau_i \in \mathcal{O}(U,\C^m)$
that generates $\Lambda$
everywhere. Let $(\xi^1, ..., \xi^{2m})$ be $\R$-dual to $(\tau_1,...,\tau_{2m})$. 
Our assumption 
on the existence of a constant
polarization 
means that there exists $Q \in \R^{2m \times 2m}$, $Q + Q^{\rm tr} = 0$, such that
$\omega = \frac{1}{2} \sum Q_{ij} \xi^i \wedge \xi^j$ restricts to a flat K{\"a}hler metric on each fiber.
The $\omega$-volume of all fundamental cells is then given by the Pfaffian of $Q$.

To see more explicitly what this means, define $T \in \mathcal{O}(U,\C^{m \times 2m})$ by
\begin{equation} \tau_i = (T_{1i}, ..., T_{mi}), \quad i = 1, ..., 2m.\end{equation} 
Now first of all, $T$ may be multivalued, so asking for $\omega$ as above to be
well-defined is saying that upon analytic continuation around a loop in $U$, $T$ must transform
into $TA$ with $A \in {\rm Gl}(2m,\R)$
such that $A^{\rm tr} Q A = Q$. 
Second, as for the positive $(1,1)$ condition, notice that there exists $S \in {\rm Gl}(2m, \R)$, unique up to right
multiplication by a matrix in ${\rm Sp}(2m,\R)$, such that $S^{\rm tr} Q S = (\begin{smallmatrix} 0 & 1 \\ -1 & 0 \end{smallmatrix})$.
We write $TS = R(1,Z)$ with multi-valued holomorphic maps $R: U \to {\rm Gl}(m,\C)$ and $Z: U \to \C^{m \times m}$. Then, 
by Griffiths-Harris \cite[Chapter 2.6]{gh}, $\omega$ is positive $(1,1)$ if
and only if $Z$ belongs to the Siegel upper half plane $\mathfrak{H}_m$.
If this is the case, then, by the same calculations,
\begin{equation} \label{flat} \omega = i H_{jk} dw^j \wedge d\bar{w}^k,  
\quad H^{-1} = 2 \bar{R}({\rm Im} \,Z)R^{\rm tr} = i \bar{T} Q^{-1} T^{\rm tr}.\end{equation}
We refer to $Z: U \to \mathfrak{H}_m$ as the \emph{period map} although this is not, strictly speaking, a well-defined map,
not even locally. Changing the matrix $S$ or the local branch of $T$ that were used in defining $Z$ 
will conjugate $Z$ by a
constant matrix in Sp$(2m,\R)$, i.e.~one has to compose $Z$ locally with a fixed isometry of $\mathfrak{H}_m$.

It is easy to check that the flat connection given by $\Lambda$ has Christoffel symbols
\begin{equation}\Gamma_i(z,w) = \frac{\partial T}{\partial z^i}
\begin{pmatrix} T\\ \bar{T}\end{pmatrix}^{-1}\begin{pmatrix} w \\ \bar{w}\end{pmatrix} \in \C^m, \quad i = 1, ..., {\rm dim}_\C U,\end{equation}
i.e.~the vectors $(e_i, \Gamma_i(z,w))$ span the horizontal space at $(z,w) \in U \times \C^m$.

Lemma \ref{sf} collects the resulting formulas for $g_{{\rm sf},\epsilon}$
and calculates the Ricci form of the metric on the base. In the hyperk{\"a}hler setting, (ia),
assuming $\Upsilon$ to have this particular form is a non-trivial condition, which can however 
always be satisfied by adapting the coordinates to $\Upsilon$. On the other hand, in (ib), 
the local expression for $\Omega$
will have the required form for \emph{every} bundle chart $(z,w)$.

\begin{lemma}\label{sf} For $H$ as in \textup{(\ref{flat})} and $\epsilon > 0$, put $H(\epsilon) := (\epsilon/\sqrt{\det Q})^{1/m}H$.

{\rm (ia)} If $\Upsilon = g(dz^1 \wedge dw^1 + ... + dz^m \wedge dw^m)$ with 
$g: U \to \C$ holomorphic, then the K{\"a}hler form of the hermitian metric $g_{{\rm sf},\epsilon}$ 
is given by
\begin{equation}\omega_{{\rm sf},\epsilon} = i |g|^2 H(\epsilon)^{-1}_{jk} dz^j \wedge d\bar{z}^k +
 iH(\epsilon)_{jk}(dw^j - \Gamma^{j}_{p} dz^p) \wedge (d\bar{w}^k - \bar{\Gamma}^{k}_{q} d\bar{z}^q).\end{equation} 
This is a closed form with top power ${2m \choose m}\Upsilon^m \wedge \bar{\Upsilon}^m$, so $g_{{\rm sf},\epsilon}$ is hyperk{\"a}hler.

{\rm (ib)} If $\Omega = g\, dz \wedge dw^1 \wedge ... \wedge dw^m$ with $g: U \to \C$ holomorphic,
then the K{\"a}hler form of the hermitian metric $g_{{\rm sf},\epsilon}$ is given by 
\begin{equation}\omega_{{\rm sf},\epsilon} = i|g|^2 \det(H(\epsilon))^{-1} dz \wedge d\bar{z} + iH(\epsilon)_{jk}(dw^j - \Gamma^{j} dz) \wedge (d\bar{w}^k - \bar{\Gamma}^{k} d\bar{z}).\end{equation}
This is a closed form with top power $(m+1)! \,i^{(m+1)^2} \Omega \wedge \bar{\Omega}$, so $g_{{\rm sf},\epsilon}$ is Calabi-Yau.

{\rm (ii)} The metrics $\omega_{S,\epsilon}$ on the base are K{\"a}hler as well in both cases and
\begin{equation}\label{siegel}\rho(\omega_{S,\epsilon}) = 
-i\partial\bar{\partial}\log \det({\rm Im}\, Z) = \frac{i}{4}(({\rm Im}\,Z)^{ab} dZ_{bc} \wedge ({\rm Im}\,Z)^{cd} d\bar{Z}_{da}),\end{equation}
the pull-back of an invariant K{\"a}hler metric on $\mathfrak{H}_m$ under the period map $Z$.
\end{lemma}

\begin{proof} Calculation. The semi-flat metrics turn out to be K{\"a}hler because 
the $\tau_i$ are holomorphic, and (\ref{siegel}) follows from $H^{-1} = 2\bar{R}({\rm Im}\,Z)R^{\rm tr}$ as noted in (\ref{flat}).
\end{proof}

We will now make these formulas completely explicit in the elliptic surface case, $m = 1$, which gives
the result that is needed for Sections 4.3--4.4. The mathematical content of the following corollary
is the same as in Gross-Wilson \cite[Example 2.2]{gw}, but the set-up there is not well suited
for our applications in Section 4.3.

\begin{corollary}\label{sfcor} Let $f: X \to S$ be an elliptic fibration over a Riemann surface with 
a holomorphic section $\sigma$ and no singular fibers. Let $\Omega$ be a holomorphic symplectic form
on $X$ and let
$\omega_{{\rm sf},\epsilon}$ 
be semi-flat hyperk{\"a}hler
constructed from  $\sigma$, $(1/\sqrt{2})\Omega$
such that the fibers of $f$ have area $\epsilon$
with respect to $\omega_{{\rm sf},\epsilon}$, so in particular
$\omega_{\textup{\begin{tiny}${\rm sf},\epsilon$\end{tiny}}}^{\textup{\begin{tiny}$2$\end{tiny}}} = 
\Omega \wedge \bar{\Omega}$.
Let $U$ be a domain in $S$, let $z$ be a holomorphic coordinate on $U$, identify $U$ with its $z$-image,
and fix an isomorphism of elliptic fibrations $X|_U \cong (U \times \C_w)/(\Z \tau_1 + \Z\tau_2)$
with multi-valued functions $\tau_1,\tau_2$ so that $(\tau_1,\tau_2)$ 
is positively oriented and $\sigma$
maps to the zero section.
Then $\Omega = g\, dz \wedge dw$ with $g: U \to \C$ holomorphic, and  
\begin{eqnarray}&&\label{gw-sf}\omega_{{\rm sf},\epsilon} = i|g|^2\frac{{\rm Im}(\bar{\tau}_1\tau_2)}{\epsilon} dz \wedge d\bar{z} + \frac{i}{2}\frac{\epsilon}{{\rm Im}(\bar{\tau}_1\tau_2)}(dw - \Gamma dz)\wedge (d\bar{w} - \bar{\Gamma}d\bar{z}),\\ 
&&\label{connex}\Gamma(z,w) = \frac{1}{{\rm Im}(\bar{\tau}_1\tau_2)}({\rm Im}(\bar{\tau}_1 w)\tau_2' - {\rm Im}(\bar{\tau}_2w)\tau_1').\end{eqnarray} 
On $U$, the Ricci form of the induced metric $\omega_{S,\epsilon}$ on the base is given by
\begin{equation}\label{wp}\rho(\omega_{S,\epsilon}) = -i\partial\bar{\partial}\log {\rm Im}(\tau) = \frac{i}{4} \frac{d\tau \wedge d\bar{\tau}}{{\rm Im}(\tau)^2},\end{equation} 
the pull-back of the hyperbolic metric of constant Gau{\ss} curvature $-2$ 
on the upper half plane $\mathfrak{H} = \{{\rm Im}(\tau) > 0\}$ under the period map $\tau = \frac{\tau_2}{\tau_1}: U \to \mathfrak{H}$.
\hfill $\Box$
\end{corollary}

This concludes the discussion of semi-flat metrics for our purposes. As mentioned in Remark \ref{sfrem}(iii), 
similar constructions have been used in a number of recent works. It is interesting to take a closer
look at some of the connections between these. \medskip\

(1) \emph{Collapsing Ricci-flat metrics on elliptic K3 surfaces} (Gross-Wilson \cite{gw}).

Let $f: X \to \P^1$ be an elliptic fibration on a K3 surface with a section and
exactly $24$ singular fibers of type 
I$_1$. Let $\Omega$ be a holomorphic symplectic form on $X$. Then
the complement of the singular
fibers in $X$ carries semi-flat metrics $\omega_{{\rm sf},\epsilon}$ associated to $\Omega$ and a chosen section of $f$.
Since $\Omega$ does not have any poles, these metrics are
incomplete at the singular fibers.
They were used in \cite{gw} to construct
a particular one-parameter family of
collapsing Ricci-flat metrics on $X$ which 
are increasingly well approximated by $\omega_{{\rm sf},\epsilon}$ as $\epsilon \to 0$.
Type I$_1$ singular fibers have a special property which plays a crucial role in \cite{gw}:
Their monodromy has an invariant vector. 
Thus one can assume $\tau_1 \equiv 1$,
and the $N$-structure in Remark \ref{sfrem}(i) has a global
section over $\Delta^*$, i.e.~there exists an isometric $S^1$-action for $\omega_{{\rm sf},\epsilon}$ on all of
$X|_{\Delta^*}$.\hfill $\Box$\medskip\

(2) \emph{K{\"a}hler-Ricci flow on properly elliptic surfaces} (Song-Tian \cite{song-tian}).
 
Let $f: X \to \P^1$ be an arbitrary elliptic fibration on a K3 surface, $[\omega]$ any K{\"a}hler class on $X$,
$[\beta_0]$ any K{\"ahler class on $\P^1$.  
Using methods developed for the analogous question of long-time behavior of the K{\"a}hler-Ricci flow on properly elliptic 
surfaces, it is shown in \cite{song-tian} that the unique CY metric in the K{\"a}hler class 
$[\omega] + t(f^*[\beta_0] - [\omega])$, $t \in [0,1)$,
converges as $t \to 1$ in $C^0_{\rm loc}$ away from the 
singular fibers to $f^*\beta$, where $\beta$ is a smooth
K{\"a}hler metric on $S := \P^1 \setminus \{$singular values of $f\}$ given by
\begin{equation}\label{st}\beta = \frac{\Omega \wedge \bar{\Omega}}{f^*\beta_0 \wedge \omega_{\rm SF}}\beta_0.\end{equation}
Here $\Omega$ is a suitable holomorphic symplectic form on $X$ and
$\omega_{\rm SF}$ denotes \emph{any} 
K{\"a}hler metric on $X\setminus\{$singular fibers$\}$ which restricts to a flat metric on every smooth fiber.
From (\ref{st}), (\ref{gw-sf}), it is then clear that $\beta = \omega_{S,\epsilon}$ for a suitable $\epsilon$.
Even for fibrations as in (1), this result is more general than that from \cite{gw},
which only applies to one particular path of K{\"a}hler classes. However,
\cite{gw} provides an excellent approximate description of the CY metrics along that path by explicit model metrics. 

The identification of the Song-Tian limit $\beta$ with $\omega_{S,\epsilon}$,
which follows from (\ref{st}), allows us to write down explicit formulas for $\beta$ in a neighborhood of each 
singular value of $f$ on $\P^1$, not only in the special case covered by \cite{gw}, cf.~Section 4.3. \hfill $\Box$\medskip\
 
(3) \emph{Affine flat structures} (Loftin \cite{loftin}), \emph{special K{\"a}hler manifolds} (Freed \cite{freed}).

For an elliptic fibration $f: X \to S$ over a Riemann surface with regular fibers and a chosen holomorphic section,
the construction at the beginning of this section yields a flat $\R$-linear connection on $(T^{1,0}S)^*$ and hence an
affine flat structure on $S$. From \cite{loftin}, if $\beta := \omega_{S,\epsilon}$, then the difference 
between the affine flat structure and the Levi-Civita connection of $\beta$ can be regarded as a holomorphic 
cubic differential $U$ on $S$, i.e.~a holomorphic section of $(T^{1,0} S)^{*\otimes 3}$, 
and $\rho(\beta) = |U|^2 \beta$ (\c{T}i\c{t}eica's equation), where $|U|^2$ denotes
norm squared with respect to $\beta$. 
Given $U$, Loftin considers this equation as a PDE for an unknown metric $\beta$
and solves it under suitable boundary conditions on $U$, 
thus producing new affine flat structures.
One observation in \cite{loftin}, attributed to Bryant, states 
that \c{T}i\c{t}eica's equation implies 
$\rho(\rho(\beta)) = -2\rho(\beta)$ when $U \neq 0$, so $\rho(\beta)$ is then a smooth K{\"a}hler 
metric of constant Gau{\ss} curvature $-2$.

The construction of affine flat structures in \cite{loftin} intersects with the setting in (1). 
The meromorphic cubic differential $U$ on $\P^1$ associated to a K3 elliptic fibration as in (1) has 
exactly $24$ poles, all of which are simple. The Weil-Petersson 
metric $\rho(\beta)$
has cusps at the poles of $U$, but also non-orbifold conical singularities 
$|z^m dz|^2$ with $m > 0$ at a finite number
of smooth fibers, due to ramification of the period map 
$\tau: \P^1 \to \mathfrak{H}/{\rm PSl}(2,\Z) = \C$. 
Applying Riemann-Hurwitz to $\tau$, one finds $\sum m = 18$. 
These extra singularities of $\rho(\beta)$ 
correspond precisely to the zeros of $U$, which we recall is a meromorphic 
section of $(T^{1,0}\P^1)^{*\otimes 3} = \mathcal{O}_{\P^1}(-6)$ with $24$ simple poles. 

There is an analogous formalism in the $2m$-dimensional hyperk{\"a}hler setting with
complex $m$-torus fibers, in which $S$ would be referred to as \emph{special K{\"a}hler},
cf.~\cite{freed}. In particular, (1.33) and (1.36) in \cite{freed} provide
an analog of \c{T}i\c{t}eica's equation. \hfill $\Box$

\subsection{Explicit Calabi-Yau metrics near a special fiber} Let $f: U \to \Delta$ be an elliptic fibration
over the disk with all its fibers regular except possibly $D = f^{-1}(0)$.
Assume that $D$ does not
contain $(-1)$-curves, and is reduced or equivalently that $f$ admits a holomorphic section, or yet
equivalently $U|_{\Delta^*} \cong (\Delta^* \times \C)/(\Z \tau_1 + \Z \tau_2)$ 
for some multi-valued holomorphic functions $\tau_1,\tau_2: \Delta^*\to \C$. Let $\Omega$ be a meromorphic $2$-form
on $U$ such that ${\rm div}(\Omega)$ is an integer multiple of $D$, and write $\Omega = g\,dz \wedge dw$, 
$g: \Delta^* \to \C$, where $z$ 
is the coordinate on $\Delta$ and $w$ the one along the fibers.

In this section, we either assume ${\rm div}(\Omega) = -D$, 
which is the interesting case on a rational elliptic surface, or ${\rm div}(\Omega) = 0$, corresponding
to the Gross-Wilson setting.
For each of the possible fiber types, we proceed from Kodaira's explicit normal form $f: U \to \Delta$ 
to construct an explicit isomorphism $\Phi: U|_{\Delta^*} \to (\Delta^* \times \C)/(\Z\tau_1 + \Z \tau_2)$,
meaning that both the map $\Phi$ and the functions $\tau_1,\tau_2$ will be explicit.
This allows us to determine the multiplicity $N$ of $\Phi^*(dz \wedge dw)$ along
$D$, so that $\Phi^*(dz \wedge dw) = f^N h\, dx \wedge dy$ with
$h(p) \neq 0$
in local coordinates $(x,y)$ on $U$ around any point $p \in D$, so $g(z) = z^{-N-1} k(z)$ 
or $g(z) = z^{-N}k(z)$, $k(0) \neq 0$, 
according as ${\rm div}(\Omega) = -D$ or ${\rm div}(\Omega) = 0$. 
Feeding
this into Corollary \ref{sfcor} yields formulas for 
$\omega_{{\rm sf},\epsilon}$ on $U \setminus D$ which are explicit modulo
higher order \textquotedblleft mass terms\textquotedblright$\;$from
the Taylor expansion of $k$.\medskip\

(1) \emph{Trivial monodromy.} In this case one can achieve $\tau_1 \equiv 1$ and $\tau_2 = \tau$, a regular
function from $\Delta$ to $\mathfrak{H}$ which has to satisfy
$D \cong \C/(\Z + \Z \tau(0))$ but is not constrained otherwise,
so that $N = 0$.
The ${\rm div}(\Omega) = 0$ case is not interesting.
If ${\rm div}(\Omega)$ $=$ $-D$, then $\Omega = z^{-1} k(z)\, dz \wedge dw$ with
$k(0) \neq 0$, thus
$$\omega_{{\rm sf},\epsilon} = i |k|^2\frac{{\rm Im}(\tau)}{\epsilon} \frac{dz \wedge d\bar{z}}{|z|^2}
 + \frac{i}{2}\frac{\epsilon}{{\rm Im}(\tau)}(dw - \Gamma dz) \wedge (d\bar{w} - \bar{\Gamma}d\bar{z}),\;
\Gamma = \frac{{\rm Im}(w)}{{\rm Im}(\tau)}\tau'.$$
Note that $\mu := |k(0)|(2 \,{\rm Im}\,\tau(0))^{1/2}$ has an intrinsic interpretation as
$\mu^2 = i \int_D R \wedge \bar{R}$ for the Poincar{\'e} residue $R$ of $\Omega$ along $D$.
Then change coordinates on $\Delta \setminus [0,1)$ by setting
$z = \exp(-u/\frac{\textup{\begin{tiny}$\mu$\end{tiny}}}{\textup{\begin{tiny}$\sqrt{\epsilon}$\end{tiny}}})$, with $u$ ranging in the 
strip $(0,\infty) + i(0, 2\pi \frac{\textup{\begin{tiny}$\mu$\end{tiny}}}{\textup{\begin{tiny}$\sqrt{\epsilon}$\end{tiny}}})$, so that
\begin{equation}\omega_{{\rm sf},\epsilon}
= \frac{i}{2}(du \wedge d\bar{u} + \frac{\epsilon}{{\rm Im}\,\tau(0)} dw \wedge d\bar{w})(1 + O_\epsilon(e^{-\frac{\sqrt{\epsilon}}{\mu}{\rm Re}\,u})).\end{equation}
The error estimate is understood to hold for all derivatives as well, with $w$ ranging in the fundamental domain
spanned by $1, \tau(z(u))$ inside $\{u\} \times \C$.
Thus $\omega_{{\rm sf},\epsilon}$ decays
to a flat cylinder $\R^+ \times S^1 \times D$, where 
$D$ carries its flat K{\"a}hler metric of area $\epsilon$ and the $S^1$ factor has length
$2\pi\frac{\textup{\begin{tiny}$\mu$\end{tiny}}}{\textup{\begin{tiny}$\sqrt{\epsilon}$\end{tiny}}}$,
at an exponential rate of $O_\epsilon(\exp(-r/\frac{\textup{\begin{tiny}$\mu$\end{tiny}}}{\textup{\begin{tiny}$\sqrt{\epsilon}$\end{tiny}}}))$.
 
This result has two noteworthy features. 
First, multiplying $\Omega$ by $\alpha \in \C^*$ directly results in
scaling the $S^1$ factor in the cross-section by $|\alpha|$ while the area of the $T^2$ factor
remains fixed. Second, the lowest eigenvalue of the cross-section
is $-\lambda^2$ with $0 < \lambda \leq \sqrt{\epsilon}/\mu$, with equality for all $\epsilon$ less than
a constant depending on $D$.
Thus, the optimal
decay $e^{-\lambda r}$ of solutions
to the complex Monge-Amp{\`e}re equation with compactly supported data,
on any complete K{\"a}hler manifold which is isometric to 
$(U\setminus D,\omega_{{\rm sf},\epsilon})$
outside a compact set,
is usually not better than, and for small enough $\epsilon$ equal to, 
the decay of the difference between $\omega_{{\rm sf},\epsilon}$ and an exact flat cylinder. \hfill $\Box$\medskip\

We use notation as in our review of Kodaira's work
in Section 4.1. In particular, in
an elliptic fibration $(\Delta^* \times \C)/(\Z \tau_1 + \Z \tau_2)$ where
$T = (\tau_1,\tau_2) \in \C^{2\times 1}$ is positively oriented, we write $A = (\begin{smallmatrix}a & b \\ c & d \end{smallmatrix}) \in {\rm Sl}(2,\Z)$
for the monodromy $T \mapsto TA$ incurred by $T$ when walking around the puncture once in the counterclockwise sense.\medskip\

(2) \emph{Finite monodromy.}
The calculations for all types of singular fibers where $A$ is of finite order 
are fairly similar, so we only discuss the most complicated 
case in detail, namely, the Kodaira type II where the central fiber is 
a cusp. 
Much of what follows is from Kodaira \cite{kodaira} but some additional work is needed to extract the data
we are interested in, i.e.~$\Phi, \tau_1,\tau_2, N$ in the notation introduced above.

Let $\mathcal{J}(0) = 0$, ${\rm mult}_0 \mathcal{J} = m \in \N$.
Classically, $j(\tau) = [(\tau - \zeta_3)/(\tau-\zeta_3^2)]^3 h(j(\tau))$ with $h(0) \neq 0$ if $\tau$ is close 
to $\zeta_3$, cf.~\cite[Abschnitt I, Kapitel 2, \S 15]{klein}. Thus, changing coordinates on $\Delta$, we can make
$(\tau-\zeta_3)/(\tau - \zeta_3^2) = z^{\frac{m}{3}}$, or in other words
\begin{equation}\tau(z) = \zeta_3 \frac{1 - \zeta_3 z^{\frac{m}{3}}}{1 - z^{\frac{m}{3}}}.\end{equation}
Walking around $z = 0$ counterclockwise once, the function $\tau$ transforms by
\begin{equation}\tau \mapsto A\tau = \frac{d\tau + b}{c\tau + a},\;\;A = \begin{pmatrix}a & b \\ c & d\end{pmatrix}
= \begin{cases} 
\pm \begin{pmatrix} 0 & 1 \\ -1 & 1\end{pmatrix} &{\rm if}\; m \equiv 1\; (3),\\
\pm \begin{pmatrix} 1 & -1 \\ 1 & 0 \end{pmatrix} &{\rm if}\; m \equiv 2\; (3),\\ \pm 1 &{\rm if}\;m \equiv 0\;(3).\end{cases}\end{equation}
If $\mathcal{J}\equiv 0$, or $m = \infty$ (the isotrivial case, cf.~Example \ref{isoex}),
any one of these six cases can occur and the following
discussion still applies with $z^\infty = u^\infty := 0$.

We focus on the first case with the plus sign, so $T \mapsto TA$, $A = (\begin{smallmatrix}0 & 1 \\ -1 & 1\end{smallmatrix})$,
${\rm ord}\, A = 6$.
The pull-back $U'$ of $U|_{\Delta^*}$ under $z = u^6$ can be written as 
$(\Delta_u^* \times \C_v)/(\Z + \Z\tau(u^6))$ and thus extends 
to the whole disk $\Delta_u$ with central fiber $\C/(\Z + \Z \zeta_3)$. The action
of the monodromy on $U'$ by deck transformations is generated by
\begin{equation}\label{monact} A(u,v) = \left(\zeta_6 u, \frac{v}{c\tau(u^6) + a}\right) = \left(\zeta_6 u, \zeta_6 \frac{1-u^{2m}}{1-\zeta_3u^{2m}}v\right).\end{equation}
One now checks that the fiberwise linear map 
\begin{equation}\label{bundlemap}
\Phi: \Delta^* \times \C \to \Delta^* \times \C,\; (u,v) \mapsto (u^6, (1-u^{2m})u^5v),\end{equation}
factors through the action (\ref{monact}), thus inducing an isomorphism of elliptic fibrations
$\Phi: U|_{\Delta^*} \to (\Delta^* \times \C)/(\Z \tau_1 + \Z\tau_2)$
with multi-valued generators
\begin{eqnarray}\label{generators}&&\tau_1(z) = {\rm pr}_2\Phi(z^{\frac{1}{6}}, 1) = (1-z^{\frac{m}{3}})z^{\frac{5}{6}},\\
\label{generators2}&&\tau_2(z) = {\rm pr}_2\Phi(z^{\frac{1}{6}}, \tau(z)) = \zeta_3(1-\zeta_3z^{\frac{m}{3}})z^{\frac{5}{6}}.\end{eqnarray} 
What remains to be done before we can write down explicit formulas for the semi-flat metrics is 
find the multiplicity $N$ of $\Phi^*(dz \wedge dw)$ along the central fiber.

To do this, we recall that by Kodaira's theorem, 
$U \to \Delta_z$ is obtained by dividing out $U' \to \Delta_u$ by the monodromy action (\ref{monact}),
which trivially recovers $U$ over $\Delta_z^*$, resolving the singularities that appear on the central fiber
of the quotient, and then blowing down vertical $(-1)$-curves if needed.
Near its
three orbits of fixed points, the action generated by (\ref{monact}) is conjugate
to the $\Z_{\textup{\begin{tiny}$6/k$\end{tiny}}}$-action on $\C^2$ generated by
$(x,y) \mapsto (\zeta_{\textup{\begin{tiny}$6$\end{tiny}}}^{\textup{\begin{tiny}$k$\end{tiny}}} x,
\zeta_{\textup{\begin{tiny}$6$\end{tiny}}}^{\textup{\begin{tiny}$k$\end{tiny}}}y)$ for $k = 1,2,3$, respectively,
where the germ of $x = 0$ corresponds to the central fiber of $U'$. To resolve the quotients, map 
\begin{equation}(x,y) \mapsto (x^{\frac{6}{k}}, y^{\frac{6}{k}}, (x:y)) \in \C \times \C \times \P^1,\;\;k = 1,2,3.\end{equation}
Globally this produces exceptional curves $E_1, E_2, E_3$ and the strict transform $E$ of the central
fiber of $U'/\Z_6$. One checks that the multiplicities of $z \circ \Phi$ 
along these are $1,2,3,6$, whereas 
the multiplicities of $\Phi^*(dz \wedge dw)$ are $1,3,5,10$. It turns out that 
$E$ is a $(-1)$-curve. Blowing down $E$ turns $E_3$ into a $(-1)$-curve, blowing down $E_3$
turns $E_2$ into a $(-1)$-curve, and after blowing down $E_2$ the remaining curve $E_1$ becomes a rational cusp
curve of self-intersection zero. In order to determine $N$, we
consider $\Theta$ $:=$ $(y^2 - x^3)^N dx \wedge dy$ on $\C^2$ 
and reverse the above process, blowing up the origin three times.
By calculation, the multiplicity of $\Theta$ is $N$ along $E_1$, $2N + 1$ along $E_2$, $3N + 2$ along $E_3$, 
and $6N + 4$ along $E$, which nicely shows $N = 1$.

We can now write down formulas for the semi-flat metrics. Since $N = 1$, we have $\Omega = g \,dz \wedge dw$
with $g(z) = z^{-2} k(z)$ or $z^{-1} k(z)$, and $k(0) \neq 0$, if ${\rm div}(\Omega) = -D$ or $0$, respectively.
In the first case, by (\ref{gw-sf}), (\ref{connex}) and (\ref{generators}), (\ref{generators2}),
\begin{eqnarray}\label{sf-ex}&&\omega_{{\rm sf},\epsilon} = i |k|^2 \frac{{\rm Im}(\bar{\tau}_1\tau_2)}{\epsilon} \frac{dz \wedge d\bar{z}}{|z|^4} +
\frac{i}{2} \frac{\epsilon}{{\rm Im}(\bar{\tau}_1\tau_2)}(dw - \Gamma dz)\wedge (d\bar{w} - \bar{\Gamma}d\bar{z}),\\
\label{sf-ex-2}&&{\rm Im}(\bar{\tau}_1\tau_2) = \frac{\sqrt{3}}{2}|z|^{\frac{5}{3}}(1-|z|^{\frac{2m}{3}}),\;\; \Gamma(z,w) = \frac{5}{6}\frac{w}{z}(1 + {\rm error}),
\end{eqnarray}
with error $=$ $\frac{1}{w}({\rm Re}(w) P + {\rm Im}(w)Q)$, $P,Q \in (x,y)\C\{x,y\}$, and $x = z^{m/3}$, $y = \bar{z}^{m/3}$. 
If $m = \infty$, so that the fibration becomes isotrivial, then error $=$ $0$.
To understand this metric more clearly,
change coordinates over $\Delta \setminus [0,1)$
by setting $z = (u/u_0)^{-6}$, $w = (u/u_0)^{-5}v$, 
with $u_0 := 6\sqrt[4]{3}\,\epsilon^{-1/2}|k(0)|$. (This constant can be understood as a renormalized
$L^2$ norm of the residue of $\Omega$ along $D$ but unlike the trivial monodromy case,  
it does not correspond to an intrinsic length scale of the metric.) Here $u$ ranges over the truncated sector
$|u| > u_0$, $0 < {\rm arg}\, u < \frac{2\pi}{6}$, and $v$ comes from a fundamental
cell in $\{u\} \times \C$ which converges to the one spanned by $1,\zeta_3$ as $|u| \to\infty$. Then,
\begin{equation}\label{algeqn}\omega_{{\rm sf},\epsilon} = \frac{i}{2}(du \wedge d\bar{u} + \frac{\epsilon}{{\rm Im}(\zeta_3)}dv \wedge d\bar{v})(1 + O_{\epsilon,m}^{\rm hmg}(|u|^{-2})). \end{equation}
The error term again vanishes in the isotrivial case, $m = \infty$, and \textquotedblleft hmg\textquotedblright$\;$is supposed
to indicate that the error estimate improves by a factor of $|u|^{-1}$ each time we take a derivative with
respect to $u$ or $v$.
Verifying (\ref{algeqn}) is a bit subtle. First, the $du \wedge d\bar{v}$ and $dv \wedge d\bar{u}$ parts of $\omega_{{\rm sf},\epsilon}$
appear to be $O(|u|^{-1})$ only, but the coefficients in front of their leading terms vanish. Second, from the above structure
of the error in (\ref{sf-ex-2}), one can see that the error in (\ref{algeqn}) is of the form 
$F_u(v,\bar{v})$ for a quadratic $F_u$ whose coefficients are power series in $u^{-1},\bar{u}^{-1}$ with quadratic
leading terms. But $\omega_{{\rm sf},\epsilon}$
has to be $(\Z + \Z\zeta_3)$-periodic in $v$ to leading order, so $F_u \equiv F_u(0,0)$.

Thus, over $\Delta \setminus [0,1)$, $\omega_{{\rm sf},\epsilon}$ decays to a flat sector of aperture
$2\pi/6$ in the plane, times a flat torus of area $\epsilon$ homothetic to $\C/(\Z +  \Z\zeta_3)$, and we simply have
to twist
by the monodromy when closing up the slit.
The error is $O^{\rm hmg}(r^{-2})$
in general, so in particular the curvature tensor is at least
$O^{\rm hmg}(r^{-4})$.
When solving a complex Monge-Amp{\`e}re equation
to deform $\omega_{{\rm sf},\epsilon}$ into a global Calabi-Yau metric,
the K{\"a}hler potential
will be $O^{{\rm hmg}}(r^{-\delta})$
for some $\delta > 0$ by Proposition \ref{decay} (in fact, any $\delta < 6$)
so this does not affect the quartic curvature decay. For an isotrivial fibration, 
$\omega_{{\rm sf},\epsilon}$ is \emph{flat}, and the ALG space will have
$|{\rm Rm}| = O^{{\rm hmg}}(r^{-10 + \delta})$ for all $\delta > 0$.

To conclude this example, we look at the Gross-Wilson setting where 
${\rm div}(\Omega) = 0$. For this, we only need to replace $|z|^{-4} dz \wedge d\bar{z}$
in (\ref{sf-ex}) by $|z|^{-2} dz \wedge d\bar{z}$. The resulting semi-flat metric is
then incomplete at the singular fiber, and the induced metric on the \emph{base} is
$|z|^{-\frac{1}{3}}|dz|^2$ to leading order. This describes a neighborhood of the vertex
of a flat cone with non-orbifold cone angle $5/6$.

All finite monodromies can be treated in this way and yield semi-flat metrics with completely 
analogous properties. Table \ref{kodtable} lists $\tau_1,\tau_2,N$, and the tangent cones of the
metrics on the base: incomplete (${\rm div}(\Omega) = 0$) and complete (${\rm div}(\Omega) = -D$).

\begin{remark} (i) Table 4.1 is consistent with Example \ref{isoex}: 
The complete semi-flat metric
obtained when removing
a non-$\ast$-type fiber $D$ is asymptotic to $(\C \times D)/\Z_k$,
where the $\Z_k$-action conjugates into SU$(2)$. Resolving the fixed points 
then in turn produces the dual $\ast$-type fiber over $0 \in \C$.
When removing \emph{this} fiber, however, the asymptotic cone angle is
$\frac{k-1}{k}$, so the cone is not even a quotient unless $k = 2$. The
isotrivial rational elliptic surfaces also explain why the complete cone 
angle when removing one fiber
is the same as the incomplete one when removing its dual.

(ii) The results show that (3.17) in Song-Tian \cite{song-tian} is incorrect. In fact, 
denoting the horizontal component of the incomplete semi-flat metric by $g$ and writing $F$ as in \cite{song-tian}, 
one has $F |ds|^2 = \psi g$ with $\psi > 0$ smooth, 
so $F \sim |s|^{-\alpha}$, some $\alpha > 0$. \hfill $\Box$
\end{remark}

\begin{sidewaystable}\label{kodtable}
\centering
\caption{Kodaira types of singular fibers $D$, data for semi-flat metrics with $D$ removed, tangent cone angles.}
\begin{tabular}{c c c c c c c c c}
\hline
$\mathcal{J}(0)$        & $\;\;\;$mult$_0\mathcal{J}\;\;\;$	& matrix $A$                                        & ${\rm ord}\,A$    & type    & generators $\tau_1,\tau_2$                                                               & $\;\;N\;\;$	& ${\rm div}(\Omega) = 0$	& ${\rm div}(\Omega) = -D$\\ \hline 
$\notin \{0,1,\infty\}$ & any                   	   	& $+1$						    & $1$      		& I$_0$   & $1,\,\tau(z)$                                                                            & $0$ 		& $1$              		& $0$\medskip\ \\
                        &                       	    	& $-1$						    & $2$      		& I$_0^*$ & $z^{\frac{1}{2}},\, z^{\frac{1}{2}}\tau(z)$                                              & $1$ 		& $\frac{1}{2}$    		& $\frac{1}{2}$\medskip\ \\ \hline 
$0$                     & $m \equiv 1 \; (3)$   		& $+\begin{pmatrix} 0 & 1  \\ -1 & 1 \end{pmatrix}$ & $6$      		& II      & $(1-z^{\frac{m}{3}})z^{\frac{5}{6}},\,\zeta_3(1-\zeta_3 z^{\frac{m}{3}})z^{\frac{5}{6}}$ & $1$ 		& $\frac{5}{6}$    		& $\frac{1}{6}$\\
                        &                       		& $-\begin{pmatrix} 0 & 1  \\ -1 & 1 \end{pmatrix}$ & $3$      		& IV$^*$  & $(1-z^{\frac{m}{3}})z^{\frac{1}{3}},\,\zeta_3(1-\zeta_3 z^{\frac{m}{3}})z^{\frac{1}{3}}$ & $1$ 		& $\frac{1}{3}$    		& $\frac{2}{3}$\\
                        & $m \equiv 2 \; (3)$   		& $+\begin{pmatrix} 1 & -1 \\  1 & 0 \end{pmatrix}$ & $6$      		& II$^*$  & $(1-z^{\frac{m}{3}})z^{\frac{1}{6}},\,\zeta_3(1-\zeta_3 z^{\frac{m}{3}})z^{\frac{1}{6}}$ & $1$ 		& $\frac{1}{6}$    		& $\frac{5}{6}$\\
                        &                       		& $-\begin{pmatrix} 1 & -1 \\  1 & 0 \end{pmatrix}$ & $3$      		& IV      & $(1-z^{\frac{m}{3}})z^{\frac{2}{3}},\,\zeta_3(1-\zeta_3 z^{\frac{m}{3}})z^{\frac{2}{3}}$ & $1$ 		& $\frac{2}{3}$    		& $\frac{1}{3}$\\
                        & $m \equiv 0 \; (3)$   		& $+1$						    & $1$      		& I$_0$   & $1,\,\tau(z)$                                                        		     & $0$ 		& $1$           		& $0$\medskip\\\
                        & 		        		& $-1$						    & $2$      		& I$_0^*$ & $z^{\frac{1}{2}},\,z^{\frac{1}{2}}\tau(z)$                                               & $1$ 		& $\frac{1}{2}$   		& $\frac{1}{2}$\medskip\ \\ \hline
$1$			& $m \equiv 1 \; (2)$   		& $+\begin{pmatrix} 0 & 1  \\ -1 & 0 \end{pmatrix}$ & $4$      		& III     & $(1-z^{\frac{m}{2}})z^{\frac{3}{4}},\,i(1+z^{\frac{m}{2}})z^{\frac{3}{4}}$               & $1$ 		& $\frac{3}{4}$    		& $\frac{1}{4}$\\
                        &                       		& $-\begin{pmatrix} 0 & 1  \\ -1 & 0 \end{pmatrix}$ & $4$      		& III$^*$ & $(1-z^{\frac{m}{2}})z^{\frac{1}{4}},\,i(1+z^{\frac{m}{2}})z^{\frac{1}{4}}$               & $1$ 		& $\frac{1}{4}$    		& $\frac{3}{4}$\\
                        & $m \equiv 0 \; (2)$   		& $+1$						    & $1$      		& I$_0$   & $1,\,\tau(z)$                                                             		     & $0$ 		& $1$           		& $0$\medskip\\\
                        &                       		& $-1$						    & $2$      		& I$_0^*$ & $z^{\frac{1}{2}},\,z^{\frac{1}{2}}\tau(z)$                                               & $1$ 		& $\frac{1}{2}$    		& $\frac{1}{2}$\medskip\ \\ \hline
$\infty$                & $-b$                  		& $+\begin{pmatrix} 1 & b  \\  0 & 1 \end{pmatrix}$ & $\infty$ 		& I$_b$   & $1,\,\frac{b}{2\pi i} \log z$                                                            & $0$ 		& $1$  				& $0$\\
                        &                       		& $-\begin{pmatrix} 1 & b  \\  0 & 1 \end{pmatrix}$ & $\infty$ 		& I$_b^*$ & $z^{\frac{1}{2}},\,\frac{b}{2\pi i} z^{\frac{1}{2}}\log z$                               & $1$ 		& $\frac{1}{2}$ 		& $\frac{1}{2}$\\ \hline
\end{tabular}\bigskip\

\bigskip

Remarks: (i) In the last two columns, an entry $\theta > 0$ means that the tangent cone is the cone
over a circle of length $2\pi \theta$. An entry $0$ means the tangent cone is a half-line.
(ii) $\mathcal{J}(z) \equiv 0$ or $\mathcal{J}(z) \equiv 1$ 
can lead to any of the corresponding monodromies, with $z^m := 0$.\hfill $\;$
\end{sidewaystable}

(3) \emph{Infinite monodromy.} For a degeneration of elliptic curves 
$f: U \to \Delta$ with I$_1$ monodromy, $U|_{\Delta^*}$ must be abstractly isomorphic to 
$(\Delta^* \times \C)/(\Z\tau_1 + \Z \tau_2)$, where 
$\tau_1 = 1$ and $\tau_2 = \frac{1}{2\pi i}\log z$. Kodaira also showed that the map 
\begin{equation}\Psi(z,w) = \left(z, -\frac{1}{12}-\frac{1}{4\pi^2}\wp_z(w),\frac{i}{8\pi^3}\wp'_z(w)\right) \in \Delta \times \C^2_{xy} \end{equation}
induces an isomorphism of fibrations from $(\Delta^* \times \C)/(\Z \tau_1 + \Z\tau_2)$ onto
the surface $y^2 = 4x^3 + x^2 - g_2(z)x - g_3(z)$ in $\Delta^* \times \P^2$ with its natural elliptic fibration
given by ${\rm pr}_1$, where $g_2, g_3$ are 
regular on $\Delta$ with $g_2(0) = g_3(0) = 0$ and $\wp_z$ is the standard Weierstra{\ss} function
associated to the lattice generated by $\tau_1(z), \tau_2(z)$. This closes as a smooth elliptic surface $\bar{U}$ in $\Delta \times \P^2$
with central fiber the node $y^2 = 4x^3 + x^2$, which is the Kodaira canonical form for I$_1$. 
Then $\Phi := \Psi^{-1}$ is the map of interest.
To find
the multiplicity $N$ of $\Phi^*(dz \wedge dw)$ along $D$,
we consider $\Theta := z^N dz \wedge (dx/y)$ 
on $\bar{U}$, which has multiplicity $N$ along
$D$. Then $\Psi^*\Theta = 2\pi i\, z^N dz \wedge dw$, so $N = 0$.
By Kodaira again, an I$_b$ degeneration can always be realized over $\Delta^*$ with 
$\tau_1 = 1$, $\tau_2 = \frac{b}{2 \pi i} \log z$, and admits an unramified fiber-preserving 
$b$-fold covering onto an I$_1$ degeneration globally, i.e.~including the central fiber, 
so again $N = 0$. Thus,
\begin{eqnarray}\label{ib}&&\omega_{{\rm sf},\epsilon} = i |k|^2 \frac{b |\log |z||}{2\pi \epsilon} \frac{dz \wedge d\bar{z}}{|z|^2} 
+ \frac{i}{2} \frac{2\pi \epsilon}{b |\log |z||} (dw - \Gamma dz) \wedge (d\bar{w}-\bar{\Gamma}d\bar{z}),\\
&&\Gamma(z,w) = \frac{1}{i}\frac{{\rm Im}(w)}{z |\log |z||}, \end{eqnarray}
gives the semi-flat metric for ${\rm div}(\Omega) = -D$. Replacing $|z|^{-2}dz \wedge d\bar{z}$ by 
$dz \wedge d\bar{z}$, we pass to the ${\rm div}(\Omega) = 0$ setting as
originally considered in \cite{gw}.
We will again 
not treat
the incomplete case carefully, but we note
that the metric on the base satisfies 
$d^2 K_{\rm Gau{\ss}} \sim |\log d\,|^{-1}$, $d = {\rm dist}(0,-)$, and has
tangent cone $\R^2$ at the origin.   

For a degeneration $f: U \to \Delta$ with I$_b^*$ monodromy, 
the pull-back of $U|_{\Delta^*}$ under $z = u^2$ extends as $U' \to \Delta_u$ 
with an I$_{2b}$ central fiber, and $\Z_2$ acts naturally on $U'$ in such a way that the quotient
has four ordinary double points on its central fiber. 
Resolving these produces the Kodaira canonical form $\bar{U}$.
Also, we can assume that $U'|_{\Delta^*}$ $=$ $(\Delta^* \times \C_v)/(\Z \tau_1 + \Z\tau_2)$,
$\tau_1 = 1$,
$\tau_2 = \frac{b}{\pi i}\log u$.
Since $du \wedge dv$ is invariant under the $\Z_2$-action, it pushes down to the quotient 
and then lifts to a volume form on the (crepant) resolution $\bar{U}$. Thus, if $\Omega$ is 
a meromorphic $2$-form on $\bar{U}$ as usual,
with ${\rm div}(\Omega) = -D$ or $0$, then we can write $\Omega = 
g \, du \wedge dv$ on 
a fundamental domain for the $\Z_2$-action
on $U'$, such as $\Delta^+ \times \C$, $\Delta^+ := 
\{u \in \Delta: {\rm Re}\,u > 0\}$, 
with $g(u)$ $=$ $u^{-2} k(u^2)$ or $k(u^2)$, and $k(0) \neq 0$.
Hence, in the complete case, on $\Delta^+ \times \C$, 
\begin{eqnarray}\label{ibstar}&&\omega_{{\rm sf},\epsilon} = i |k(u^2)|^2 \frac{b |\log |u||}{\pi \epsilon}  
\frac{du \wedge d\bar{u}}{|u|^4}
+ \frac{i}{2}\frac{\pi \epsilon}{b |\log |u||}(dv - \Gamma du)\wedge (d\bar{v} - \bar{\Gamma}d\bar{u}),\\
\label{ibstar2}&&\Gamma(u,v) = \frac{1}{i} \frac{{\rm Im}(v)}{u |\log |u||}.\end{eqnarray}
One can push all this down to $\Delta_z^* \times \C_w$ by mapping
$(z,w) = (u^2,uv)$, so $dz \wedge dw = 2u^2 du \wedge dv$ and hence $N = 1$, which is the presentation 
given in Table 4.1 for sake of analogy with the finite monodromy cases, especially I$_0^*$.
However, (\ref{ibstar}), (\ref{ibstar2}) 
are better suited for calculations. In the incomplete case ${\rm div}(\Omega) = 0$, 
the metric 
on the base is simply a $\Z_2$-quotient of the corresponding I$_b$ metric. \medskip\

We now discuss the asymptotic geometry of the complete I$_b$ and I$_b^*$ metrics from (\ref{ib}) and (\ref{ibstar}).
The fibers now \emph{diverge} in the moduli space of flat
tori of constant area, i.e.~the 
intrinsic length of one edge of any fundamental cell tends to $\infty$, while the length of the other edge, corresponding
to the monodromy invariant period $\tau_1$, goes to zero. This effect is sublinear in the distance from
any fixed base point, so in particular the asymptotic cones are $\R^+, \R^2/\Z_2$, as for the complete I$_0,{\rm I}_0^*$ 
metrics, but it makes other features of the geometry
somewhat tricky to understand.

We will determine
$|{\rm Rm}|$ by calculating the norm squared of the curvature form $\Theta$ of the
Chern connection. The exact formulas we use are:
\begin{equation}\label{chern}\Theta = \partial(g^{-1}\bar{\partial}g),\quad |A \, dz^j \wedge d\bar{z}^k|^2 = g^{j\bar{j}}g^{k\bar{k}}{\rm trace}(g^{-1}{A}^{\rm tr}g \bar{A}),
\end{equation}
for any matrix $A \in \C^{m \times m}$, viewed as an endomorphism of the tangent space. 
Also, due to the local isometric $T^2$-action, it suffices to evaluate $|\Theta|^2$ at $w = 0$, where the forms $dz \wedge d\bar{z}$, $dz \wedge d\bar{w}$, $dw \wedge d\bar{w}$ are orthogonal 
with respect to $g_{{\rm sf},\epsilon}$.

$\bullet$ \emph{Geometry of the complete ${\rm I}_b$ metrics.} For sake of reference, note that
\begin{equation}\label{ibintegrals}\int \frac{1}{t}\left(\log \frac{1}{t}\right)^{\frac{1}{2}} dt = -\frac{2}{3}\left(\log \frac{1}{t}\right)^{\frac{3}{2}},
\;\int \frac{1}{t} \log \frac{1}{t} \;dt = -\frac{1}{2}\left(\log \frac{1}{t}\right)^2.\end{equation}
For the rest of this section, $\Delta^* := \{0 < |z| \leq 
1/2\}$ and $U := f^{-1}(\Delta^*)$ 
by abuse of notation. Put $z_x := f(x) \in \Delta^*$ for $x \in U$, 
and $g := 
(b/2\pi \epsilon) |k(z)|^2 |z|^{-2}|\log|z||\,g_{{\rm euc}}$
for the metric on $\Delta^*$ such that $f: (U,g_{{\rm sf},\epsilon}) \to (\Delta^*, g)$ is a Riemannian submersion.
Assume that $(U, g_{{\rm sf},\epsilon})$ is imbedded into a complete Riemannian manifold $M$ 
as the only end of $M$. 
Fix $x_0 \in U$ and let $r_x := {\rm dist}_M(x_0,x)$.

Since the intrinsic 
diameter of $f^{-1}(z)$ and the $g$-length of 
$\{z' \in \Delta: |z'| = |z|\}$ 
are $\sim |\log |z||^{\textup{\begin{tiny}$1/2$\end{tiny}}}$,
we have $r_x \sim {\rm dist}_g(z_x, z_{x_0})\sim |\log |z_x||^{\textup{\begin{tiny}$3/2$\end{tiny}}}$
if $|z_x| \ll 1$ from (\ref{ibintegrals}).
Also, ${\rm vol}(f^{-1}(B),g_{{\rm sf},\epsilon}) = \epsilon \cdot {\rm vol}(B,g)$ for every
$B \subset \Delta^*$ because $f$ is a Riemannian submersion 
with fibers of area $\epsilon$, and hence
$|B(x_{0},s)| \sim s^{\textup{\begin{tiny}$4/3$\end{tiny}}}$ if $s \gg 1$ 
from (\ref{ibintegrals}) because $f$ is distance preserving up to a bounded factor. 
Taken together, this shows that $M$ satisfies CYL$(4/3,1/3)$ from Definition \ref{asympt1} as
well as (trivially) the bound on diameter growth required for Proposition \ref{decay}(ii).
In view of condition SOB$(4/3)$ (Definition \ref{sob-beta}), we 
also need to prove that 
$|B(x,\frac{\textup{\begin{tiny}$1$\end{tiny}}}{\textup{\begin{tiny}$2$\end{tiny}}}r_x)| \geq \frac{\textup{\begin{tiny}$1$\end{tiny}}}{\textup{\begin{tiny}$C$\end{tiny}}}r_{\textup{\begin{tiny}$x$\end{tiny}}}^{\textup{\begin{tiny}$4/3$\end{tiny}}}$
if $r_{\textup{\begin{tiny}$x$\end{tiny}}} \gg 1$.
By the coarea formula and (\ref{ibintegrals}), this would follow if we knew that 
\begin{equation}\label{sob}B(x,\frac{1}{2}r_x) \supset \{y \in U: |z_x| < |z_y| < |z_x|^{1-\alpha}\}\end{equation}
for all $x \in U$ with $r_x \gg 1$ and a
small universal $\alpha > 0$. To show (\ref{sob}),
notice again that both the intrinsic
diameter of $f^{-1}(z)$ 
and the $g$-length of 
$\{z' \in \Delta: |z'| = |z|\}$ 
are $\sim |\log |z||^{\textup{\begin{tiny}$1/2$\end{tiny}}} \sim r_{{\textup{\begin{tiny}$y$\end{tiny}}}}^{{\textup{\begin{tiny}$1/3$\end{tiny}}}}$ 
for any $y \in U$ with $|z_y| = |z|$ as $|z| \to 0$. 

Next, we determine $|{\rm Rm}|$ from the norm of the Chern curvature: 
\begin{equation}\label{curvdecay}|\Theta(x)|^2 = \frac{6\pi^2\epsilon^2}{b^2|k(0)|^4}|\log |z_x||^{-6}(1 + o(1)) \sim r_x^{-4} \end{equation}
as $r_x \to \infty$. This was calculated on a computer, proceeding from (\ref{chern}).

The injectivity radius ${\rm inj}(x) \sim |B(x,1)| \sim r_{\textup{\begin{tiny}$x$\end{tiny}}}^{\textup{\begin{tiny}$-1/3$\end{tiny}}}$ as $r_x \to \infty$.
Roughly speaking, this is because one direction in the fibers
collapses at this rate while all others stay bounded away from zero. 
For the lower bound, recall
${\rm inj}(x) \geq \frac{1}{C}|B(x,1)|$ from \cite[Theorem 4.3]{cgt}
since $|{\rm Rm}|\leq C$. Bearing in mind that the horizontal distribution 
is integrable, consider the horizontal lift, $H$, of $B_g(z_x,\frac{1}{C})$. Then $H$ is a small
totally geodesic disk in $U$ centered at $x$. For each $y \in H$, consider the set $S_y$ of all points
in the fiber through $y$ whose intrinsic distance to $y$ inside the fiber is at most 
$1/C$. Then $B(x,1) \supset \bigcup S_y$, and 
$|S_y| \sim |\log |z_x||^{\textup{\begin{tiny}$-1/2$\end{tiny}}} \sim r_{\textup{\begin{tiny}$x$\end{tiny}}}^{\textup{\begin{tiny}$-1/3$\end{tiny}}}$, so
$|B(x,1)| \geq \frac{\textup{\begin{tiny}$1$\end{tiny}}}{\textup{\begin{tiny}$C$\end{tiny}}} r_{\textup{\begin{tiny}$x$\end{tiny}}}^{\textup{\begin{tiny}$-1/3$\end{tiny}}}$.
As for the upper bound, note 
that for all $x \in U$ there exists a curve $\gamma: [0,t] \to U$ such that $\gamma(0) $ $=$ $\gamma(t) = x$, 
$\gamma$ is a unit speed closed geodesic for the intrinsic flat metric on the fiber through $x$,
and $t \sim r_{\textup{\begin{tiny}$x$\end{tiny}}}^{\textup{\begin{tiny}$-1/3$\end{tiny}}}$.
Then $\gamma$ satisfies $\ddot{\gamma} + A(\dot{\gamma},\dot{\gamma}) = 0$, where $A$ denotes the 
second fundamental form of
the fiber, and $|A| \leq C$
because the fiber is flat and $|{\rm Rm}| \leq C$,
so $|\det A| \leq C$, and ${\rm tr}\, A = 0$ 
because the fiber is a holomorphic curve.
Hence $|\ddot{\gamma}|$ is bounded. 
Put $\iota := {\rm inj}(x)$. 
From \cite[Satz 2.1]{joka}, since $|{\rm Rm}| \leq C$,
there exists a diffeomorphism $\Phi: B = B_{\R^4}(0,\iota) \hookrightarrow U$  such that
$$\Phi(0) = x, \;\; \frac{1}{C}g_{\rm euc} \leq \Phi^* g_{{\rm sf},\epsilon} \leq C g_{\rm euc},
\;\;|\Gamma_{jk}^l [\Phi^*g_{{\rm sf},\epsilon}](v)| \leq C|v|.$$
Therefore, if $\iota > Ct$, then $\alpha := \Phi^{-1}\circ\gamma$ is a smooth loop in $B$ with 
$|\dot{\alpha}| \geq 1/C$ and
$|\ddot{\alpha}| \leq C$ in the Euclidean sense, hence $t \geq 1/C$, a contradiction for $r_x \gg 1$.

To conclude, we compute the asymptotic cone of $M$. 
Since the intrinsic diameter of the fiber through $x$
is only $\sim r_{\textup{\begin{tiny}$x$\end{tiny}}}^{\textup{\begin{tiny}$1/3$\end{tiny}}}$,
it suffices to show that the pointed 
rescalings $(\Delta^*, \lambda^2 g, \frac{1}{2})$ converge in the pointed Gromov-Hausdorff sense as $\lambda \to 0$.
Define 
\begin{equation} \Phi_\lambda: [0,\infty) \times S^1 \ni (s,\theta) \mapsto  \frac{1}{2}\exp(-\lambda^{-\frac{2}{3}}s^{\frac{2}{3}})e^{i\theta} \in \Delta^*,\;\;\Phi_\lambda(0,0) = \frac{1}{2}.\end{equation}
It is easy to check that $\Phi_\lambda^*(\lambda^2 g)$ converges to $C ds^2$ as smooth tensor fields,
locally uniformly on $(0,\infty) \times S^1$.
Thus $M$ has a unique asymptotic cone, $\R^+$.

$\bullet$ \emph{Geometry of the complete \textup{I}$_b^*$ metrics.} Observe that as $s \to 0$,
\begin{equation}\label{ibstarint}\begin{split}
&\int_{s}^{s_0} \frac{1}{t^2}\left(\log \frac{1}{t}\right)^{\frac{1}{2}} dt 
= \frac{1}{s}|\log|s||^{\frac{1}{2}}(1 + O_{s_0}(|\log |s||^{-1})), \\
&\int \frac{1}{t^3} \log \frac{1}{t}\;dt = -\frac{1}{2t^2}\log\frac{1}{t} + \frac{1}{4t^2}.\end{split}\end{equation}
From (\ref{ibstarint}), and using notation as before,
$r_{\textup{\begin{tiny}$x$\end{tiny}}} \sim |u_{\textup{\begin{tiny}$x$\end{tiny}}}|^{\textup{\begin{tiny}$-1$\end{tiny}}}|\log|u_{\textup{\begin{tiny}$x$\end{tiny}}}||^{\textup{\begin{tiny}$1/2$\end{tiny}}}$ because the fiber over
$u \in \Delta^+$ has intrinsic diameter $\sim |\log|u||^{\textup{\begin{tiny}$1/2$\end{tiny}}}$
and the $g$-length of the half-circle $\{u' \in \Delta^+: |u'| = |u|\}$ is $\sim |u|^{\textup{\begin{tiny}$-1$\end{tiny}}}|\log|u||^{\textup{\begin{tiny}$1/2$\end{tiny}}}$. 
Hence the intrinsic diameter of the fiber through $x$ is $\sim (\log r_x)^{\textup{\begin{tiny}$1/2$\end{tiny}}}$,
and the $g$-length of the half-circle through $u_x$ is $\sim r_x$,
i.e.~exactly linear in the distance from any fixed base point rather than sublinear.
This is still good enough for Proposition \ref{decay}(ii), but not (ia).
From this and (\ref{ibstarint}), $|B(x_0,s)| \sim s^2$ if $s \gg 1$. In place of (\ref{sob}) we now have
\begin{equation}\label{sobstar}|B(x,\frac{1}{2}r_x)| \supset \{y \in U: |u_x| < |u_y| < (1+\alpha)|u_x|, \; |{\rm arg}\frac{u_y}{u_x}| < \alpha \} \end{equation}
for all $x \in U$ with $r_x \gg 1$ and a small universal $\alpha > 0$, but
(\ref{sobstar}) and (\ref{ibstarint}) imply $|B(x,\frac{{\textup{\begin{tiny}$1$\end{tiny}}}}{{\textup{\begin{tiny}$2$\end{tiny}}}}r_x)| \geq \frac{{\textup{\begin{tiny}$1$\end{tiny}}}}{{\textup{\begin{tiny}$C$\end{tiny}}}}r_x^{\textup{\begin{tiny}$2$\end{tiny}}}$ as before, and hence SOB($2$).
The Chern curvature
\begin{equation}|\Theta(x)|^2 =  \frac{\pi^2\epsilon^2}{2b^2|k(0)|^4}|u_x|^4|\log|u_x||^{-4}(1 + o(1)) \sim r_x^{-4}(\log r_x)^{-2}\end{equation}
as $r_x \to \infty$. One then proves ${\rm inj}(x) \sim |B(x,1)| \sim (\log r_x)^{-1/2}$ if $r_x \gg 1$, using the
same arguments as before. The asymptotic cone is obtained as the pointed limit of 
$(\Delta^+, \lambda^2 g, \frac{1}{2})$
as $\lambda \to 0$. If $u_\lambda \in \R^+$ is defined by 
$u_\lambda |\log |u_\lambda||^{\textup{\begin{tiny}$-1/2$\end{tiny}}} = \lambda$, and
\begin{equation}\Phi_\lambda: \{|u| \geq u_\lambda,\, {\rm Re}\,u > 0\} \ni u \mapsto \frac{u_\lambda}{2u}\in \Delta^+, \;\;\Phi_\lambda(u_\lambda) = \frac{1}{2},\end{equation} 
we see that $\Phi_\lambda^*(\lambda^2g)$ converges as smooth tensor fields to $C|du|^2$, locally uniformly on 
$\{{\rm Re}\,u> 0\}$. Thus $M$ has a unique
asymptotic cone, given by $\R^2/\Z_2$.\hfill $\Box$

\subsection{Proof of Theorems \ref{main} and \ref{main1}} Let $f: X \to \P^1$ be 
the rational elliptic fibration obtained by blowing up 
the base points of a pencil of cubics. Let $p \in \P^1$ and $D := f^{-1}(p)$. 
Fix a disk 
$\Delta = \{|z| < 1\} \subset \P^1$ with $z(p) = 0$ such
 that $f$
has no singular fibers over $\Delta^*$. Write 
$X|_{\Delta^*} = (\Delta^* \times \C)/(\Z \tau_1 + \Z \tau_2)$ with generators $\tau_1,\tau_2$ 
as in Table 4.1. Let 
$\Omega$ be a meromorphic $2$-form on $X$ 
with ${\rm div}(\Omega) = -D$. Fix a K{\"a}hler metric $\omega$ on $M := X \setminus D$, complete or incomplete,
but with
$\int_{M} \omega^2 < \infty$ and $\int_C \omega = 0$
for all bad cycles $C$.
Choose 
a holomorphic section $\sigma: \Delta^* \to X$, and let 
$\epsilon$ denote the $\omega$-area of the fibers of $f$.
Then for all $\alpha > 0$ we have a semi-flat metric $\omega_{{\rm sf}}(\alpha)$ on $X|_{\Delta^*}$, depending 
on $\sigma$, such that $\omega_{{\rm sf}}(\alpha)^2 = \alpha \Omega \wedge \bar{\Omega}$ and the fibers
of $f$ have area $\epsilon$ with respect to $\omega_{{\rm sf}}(\alpha)$ as well. The geometry of $\omega_{\rm sf}(\alpha)$
was studied in Section 4.3. 
Our goal is to glue $\omega$ and $\omega_{\rm sf}(\alpha)$ 
to produce a complete K{\"a}hler metric $\omega_0$ on $M$ 
to which the existence and decay results from Section 3 apply.

\begin{proposition}\label{extend} There exist a holomorphic section $\tilde{\sigma}$ over $\Delta^*$,
concentric disks $\Delta' \subset \Delta'' \subset \Delta''' \subset \Delta$, a $(1,1)$-form $\beta \geq 0$ on
$\P^1$ such that ${\rm supp}(\beta) \subset \Delta''' \setminus \Delta'$,
and a constant $\alpha_0 > 0$, such that for all $\alpha > \alpha_0$ there exist
$u_\alpha^{\rm int} \in C^\infty(f^{-1}(\P^1\setminus\Delta'),\R)$, $u_\alpha^{\rm ext} \in C^\infty(f^{-1}(\Delta'' \setminus\{0\}),\R)$
whose complex Hessians coincide over $\Delta'' \setminus \Delta'$, and 
$t_\alpha > 0$, so that for all $t > t_\alpha$ the closed $(1,1)$-form $\omega_0(\alpha,t) := 
\omega + t f^*\beta + i\partial\bar{\partial} u_\alpha^{\rm int,ext}$ on $M$ is positive,
and $\omega_0(\alpha,t) = \omega$ over $\P^1 \setminus \Delta'''$, $\omega_0(\alpha,t) = T^*\omega_{\rm sf}(\alpha)$
over $\Delta'\setminus\{0\}$, where $T$ denotes vertical translation by $\tilde{\sigma}$ relative to $\sigma$
as defined in Remark \ref{sfrem}(i).
Moreover, $\int_M (\omega_0(\alpha,t)^2 - \alpha\Omega\wedge\bar{\Omega}) = 0$ for exactly one $t > t_\alpha$.
\end{proposition}

The basic idea here is that there are two obstructions for gluing $\omega$ and $\omega_{\rm sf}(\alpha)$ as 
closed $(1,1)$-forms, coming from $H^2(X|_{\Delta^*},\R)$ and $H^{0,1}(X|_{\Delta^*})$. We killed the first one
by making $\langle \omega_{{\rm sf}}(\alpha), F\rangle = \langle
\omega, F\rangle$ for every fiber $F$ and $\langle \omega, C\rangle = 0$
for every bad cycle $C$.
For these, $\langle \omega_{{\rm sf}}(\alpha),C\rangle = 0$
can be verified directly, and together with $[F]$ they generate
$H_2(X|_{\Delta^*},\Z)$.
The second obstruction
is killed by translating $\omega_{{\rm sf}}(\alpha)$ by a section $\tilde{\sigma}$, thus balancing 
the freedom we had in choosing $\sigma$ in the first place.
So far this emulates work from Gross-Wilson \cite{gw}
for the case of an I$_1$ singular fiber.
One issue not present in \cite{gw} is that 
$\omega$ and $\omega_{{\rm sf}}(\alpha)$ 
need not be close on their overlap,
thus gluing them by cutting off the K{\"a}hler
potential
of their difference 
might create
large negative $dz \wedge d\bar{z}$ components.
We balance
these 
by adding $tf^*\beta$ with a bump 
form $\beta \geq 0$ on $\P^1$, $t \gg 1$.
The point then is that by doing everything carefully, 
the integrability condition 
$\int (\omega_0(\alpha,t)^2 - \alpha\Omega \wedge \bar{\Omega}) = 0$ becomes linear 
in $\alpha$ and $t$, with positive slope. (Notice
$T^*\Omega = \Omega$ so the integrand here has compact support.)\medskip\

\noindent \emph{Proof of Theorems \ref{main} and \ref{main1}, assuming Proposition \ref{extend}.} Most things are
obvious now from Propositions \ref{solve}, \ref{decay}, \ref{extend}, Lemma \ref{hmgcrit}, 
and the discussion of the
geometry of $\omega_{\rm sf}(\alpha)$ in Section 4.3. Note that $[\omega_{\rm CY}] = [\omega]$ in $H^2(M,\R)$ by homotopy invariance because this is obviously true in $H^2(X|_{{\textup{\begin{tiny}$\P^1\setminus \Delta$\end{tiny}}}},\R)$.
As for the injectivity radius, recall that $\omega_{\rm sf}(\alpha)$ satisfies
${\rm inj}(x) \sim |B(x,1)|$ in the I$_{\textup{\begin{tiny}$b$\end{tiny}}}$ and ${\rm I}_{\textup{\begin{tiny}$b$\end{tiny}}}^*$ cases. Now 
again from \cite[Theorem 4.3]{cgt}, since $\omega_{\rm CY}$ has bounded curvature,
${\rm inj}(x) \geq \frac{\textup{\begin{tiny}$1$\end{tiny}}}{\textup{\begin{tiny}$C$\end{tiny}}}|B(x,1)|$ for $\omega_{\rm CY}$, so the lower bounds are preserved.
On the other hand, $(M,\omega_{\rm CY})$ contains
sufficiently short loops of bounded geodesic curvature through every point because this is true for $\omega_{\rm sf}(\alpha)$,
so the same argument as for $\omega_{\rm sf}(\alpha)$ yields the desired sharp upper bound.
The claims about $|\nabla^k{\rm Rm}|$ follow from scaling and local Shi estimates.
\hfill $\Box$
    
\begin{remark}\label{growsol} (i) The form $\beta$ in Proposition \ref{extend} can be made compactly supported 
in an annulus $\{r_1 < |z| < r_2\} \subset \Delta$ and radial, so $\beta = i\partial\bar{\partial} \phi$ with $\phi = 0$
for $|z| > r_2$ and $\phi \sim \log |z|$ for $|z| < r_1$.
Since the metric $g$ on $\Delta^*$ induced by $g_{{\rm sf}}(\alpha)$ is
conformal to $|dz|^2$ and $g_{\rm sf}(\alpha)$ 
has the same volume growth as $g$, this shows that $\phi \circ f$ grows 
like a Green's function for $g_0(\alpha,t)$.
Therefore the equation $(\omega_0(\alpha, t) + i\partial\bar{\partial}u)^2 = \alpha \Omega \wedge \bar{\Omega}$ 
is solvable for \emph{all} $\alpha,t \gg 1$ but with $u$ growing like a Green's function in general.

(ii) Every bounded harmonic function $u: \Delta^* \to \R$ satisfies $u = \bar{u} + O(|z|)$ for some constant $\bar{u}$, and $r \sim |\log |z||$ in the ALH case, $r \sim |z|^{-\theta}$ in the ALG case with cone angle $\theta$, 
$r \sim |\log|z||^{3/2}$ for I$_b$, and $r \sim |z|^{-1/2}|\log |z||^{1/2}$ for I$_b^*$. Therefore, the
decay estimates from Proposition \ref{decay} are qualitatively sharp in this setting.
 \hfill $\Box$\end{remark}

\noindent \emph{Proof of Proposition \ref{extend}.} Write $\Delta(r) := \{|z| < r\} \subset \Delta$. 
For $0 < s < r < \frac{\textup{\begin{tiny}$1$\end{tiny}}}{\textup{\begin{tiny}$4$\end{tiny}}}$, choose
a radial cut-off function $\psi = \psi(r,s)$, $0 \leq \psi \leq 1$, with 
$\psi \equiv 1$ on $\Delta(r+s)$, 
${\rm supp}(\psi)$ $\subset$ $\Delta(r + 2s)$,
and $s|\psi_z| + s^2 |\psi_{z\bar{z}}| \leq C_0$, where
here as well as in the rest of the proof,
$C_0$ stands for a generic constant which is allowed to depend on $X,\Omega,\omega,\tau_1,\tau_2$ 
only. If a constant may also depend on $r,s$, we write $C_0(r,s)$. 
Even though $r,s$ will only depend on $X,\Omega,\omega,\tau_1,\tau_2$ eventually,
this distinction is needed to avoid a cycle. 
Fix a $(1,1)$-form
$\beta$ on $\P^1$ with ${\rm supp}(\beta) \subset \Delta(r+3s)\setminus\Delta(r)$,
$0 \leq \beta \leq |dz|^2$,
and $\beta = |dz|^2$ 
on $\Delta(r+2s)\setminus \Delta(r+s)$, identifying $\beta$ with its associated symmetric bilinear form.
We also identify $\psi,\beta$ with their pull-backs under $f$ for convenience.\medskip\

\noindent {\bf Claim 1 ($\partial\bar{\partial}$-lemma).} There exist a holomorphic section $\tilde{\sigma}$ of $f$ over
$\Delta^*$ and a smooth function $u_1: X|_{\Delta^*} \to \R$ such that
$T^*\omega_{\rm sf}(1) = \omega + i \partial\bar{\partial} u_1$, where $T$ denotes
translation by $\tilde{\sigma}$ relative to $\sigma$ as defined in Remark \ref{sfrem}(i).\medskip\

\noindent \emph{Proof.} For the case of an I$_1$ singular fiber and ${\rm div}(\Omega) = 0$, this is 
Lemma 4.3 in \cite{gw}. Nothing essential changes in the general case, but since this proof is quite
delicate,
we review it here, and give a slightly more direct argument for the last step.

First we need to see why 
$[\omega_{\rm sf}(1)] = [\omega]$ as de Rham cohomology classes on $X|_{\Delta^*}$.
Since $\langle \omega_{{\rm sf}}(1), F\rangle = \langle
\omega, F\rangle$ and $\langle \omega, C\rangle = 0$
for all bad cycles $C$, and bad cycles are Lagrangian for $\omega_{\rm sf}(1)$ up to isotopy,
it would be enough to show that $H_2(X|_{\Delta^*},\Z)$ is 
generated by $[F]$ and the classes of the bad cycles.
Retract $X|_{\Delta^*}$ to a $T^2$-bundle 
$Y \to S^1$. By the Leray 
spectral sequence, $H_2(Y,\Z)/\Z[F] \cong H_1(S^1, \mathcal{A})$, with
$\mathcal{A}$ the rank-$2$ local system on $S^1$ defined by the monodromy $A \in {\rm Sl}(2,\Z)$.
Write $S^1$ as the union of a $0$-cell and a $1$-cell and set up the corresponding
standard chain complex $C_\bullet(S^1, \mathcal{A})$ as in \cite[p.328]{hatcher}. We have
$\Z[\pi_1(S^1)] \cong \Z[t,t^{-1}] =: R$,
and both $C_0$ and $C_1$ are isomorphic to $R \otimes_R \Z^2 \cong \Z^2$ as $R$-modules, 
where $R$ acts on $\Z^2$ through $A$. The boundary operator then becomes
$(t - 1) \otimes {\rm id} \cong A - {\rm id}$, hence the claim.

Next, if $\zeta$ is a smooth real $1$-form on $X|_{\Delta^*}$ with $d\zeta$ of type $(1,1)$, then
$\xi := \zeta^{0,1}$ satisfies $\bar{\partial}\xi = 0$. From the Leray spectral sequence again,
 $H^{0,1}(X|_{\Delta^*}) \cong H^0(\Delta^*, \mathcal{H})$ for the holomorphic line bundle
$\mathcal{H} = R^1f_*\mathcal{O}_X$, whose fiber over $z$ can be naturally
identified with $H^{0,1}(f^{-1}(z))$. 
The fiberwise constant $(0,1)$-form $\Theta := d\bar{w}/{\rm Im}(\bar{\tau}_1\tau_2)$ defines
a holomorphic section of $\mathcal{H}$. Thus, restricting $\xi$ to $f^{-1}(z)$ and writing the constant part of
the restriction as $s(z)\Theta(z)$ identifies the class of $\xi$ in $H^{0,1}(X|_{\Delta^*})$ with 
$s: \Delta^* \to \C$, which is
necessarily holomorphic. Apply this to $d\zeta = \omega_{\rm sf}(1) - \omega$. 
Next, for any given holomorphic function $\tilde{\sigma}: \Delta^* \to \C$, 
which we view as inducing a section
of $X$,
if $T(z,w) := (z, w + \tilde{\sigma}(z))$, a direct calculation 
(best done using the frame $\partial_{\rm h},\partial_{\rm v}$ from \cite{gw}, the key identity being $\partial_{\rm h}\bar{\Gamma} = 0$) shows that
\begin{equation}\label{balance}
T^*\omega_{\rm sf}(1) - \omega_{\rm sf}(1) = d\tilde{\zeta} + \frac{i}{2} \frac{\epsilon}{{\rm Im}(\bar{\tau}_1\tau_2)} |\tilde{\sigma}'(z) - \Gamma(z,\tilde{\sigma}(z))|^2 dz \wedge d\bar{z},
\end{equation}
\begin{equation}\tilde{\xi} = \tilde{\zeta}^{0,1} := \frac{i}{2}\frac{\epsilon}{{\rm Im}(\bar{\tau}_1\tau_2)}\tilde{\sigma}(z) (d\bar{w} - \bar{\Gamma}(z,w) d\bar{z}).\end{equation}
Thus $\xi + \tilde{\xi}$ becomes $\bar{\partial}$-exact if $\tilde{\sigma} := 2is/\epsilon$. 
The correction term on the rhs in (\ref{balance}) has a smooth real potential as well, as
does every smooth real $2$-form on every open Riemann surface $S$ because $H^2(S,\R) = H^{0,1}(S) = 0$. \hfill $\Box$\medskip\

We now use Claim 1 to construct a potential for $T^*\omega_{\rm sf}(\alpha) - \omega$,
$\alpha > 0$ arbitrary, which depends linearly on $\alpha$. With $\Omega = z^{-(N+1)} k(z)dz \wedge dw$, $k(0) \neq 0$, as usual,
\begin{equation}\label{onthebase}\omega_{\rm sf}(\alpha) - \omega_{\rm sf}(1) = i(\alpha - 1)|k(z)|^2\frac{{\rm Im}(\bar{\tau}_1\tau_2)}{\epsilon}
\frac{dz \wedge d\bar{z}}{|z|^{2N+2}} = (\alpha-1)i\partial\bar{\partial}u\end{equation}
for a smooth function $u: \Delta^* \to \R$, which must exist for the same general reason as in the proof of Claim 1.
We identify $u$ with $u \circ f$.
Thus, with $T, u_1$ as before, 
\begin{equation}\label{pot1} T^*\omega_{\rm sf}(\alpha) = \omega + i\partial\bar{\partial}u_\alpha, \quad u_\alpha := u_1 + (\alpha-1)u.\end{equation}
Notice at this point that neither $u_1$ nor $u$ are uniquely determined, but we make a fixed choice for both of them
and then allow any generic constant $C_0$ to depend on these choices as well. It will be necessary to normalize
$u$ by a harmonic function in the gluing region $\Delta(r + 2s)\setminus\Delta(r+s)$ 
to gain more specific
control there, but then the $(r,s)$-dependence of this correction needs to be tracked carefully.\medskip\

\noindent {\bf Claim 2 (normalizing the potential).} If $C_0 r < 1$, $C_0 s < r$, and $v$
is harmonic with the same boundary values as $u$ on $\Delta(r + 3s)\setminus\Delta(r)$, then
\begin{equation}\label{bernstein}\sup_{\Delta(r+2s)\setminus\Delta(r+s)}(s^{-2}|u-v| + s^{-1}|(u-v)_z|) \leq C_0\sup_{\Delta(r+3s)\setminus\Delta(r)} u_{z\bar{z}}.\end{equation}
\noindent \emph{Proof.} Put $w := u - v$ and take $b(z) := (|z|-r)(r + 3s - |z|)$ as a barrier. If $3s \leq r$, 
then $4b_{z\bar{z}} \leq -1$, thus
$\sup |w| \leq 4(\sup |b|)(\sup |w_{z\bar{z}}|) = 9 s^2 \sup u_{z\bar{z}}$ by the maximum principle, where
all suprema are over $\Delta(r + 3s)\setminus\Delta(r)$. We use Bernstein's method
for
the gradient estimate. Thus we fix a cut-off function $\chi$ with $0 \leq \chi \leq 1$, 
${\rm supp}(\chi)$ $\subset$ $\Delta(r + 3s)\setminus\Delta(r)$,
$\chi \equiv 1$ on $\Delta(r + 2s)\setminus\Delta(r+s)$,
$s\chi^{\textup{\begin{tiny}$-1/2$\end{tiny}}}|\chi_z| + s^{\textup{\begin{tiny}$2$\end{tiny}}}(\chi_{z\bar{z}})^- \leq C_0$, and define
$S := \chi|w_z|^2 + C_0s^{-2}|w|^2$. By the usual calculation,
$$S_{z\bar{z}} \geq -C_0\sup_{\Delta(r+3s)\setminus\Delta(r)}(|w_{z\bar{z}}|^2 + s^2|w_{z\bar{z},z}|^2).$$
Thus, using the barrier $b$ from the first step and the maximum principle again, 
$$\sup_{\Delta(r+2s)\setminus\Delta(r+s)}|w_z| \leq C_0 s \sup_{\Delta(r+3s)\setminus\Delta(r)}(u_{z\bar{z}} + s|u_{z\bar{z},z}|).$$
From (\ref{onthebase}) and Table 4.1, certainly $s|u_{z\bar{z},z}| \leq C_0 u_{z\bar{z}}$ on $\Delta(r + 3s)\setminus\Delta(r)$. \hfill $\Box$\medskip\

We now construct the desired complete background K{\"a}hler metric $\omega_0$. For every
$\alpha > 0$ and $t > 0$, define a closed $(1,1)$-form $\omega_0(\alpha,t)$ on $M$ by setting
\begin{eqnarray}\label{properties1}
&&{\omega}_0(\alpha,t) := \begin{cases} \omega + t \beta + i \partial\bar{\partial}(\psi \tilde{u}_\alpha)
&{\rm outside} \; X|_{\Delta(r)},\\
\omega + t\beta + i\partial\bar{\partial}u_\alpha  &{\rm over}\;\Delta(r+s)^*,\end{cases}\\
\label{pot2}&&\tilde{u}_\alpha := u_1 + (\alpha-1)(u-v)\;\,{\rm on}\;\Delta(r + 3s)\setminus\Delta(r), \end{eqnarray}
where we identify $v$ with $v \circ f$. The definitions match over $\Delta(r+s)\setminus\Delta(r)$
because $\psi \equiv 1$ there and 
$u_\alpha - \tilde{u}_\alpha = (\alpha - 1)v$, which is pluriharmonic.
Note that
\begin{equation}\label{properties2}\omega_0(\alpha,t) = \begin{cases} \omega &{\rm outside}\;X|_{\Delta(r + 3s)},\\
T^*{\omega}_{\rm sf}(\alpha) &{\rm over}\;\Delta(r)^*.\end{cases}\end{equation}
\noindent {\bf Claim 3 (positivity of the glued form).} There exists a constant $C_0(r,s) > 0$ 
such that if $t > C_0(r,s) + C_0 |\alpha-1|\sup_{\Delta(r+3s)\setminus \Delta(r)} u_{z\bar{z}}$,
then, on all of $M$,
\begin{equation}\label{positive}
{\omega}_0(\alpha,t) \geq \frac{1}{2}(\omega + \psi i \partial\bar{\partial}u_\alpha) > 0. \end{equation}
\noindent \emph{Proof.} This only needs to be checked over $\Delta(r + 2s)\setminus \Delta(r+s)$. Calculate
\begin{align*}{\omega}_0(\alpha,t) =&\;\;
\omega + t\beta + \psi i\partial\bar{\partial}u_\alpha + [\psi_z u_{1,\bar{w}} \, i dz \wedge d\bar{w}
+ \psi_{\bar{z}} u_{1,w} \, i dw \wedge d\bar{z} \,+ \\
&\;+\{\psi_{z\bar{z}}\tilde{u}_\alpha +
\psi_{\bar{z}}\tilde{u}_{\alpha,z} + \psi_z \tilde{u}_{\alpha,\bar{z}}\}\, i dz \wedge d\bar{z}].\end{align*}
Since $\omega + i\partial\bar{\partial}u_\alpha = T^*{\omega}_{\rm sf}(\alpha) > 0$ over $\Delta^*$, we have 
$\omega + \psi i \partial\bar{\partial}u_\alpha > 0$ by convexity.
To compensate for the $dz \wedge d\bar{w}$ and $dw \wedge d\bar{z}$ terms in square brackets,
notice
$$(z,w)\begin{pmatrix} 0 & a \\ \bar{a} & 0 \end{pmatrix}\begin{pmatrix}\bar{z}\\ \bar{w}\end{pmatrix} = 2 {\rm Re}(az\bar{w}) \geq -\frac{|a|}{\delta} |z|^2
- \delta |a| |w|^2\;\;(a \in \C, \delta > 0).$$
Apply this by making $\delta$ small (depending on $\sup|\psi_z u_{1,\bar{w}}|$ and $\inf(\omega + \psi i\partial\bar{\partial} u_1)_{w\bar{w}}$
over $\Delta(r+2s)\setminus\Delta(r+s)$) and then $t > C_{0}(r,s)$ with $C_0(r,s)$ sufficiently large.
To treat the $dz \wedge d\bar{z}$ terms, consider (\ref{pot2}).
The
$u_1$ part can be controlled by making $C_0(r,s)$ larger.
To deal with the $(\alpha-1)(u-v)$ contribution, use (\ref{bernstein}). \hfill $\Box$\medskip\

Thus, given any $\alpha > 0$, if $t$ is large enough as specified in Claim 3, 
then $\omega_{0}(\alpha,t)$
defines a complete K{\"a}hler metric on $M$ which has all the required properties,
(\ref{properties1}), (\ref{properties2}), 
\emph{except} that the compactly supported volume form 
$\omega_{0}(\alpha,t)^2-\alpha\Omega\wedge\bar{\Omega}$ will 
not have zero mass in general, as needed for the integrability condition.\medskip\

\noindent {\bf Claim 4 (integrability condition).} One can choose $r,s$,
depending on $X,\Omega,\omega,$ $\tau_1,\tau_2$, such that for all $t' > 1$ there exists
precisely one $\alpha > 0$ such that $I(\alpha,t)$ $:=$ 
$\int ({\omega}_{0}(\alpha,t)^2 - \alpha\Omega \wedge \bar{\Omega}) = 0$ 
with $t := C_0(r,s)t' + C_0|\alpha-1|\sup_{\Delta(r+3s)\setminus \Delta(r)} u_{z\bar{z}}$.\medskip\

The key point to note here is that $I(\alpha,t)$ is \emph{linear} in both variables, from (\ref{properties1}).
Thus, Proposition \ref{extend} follows immediately from what is established in Claim 4.\medskip\

\noindent \emph{Proof.} Let $r,s > 0$ be arbitrary for now but with $C_0 r < 1$, $C_0 s < r$. For
any given $t' > 1$ and $\alpha > 0$, define $t$ as in the Claim, where $C_0(r,s)$ is as least as large as the generic
constant from Claim 3. We will show that $r,s$ can be chosen in such a way, but 
depending only on $X,\Omega,\omega,\tau_1,\tau_2$, that 
$I(\alpha,t) > 0$ if $\alpha$ is small and $I(\alpha,t) < 0$ if $\alpha$ is large, depending on $t'$ and all the other data.
Indeed, from (\ref{positive}),
$$\int_{X|_{\P^1 \setminus \Delta(r)}} {\omega}_{0}(\alpha,t)^2 \geq \frac{1}{4}\int_{X|_{\P^1 \setminus \Delta(r + 2s)}} \omega^2,$$
and hence $I(\alpha,t) > 0$ if $\alpha$ is sufficiently small, depending on $\Omega,\omega$ and $r,s$. 
For the other direction, fix $t' > 1$ and let $\alpha > 1$. Arguing as in the proof of Claim 3 then,
${\omega}_{0}(\alpha,t)$ $\leq$ $2(\omega + \psi i \partial\bar{\partial}u_1) + t\beta$ 
over $\P^1 \setminus \Delta(r)$. From (\ref{onthebase}) and Table 4.1,
$$\sup_{\Delta(r+3s)\setminus\Delta(r)} u_{z\bar{z}} \leq C_0 
\inf_{\Delta(r + 3s)\setminus \Delta(r)} u_{z\bar{z}}.$$
Thus, denoting by $\chi$ the characteristic function of $\Delta(r + 3s)\setminus\Delta(r)$ as well as its pull-back under
$f$, we find that over $\P^1\setminus\Delta(r)$,
$$\omega_0(\alpha,t) \leq 2(\omega + \psi i\partial\bar{\partial}u_1) 
+ C_0(r,s)t'\beta + C_0(\alpha-1)\chi i\partial\bar{\partial}{u}.$$
By (\ref{onthebase}) and Table 4.1, $\beta \leq C_0 i\partial\bar{\partial}u$.  
Assuming $\alpha - 1 > C_0(r,s)t'$, by (\ref{pot1}), 
$$I(\alpha,t) \leq 4\int_{X|_{\P^1\setminus\Delta(r + 3s)}} \omega^2 + C_0 \int_{X|_{\Delta(r + 3s)\setminus\Delta(r)}} T^*\omega_{\rm sf}(\alpha)^2 
- \int_{X|_{\P^1\setminus\Delta(r)}} \alpha\Omega\wedge \bar{\Omega}.$$
Since $T^*\omega_{\rm sf}(\alpha)^2 = \Omega$, choosing first $r$ and then $s$ in
such a way that
$$\int_{X|_{\P^1 \setminus \Delta(r)}} \Omega \wedge \bar{\Omega} > 4 \int_M \omega^2 + 
C_0 \int_{X|_{\Delta(r+3s)\setminus\Delta(r)}} \Omega \wedge \bar{\Omega},$$
we can achieve $I(\alpha,t) < 0$ as desired. \hfill $\Box$

\section{Uniqueness and moduli in the cylindrical case}

This section studies uniqueness and deformation questions for hyperk{\"a}hler
ALH metrics on open $4$-manifolds $M$ which are diffeomorphic to complements of 
smooth fibers $D$ in rational elliptic fibrations $f: X \to \P^1$. This relies
on a computation of the topological invariants of $M$ and on Hodge theory for
asymptotically cylindrical manifolds, which we work out and summarize as needed in Section 5.1.
Section 5.2 contains the proof of our partial uniqueness result, Theorem \ref{alh}.
The deformation space of an ALH hyperk{\"a}hler metric is studied in Section 5.3, clarifying 
some issues that appear in applying Kovalev \cite{kovalev2} to this particular setting.

\subsection{Topology and analysis on ALH spaces} $H^2(X) = H^{1,1}(X) = \R^{10}$ and $M$ is simply-connected, hence
the Mayer-Vietoris sequence
$$0 \to H^1(T^2) \to H^1(T^3) \to H^2(X) \to H^2(M) \oplus H^2(T^2) 
\to H^2(T^3) \to 0$$
yields $H^2(M) = \R^{11}$. More precisely, since the map $H^2(X) \to H^2(M) \oplus H^2(T^2)$ has a $9$-dimensional image 
and $1$-dimensional kernel (spanned by the Poincar{\'e} dual of a fiber,
i.e.~$c_1(X)$), and since the restriction $H^2(X) \to H^2(M)$ maps into 
$H^{1,1}(M)$, the subspace of all de Rham classes
representable by closed forms which are $(1,1)$ for the complex structure on $X$, we find
$\dim H^{1,1}(M) \geq 9$.
However, for a meromorphic $2$-form $\Omega$ with ${\rm div}(\Omega) = -D$, 
$[{\rm Re}(\Omega)]$ and $[{\rm Im}(\Omega)]$ are linearly 
independent and not contained in $H^{1,1}(M)$,
so $H^{2}(X) \to H^{1,1}(M)$ is onto with kernel spanned by $c_1(X)$. 
In particular, all closed $(1,1)$-forms on $M$ pair to zero
with all $2$-cycles
in $M$ which are boundaries in $X$, such as the bad cycles from Definition \ref{bad}.
 
The cohomology long exact sequence for $(X\setminus U, \partial U)$ for
a tubular neighborhood $U$ of $D$ gives
$H^2_c(M) = \R^{11}$ and ${\rm im}(H^2_c(M) \to H^2(M)) = \R^8$. The classes in the image are precisely the ones
that can be represented by compactly supported forms.
If $g$ is any ALH Riemannian
metric on $M$, then the de Rham cohomology class of every $L^2$-harmonic $2$-form with respect to $g$ is of this type,
and conversely all such classes are uniquely representable by $L^2$-harmonic forms.
(For this last fact, and for all other standard results from Hodge theory on cylindrical manifolds 
that appear in the following,
we refer to the nice overview in Nordstr{\"o}m \cite[Sections 5.1--5.2]{nordstrom}.)
If $g$ is K{\"a}hler with respect to the complex structure on $X$ and if ${\rm Ric}(g) \geq 0$, 
we see that $H^{1,1}(M)$ splits as the direct sum of $\R[\omega]$ and the space 
of $L^2$-harmonic forms, 
all of which are primitive $(1,1)$ since $\wedge \, \omega$ preserves
harmonicity and ED harmonic $(2,0)$-forms are trivial by the Bochner formula. 

Back to a general ALH metric, the space of \emph{bounded} $\Delta_d$-harmonic $r$-forms with respect to $g$
has dimension
$\dim H^r(M) + \dim H^r_c(M) - \dim {\rm im}(H^r_c(M) \to H^r(M))$. 
Outside a compact set, 
all such forms split into one part which is parallel for $g$ and another one which is exponentially
decaying (ED) as in Definition \ref{ed}. Therefore,
the space of bounded $\Delta_d$-harmonic $2$-forms
splits (not uniquely) as the direct sum of its 
$8$-dimensional $L^2$-harmonic subspace, and a $6$-dimensional subspace of 
forms which are asymptotic to
$\alpha \wedge dt + \beta$ ($\alpha \in \wedge^1\R^3$, $\beta \in \wedge^2\R^3$) with respect to
any ALH coordinate system $\Phi$ as in Definition \ref{alhdef}(ii). Finally, we will need to know that
$\Delta_d$ is Fredholm of index $-(b_r(T^3) + b_{r-1}(T^3))$ 
when acting on ED $r$-forms with a small enough decay rate, some suitable
weighted H{\"o}lder topology to be understood.

\begin{corollary}\label{alhsolve} Let $\beta$ be an {\rm ED} $r$-form on $M$ with $r \in \{0,1,2\}$.

{\rm (i)} If $r = 0$, then for any fixed {\rm ALH} coordinate system there exists 
a unique $\alpha$ $=$ $Ct$ $+$ ${\rm ED}$ such that $\Delta_d\alpha = \beta$,
and moreover $C = 0$ if and only if $\int \beta \, d\vol = 0$.

{\rm (ii)} If $r = 1$, there exists a unique $\alpha$ $=$ parallel $+$ ${\rm ED}$ with $\Delta_d\alpha = \beta$. The parallel part of $\alpha$ 
vanishes if and only if $\int \langle \beta,\eta\rangle \, d\vol = 0$ for all $\Delta_d$-harmonic
$1$-forms $\eta$ of linear growth. The space
of all such harmonic $1$-forms has dimension four.

{\rm (iii)} If $r = 2$, then a smooth and in addition ${\rm ED}$ $2$-form $\alpha$ with $\Delta_d\alpha = \beta$ exists if and only
if $\int \langle \beta,\eta\rangle \, d\vol = 0$ for all bounded $\Delta_d$-harmonic $2$-forms $\eta$.
\end{corollary} 

\begin{proof} If $r = 0$, the condition $\int \beta = 0$ is clearly necessary for an ED solution $\alpha$ to exist,
and thus sufficient because the index is $-1$ and the kernel is trivial. All $\alpha_0$ with $\alpha_0 = t + {\rm ED}$ outside 
a compact set satisfy $\Delta_d\alpha_0 = {\rm ED}$, but $\Delta_d \alpha_0 = 0$ implies
$d\alpha_0$ is bounded $\Delta_d$-harmonic and hence still $\alpha_0 = 0$,
and so $\Delta_d: {\rm ED} \oplus \R \alpha_0 \to {\rm ED}$ is onto. For $r = 1$,
note that the index on ED $1$-forms is $-4$, and we can obviously find four linearly independent $1$-forms
which are parallel + ED, while, again, there are no bounded $\Delta_d$-harmonic $1$-forms at all. 
This in turn proves that we can take any $1$-form which is equal to $t \xi + {\rm ED}$ outside a compact set for 
some constant $\xi$, and correct it by a bounded form so as to become $\Delta_d$-harmonic, providing us with a 
$4$-dimensional space of linearly growing harmonic $1$-forms, which must then define the cokernel of $\Delta_d$ on
ED $1$-forms because the index is $-4$.
If $r = 2$, the given condition is necessary, and so $\dim {\rm coker} \geq 14$, but also
${\rm index} = -6$ and $\dim {\rm ker} = 8$.
\end{proof}

The key fact here is that while on a \emph{general} asymptotically cylindrical manifold, the space of bounded harmonic $r$-forms mod ED
is of dimension $\leq b_{r}(Y) + b_{r-1}(Y)$ if $Y$ denotes the cross-section, in this particular case, we have equality for $r = 2$.

\subsection{Proof of Theorem \ref{alh}} To prove Part (i), define $\omega := \Psi^*\omega_2-\omega_1$. From our 
assumptions, $\omega = {\rm ED}$ on $M_1$. By Corollary \ref{alhsolve}(ii), there exists a bounded $1$-form
$\alpha = {\rm parallel} + {\rm ED}$ such that $\Delta_d\alpha = d^*\omega$. Since $\omega$ is exact by assumption,
the form $d\alpha-\omega$ is $\Delta_d$-harmonic, ED, and exact, and hence $d\alpha = \omega$ by the Hodge theorem. 
Now
$\bar{\partial}^*{\alpha}^{0,1}$ is an ED function. Thus, by Corollary \ref{alhsolve}(i) and the K{\"a}hler identities, there exists a function $\gamma$ of at most linear growth such that $\bar{\partial}^*\bar{\partial}\gamma = \bar{\partial}^*\alpha^{0,1}$. Since
$\bar{\partial}\alpha^{0,1} = 0$ for type reasons, 
$\bar{\partial}\gamma - {\alpha}^{0,1}$ is a bounded $\Delta_d$-harmonic $1$-form, thus zero, so
eventually $\omega = i\partial\bar{\partial}u$ with $u := 2\, {\rm Im}(\gamma)$. Now
$\Psi^*\omega_2^2 = h \omega_1^2$
with
 $h = 1 + {\rm ED}$, but $h$ must
be constant because $\omega_1$ and
$\Psi^*\omega_2$ are globally hyperk{\"a}hler and $\Psi^*J_2 = J_1$,
and hence $\Psi^*\omega_2^2 = \omega_1^2$. With $\omega = i\partial\bar{\partial} u$, this 
yields $\Delta u = 0$, where the Laplacian is taken with respect to the metric $g_1 + \Psi^*g_2$. This metric is still ALH,
so $u = const$ from Section 4.1, and thus $\omega = 0$ as desired.\medskip\

Part (ii) is completely elementary. There
exist $g_{\rm flat}$-parallel and $g_{\rm flat}$-orthogonal 
almost-complex structures $J_{i,0}$ on $\R^+ \times S^1 \times T^2$ such that
$\Phi_i^*J_i = J_{i,0} + {\rm ED}$, where
$g_{\rm flat} = dt^2 \oplus \ell^2 d\phi^2 \oplus g_{\epsilon,\tau}$ and $\phi \in S^1 = \R/2\pi\Z$
as in Definition \ref{alhdef}(ii). To see this, define a family
$J_{i,0}(T)$, $T \in \R^+$, of 
such almost-complex structures by
setting 
$$J_{i,0}(T)|_{(T,\ast)} := P[(\Phi_i^*J_i)|_{(T,\ast)}],$$
where $P$ denotes nearest-neighbor projection onto $g_{\rm flat}$-orthogonal almost-complex structures
and $\ast$ is some base point in $S^1 \times T^2$,
and then by $g_{\rm flat}$-parallel transport, noting
that flat cylinders $\R \times T^3$ are hyperk{\"a}hler so there are no holonomy issues.
From ODE estimates, $|J_{i,0}(T) - \Phi_i^*J_i|_{g_{\rm flat}} = O(e^{-\delta t} + e^{-\delta T})$ on the
whole cylinder, uniformly as $T \to \infty$. Thus $J_{i,0}(T)$ limits out to some $J_{i,0}$
as desired.

Since $(M_i, J_i)$ is $\Phi_i$-compactifiable,
$\mu^*J_{i,0}$ must
extend continuously 
to $\Delta \times T^2$.
This already forces $J_{i,0} = \pm J_\ell \oplus \pm J_{\tau}$,
where $J_\ell(\partial_t) = \ell^{-1}\partial_\phi$ and $J_\tau$ is one of the two 
$g_{\epsilon,\tau}$-parallel, 
$g_{\epsilon,\tau}$-orthogonal
almost-complex structures on $T^{\textup{\begin{tiny}$2$\end{tiny}}}$. 
Since the map $\Pi := \mu^{\textup{\begin{tiny}$-1$\end{tiny}}} \circ \Phi_{\textup{\begin{tiny}$2$\end{tiny}}}^{\textup{\begin{tiny}$-1$\end{tiny}}} \circ \Psi \circ \Phi_{\textup{\begin{tiny}$1$\end{tiny}}} \circ \mu$ 
extends as a diffeomorphism from $\Delta \times T^{\textup{\begin{tiny}$2$\end{tiny}}}$ to itself, 
preserving $\{0\} \times T^{\textup{\begin{tiny}$2$\end{tiny}}}$, 
$\Pi$ must pull back $\mu^*J_{2,0}$ to $\mu^*J_{1,0}$ along $\{0\} \times T^{\textup{\begin{tiny}$2$\end{tiny}}}$. 
Therefore,
$\Pi^*(\pm J_\C \oplus \pm J_{\tau}) = \pm J_\C \oplus \pm J_{\tau}$ 
along $\{0\} \times T^{\textup{\begin{tiny}$2$\end{tiny}}}$ for some combination
of the signs. We can make 
all signs positive by composing with isometries of $g_{\rm flat}$. 
Thus, in complex coordinates $(z,w)$ on $\Delta \times T^{\textup{\begin{tiny}$2$\end{tiny}}}$, 
with $w$ uniformizing $(T^{\textup{\begin{tiny}$2$\end{tiny}}}, J_\tau)$,
$\Pi_* \in {\rm Gl}(2,\C)$ along $\{0\} \times T^2$, more precisely $\Pi_* = (\begin{smallmatrix}a & 0 \\ \ast & b \end{smallmatrix})$ 
with $|b| = 1$. Hence, by Taylor expansion, 
$$\Pi(z,w) = (az, bw + c) + (O(|z|^2),O(|z|)),\;c \in \C,$$
thus $\Pi^*(\mu^*g_{\rm flat}) = (\mu^*g_{\rm flat})(1 + O(|z|))$ and  
so $\Psi^*g_2 - g_1 = {\rm ED}$.\hfill $\Box$

\subsection{Proof of Theorem \ref{alh2}} Koiso \cite{koiso} showed that the kernel of the linearized Einstein operator
on a compact Calabi-Yau manifold, modulo Lie derivatives of the metric, is invariant under the action of the almost-complex structure $J$. Thus, the kernel
contains its hermitian and skew-hermitian components, which can be viewed as spaces of first-order deformations of the K{\"a}hler class, and of $J$, respectively.
If $J$ has unobstructed deformations (which is always true, by later work of Tian \cite{tiatod} and Todorov \cite{tod}),
Koiso then gave a method to integrate all Ricci-flat deformations
of the metric by decomposing them into their hermitian and skew-hermitian parts
and using the implicit function theorem
to solve a family of Monge-Amp{\`e}re equations
on the resulting 
family of compact K{\"a}hler manifolds. 
Kovalev \cite{kovalev2} has sketched an 
extension of Koiso's theory to the setting
of open Calabi-Yau manifolds asymptotic
to a complex cylinder $\C^* \times D$ with metric $|d\log z|^2 \oplus g_D$,
where $(D, g_D)$ is compact Calabi-Yau with $\pi_1(D) = 0$. The condition on $\pi_1$ implies that 
the complex splitting at infinity is preserved under bounded Ricci-flat deformations. We will
develop this theory carefully on ALH spaces, in which case $D$ is a torus
and so $\pi_1(D) \neq 0$. This ultimately leads to various complications when counting moduli.

In \emph{Step 1}, we study deformations of the flat metric on $\R \times T^3$ from Koiso's point of view, viewing $\R \times T^3 = \C^2/\Lambda$ for some rank-$3$ lattice $\Lambda \subset \C^2$.

In \emph{Step 2}, we show that on ALH hyperk{\"a}hler manifolds diffeomorphic to $X \setminus D$,
the kernel $\mathcal{E} \subset \mathcal{H}$ of the linearized Einstein operator contains a harmonic space $\mathcal{D}$
of dimension $34$  (the kernel of the Lichnerowicz Laplacian) such that $\mathcal{E} = \mathcal{D} \oplus \mathcal{L}_{\rm bd}$,
where $\mathcal{L}_{\rm bd} := \{L_X g: X$ $=$ parallel $+$ ED$\}$. Except for the precise dimension count, this follows from work in \cite{kovalev2}.

In \emph{Step 3}, we describe the space of Lie derivatives that are still contained in $\mathcal{D}$. 
This turns out to be the space $\mathcal{L}_{\rm lin} := \{L_X g:$ $X$ linearly growing harmonic$\}$,
which reduces the number of moduli from $34$ to $30$, still counting the multiples of $g$.

In \emph{Step 4}, we obtain integrability of all first-order Ricci-flat deformations if the 
hyperk{\"a}hler ALH space to be deformed is one of those constructed in Tian-Yau \cite{ty1}.
Let $^\pm$ denote the projections onto $J$-(skew-)hermitian forms, where $J$ is the unique 
parallel orthogonal almost-complex structure on $M$ 
that can be compactified. Then $\mathcal{D}^\pm \subset \mathcal{D}$, 
${\rm dim} \,\mathcal{D}^+ = 12$, $\dim \,\mathcal{D}^- = 22$, $\dim \mathcal{L}_{\rm lin} = 4$, $\mathcal{L}_{\rm lin} \cap \mathcal{D}^+ = 0$, $\mathcal{L}_{\rm lin} \cap \mathcal{D}^- = \R$, so $\mathcal{L}_{\rm lin}^-$ is of codimension $18$ in $\mathcal{D}^-$. 
Luckily, $(M,J)$ admits obvious complex moduli 
from deforming the pencil or moving $D$, and we will see that these yield a subspace $\mathcal{S} \subset \mathcal{D}^-$ 
with $\dim_\C \mathcal{S} = 9$ and $\mathcal{S} \cap \mathcal{L}_{\rm lin}^- = 0$. 
Every
element of $\mathcal{E}$ 
thus differs from a unique element of $\mathcal{D}^+ \oplus \mathcal{S}$ by a Lie derivative. All elements of $\mathcal{S}$ are integrable as deformations of $J$,
and it is fairly easy to see that the $\mathcal{S}$-deformations of $(M,J)$ still admit
ALH K{\"a}hler metrics.
This then suffices for Koiso-Kovalev to go through.\medskip\ 

\noindent \emph{Step 1: Deformations  of flat cylinders.} Equip $\C^2$ with its standard K{\"a}hler structure $(g,J)$ and 
let $\Lambda \subset \C^2$ be a rank-$3$ lattice. Then $\C^2/\Lambda = \R \times T^3$ isometrically for the flat metric on  $T^3 = \R^3/\Z^3$ induced by any 
marking $\Z^3 \to \Lambda$. Thus, the deformation 
space of the flat metric on $\C^2/\Lambda$, including a scale, is naturally identified with the
$6$-dimensional space ${\rm Gl}(3,\R)/{\rm O}(3)$ of marked flat metrics on $T^3$. 
We are now going to compare this with Koiso's general approach \cite{koiso}.

The natural space of
infinitesimal Ricci-flat deformations here is $\mathcal{D} = {\rm Sym}^{2}\R^4$. We decompose
$\mathcal{D}$ into its $J$-hermitian and $J$-skew-hermitian parts, $\mathcal{D} = \mathcal{D}^+ \oplus \mathcal{D}^- $, 
and, writing $h(u,v) = g(Hu,v)$, $H = H^{\rm tr}$, for $h \in \mathcal{D}$, we introduce a linear map \begin{equation}\label{i}\mathfrak{I}: \mathcal{D} \to \R^{4\times 4}, \, \mathfrak{I}(h) := H^+J + H^-.\end{equation}
This defines an isomorphism of U$(2)$-modules 
between $\mathcal{D}$ and $\mathfrak{sp}(4,\R)$:
\begin{align*}\begin{split}\mathfrak{I}(\mathcal{D}^+)  &= \{I \in \R^{4\times 4}: I^{\rm tr} + I =  IJ - JI = 0\} = \mathfrak{u}(2),\\
\mathfrak{I}(\mathcal{D}^-) &= \{I \in \R^{4\times 4}: I^{\rm tr} - I =  IJ + JI = 0\}\cong \mathfrak{sp}(4,\R)/\mathfrak{u}(2).\end{split}\end{align*}
We can understand $\mathfrak{I}(\mathcal{D}^-)$ as a space of infinitesimal deformations of $J$. In general, if
$I \in \R^{4 \times 4}$ and $IJ + JI = 0$, there is an associated complex structure $\mathfrak{C}(I)$ on $\R^4$ defined by
having its $(0,1)$-space in $\R^4 \otimes \C$ equal to $\{u'' + (Iu)': u \in \R^4\}$,
where $u' = u - iJu$ and $u'' = u + iJu$. This is the standard
way of exponentiating first-order complex deformations in Kodaira-Spencer theory; when applied pointwise on a complex manifold, 
the integrability of $\mathfrak{C}(I)$ as an almost-complex structure then translates into the Maurer-Cartan equation $\bar{\partial}I + [I,I] = 0$.
Explicitly,
\begin{equation}\label{cayley}\mathfrak{C}(I) = J(1 - I)(1 + I)^{-1}.\end{equation}
If in addition $I^{\rm tr} = I$, or equivalently $I \in \mathfrak{I}(\mathcal{D}^-)$, then $\omega(u,v) = g(Ju,v)$ is still of type $(1,1)$ 
with respect to $\mathfrak{C}(I)$, and $\mathfrak{C}$ identifies its domain of definition in $\mathfrak{I}(\mathcal{D}^-)$, which is a bounded domain, with 
the Siegel upper half plane $\mathfrak{H}_2 = {\rm Sp}(4,\R)/{\rm U}(2) = \{K \in \R^{4 \times 4}: K^2 = -1, K^*\omega = \omega\}$. In fact,
$\mathfrak{C}$ is the generalized Cayley transform. On the other hand, $\mathfrak{C}$ is globally defined on the
space of $I \in \R^{4\times 4}$ with $IJ + JI = 0$ and $I^{\rm tr} + I = 0$, which is a complex line (the infinitesimal twistor line of $\R^4$),
and maps this to the complement of $-J$ in the actual twistor $\P^1$.

(i) $\mathcal{D}$ has dimension $10$ and hence overcounts the number of moduli of the flat
metric $g$ on $\C^2/\Lambda$ by four. This is because $\mathcal{D}$ contains a four-dimensional space $\mathcal{L}_{\rm lin}$
of Lie derivatives of $g$  by linear vector fields:
On $\C^2$, for any $h \in \mathcal{D}$, $h = L_Xg$ with 
$X|_u = (1/2)Hu$, $H$ as above, and $X$ 
descends to $\C^2/\Lambda$ precisely if $H(\Lambda) = 0$.

(ii) Suppose the cross-section $T^3 = S^1 \times T^2$ as an isometric product.
Then
$$\Lambda = \Z (2\pi \ell e_2) +  \Lambda',\;\Lambda' = \frac{\epsilon}{{\rm Im}(\tau)}(\Z + \Z\tau) \subset \C = \R e_3 \oplus \R e_4,$$
with $\epsilon, \ell > 0$ and $\tau \in \mathfrak{H}$, after rotating the lattice if necessary. In this case, $\mathcal{L}_{\rm lin} \subset \mathcal{D}$
has a canonical basis $h_i = L_{x^1 e_i}g$ for $1 \leq i \leq 4$, and $h_2 \in \mathcal{D}^-$ while ${\rm span}\{h_1,h_3,h_4\}$ projects isomorphically into both $\mathcal{D}^+$ and $\mathcal{D}^-$.
Hence, if we quotient by $\mathcal{L}_{\rm lin}$, 
there remain two skew-hermitian deformations dual to changing the complex modulus $\tau$, and four hermitian ones. One of these is an overall scale, one is dual to changing
$\ell^2/\epsilon$,
and the remaining two break the isometric splitting of the cross-section. \hfill $\Box$

\begin{remark}\label{comprem} Step 1(ii) shows that if $\Lambda$ isometrically splits off a $\Z$-factor, then even though the metric on $\C^2/\Lambda$ 
can be deformed so that its cross-section no longer splits, the complex structure on $\C^2/\Lambda$ nevertheless remains compactifiable as $\P^1 \times (\C/\Lambda')$ after hyperk{\"a}hler rotation.
Said differently, we can transform any $\Lambda$ into split form by hyperk{\"a}hler rotation and a $\C$-linear map, but this may twist the metric. \hfill $\Box$
\end{remark}

\noindent \emph{Step 2: The kernel of the linearized Einstein operator.} Let $g$ be ALH hyperk{\"a}hler on a manifold $M$ diffeomorphic
to the complement of a smooth fiber in a rational 
elliptic surface, and recall $\mathcal{H} = \{h \in C^\infty({\rm Sym}^2T^*M): 
h$ $=$ parallel $+$ ED outside a compact set$\}$. The linearization $E$ of the Einstein operator at $g$ is given by
$$\begin{cases}E(h) = \frac{1}{2}(\nabla^* \nabla h - 2 R(h) + L_{D(h)}g), \smallskip\ \\ 

R(h)(X,Y) := \sum h(R(E_i,X)Y,E_i),\smallskip\ \\ 

\langle D(h), X\rangle := \sum \nabla_{E_i}(h - \frac{1}{2}({\rm tr}\,h)g)(E_i, X).\end{cases}$$
Our goal is to describe the space $\mathcal{E} := \mathcal{H} \cap {\rm ker} \,E$ as explicitly as possible. For this, 
we first of all observe that  $D(L_X g) = - \nabla^*\nabla X + {\rm Ric}\,X$ for \emph{all} Riemannian metrics $g$ and vector fields $X$. 
Thus, from Corollary \ref{alhsolve},  any given $h \in \mathcal{H}$ can be corrected by a Lie derivative $L_X g$ with $X$ $=$ parallel $+$ ED outside a compact set
such that $D(h + L_X g) = 0$. Denote the space of all such Lie derivatives by $\mathcal{L}_{\rm bd}$. We conclude
$\mathcal{E} = (\mathcal{D} \cap \ker D) \oplus \mathcal{L}_{\rm bd}$, with $\mathcal{D} := \{h \in \mathcal{H}: \nabla^*\nabla h - 2 R(h) = 0\}$, and the two spaces
are transversal since otherwise $(M,g)$ would admit a nontrivial parallel vector field. The space $\mathcal{D}$ is the counterpart over $(M,g)$
to ${\rm Sym}^2\R^4$ over $\C^2/\Lambda$ as in Step 1.

Fix a parallel orthogonal almost-complex structure $J$. Then $\mathcal{D} = \mathcal{D}^+ \oplus \mathcal{D}^-$, and there are Bochner formulas showing that $\mathfrak{I}(\mathcal{D}^+)$ and $\mathfrak{I}(\mathcal{D}^-)$ (recall (\ref{i})) are exactly the spaces of 
bounded $\Delta_d$-harmonic real $(1,1)$-forms and $\Delta_{\bar{\partial}}$-harmonic symmetric $T^{1,0}M$-valued $(0,1)$-forms, respectively.
As $\Delta_{\bar{\partial}}\alpha = 0 \Leftrightarrow \bar{\partial}\alpha = \bar{\partial}^*\alpha = 0$,
$\alpha \in \mathfrak{I}(\mathcal{H}^-)$, we find that $\mathfrak{I}(\mathcal{D}^-)$ consists of first-order deformations of $J$ as an
integrable almost-complex structure, i.e.~if $I \in \mathfrak{I}(\mathcal{D}^-)$, then all curves $I_t$ with $I_0 = 0$, $I_t'|_0 = I$,
satisfy 
$\bar{\partial}I_t + [I_t,I_t] = 0$ to first order in $t$, cf.~Step 1.
Note that $\Delta_{\bar{\partial}}$-harmonic
$T^{1,0}M$-valued $(0,1)$-forms are in bijection with $\Delta_d$-harmonic $\C$-valued $(1,1)$-forms, from dualizing 
with a holomorphic symplectic form. Thus
$\mathcal{D} \cap \ker D = \mathcal{D}$ because wedging a scalar
form with $\omega$ preserves harmonicity,
so ${\rm tr}\,h$ is harmonic and hence constant for any
$h \in \mathcal{D}$, and ${\rm div}\,h = 0$ 
since bounded harmonic forms are $d^*$-closed.

In conclusion, 
$\mathcal{E} = \mathcal{D} \oplus \mathcal{L}_{\rm bd}$. Also, $\dim \mathcal{D} = (1+2)\cdot 8 + 10 = 34$. This is because by Section 5.1, the space of ED $\Delta_d$-harmonic real
$(1,1)$-forms is $8$-dimensional, and the space of bounded
{$\Delta_d$-}harmonic $2$-forms modulo ED is isomorphic to the space of constant $2$-forms on the flat cylinder $\C^2/\Lambda$ to which $(M,g)$ asymptotes,
through the map $P$ which sends any tensor to its 
constant part outside a compact set; now
$P$ preserves exterior form types with respect to $J$ and $P(J)$, 
and maps the twistor line $\C\Omega$ of $(M,g,J)$ to that of $\C^2/\Lambda$, so
$(P \circ I)(\mathcal{D}) = {\rm Sym}^2\R^4 = \R^{10}$.\hfill $\Box$\medskip\

\noindent \emph{Step 3: Lie derivatives in the harmonic space.} Let $h = L_X g \in \mathcal{D}$, so $\nabla^*\nabla X = 0$.\medskip\

\noindent \emph{Claim:} If $h = {\rm ED}$, then $|\nabla X|$ is uniformly bounded on $M$. \medskip\

Assuming this claim, $X$ is harmonic of at most linear growth. Then the proof of Corollary
\ref{alhsolve}(ii) shows $X = tX_0 + {\rm parallel} + {\rm ED}$ outside a compact set, where $X_0$
is a constant vector field 
on $\C^2/\Lambda$. But then $h = {\rm ED}$ implies $tX_0$ is Killing for the flat metric, 
and so $X_0 = 0$, which implies $X \in \mathcal{D} \cap \mathcal{L}_{\rm bd} = \{0\}$ and hence $h = 0$.

Thus, if $h \neq 0$, then $h = h_0 + {\rm ED}$ outside a compact set
with $h_0 \neq 0$ a constant bilinear form on $\C^2/\Lambda$. Moreover, $h_0 \in \mathcal{L}_{\rm lin}$
as defined in Step 1(i) because otherwise $h$ would change the cross-sectional geometry of $(M,g)$ even though $h = L_X g$.

Conversely, given
any $h_0 \in \mathcal{L}_{\rm lin}$, there
exists a unique $h = h_0 + {\rm ED} = L_X g \in \mathcal{D}$. Here, \emph{uniqueness} holds 
by the above,
and \emph{existence} holds because if $h_0 = L_{X_0}g_{\rm flat}$ with $X_0$ linear
on $\C^2/\Lambda$, then from Corollary \ref{alhsolve}(ii) 
there exists a harmonic vector field $X$ on $M$ with
$X - X_0$ $=$ ${\rm parallel}$ $+$ ${\rm ED}$ outside a compact set,
and then $h$ $:=$ $L_X g$ $=$ $h_0 + {\rm ED}$ satisfies 
$E(h) = D(h) = 0$, so that $h \in \mathcal{D}$.  

Thus, the subspace of $\mathcal{D}$ consisting of Lie derivatives of $g$ is given by $\{L_Xg:$ $X$ linearly growing harmonic$\}$. We denote this space by $\mathcal{L}_{\rm lin}$, abusing notation. \medskip\

\noindent \emph{Proof of the claim.} Consider $[0,\infty) \times \R^3$. Write $x^\mu$,
$\mu \in \{0,1,2,3\}$, for the obvious coordinates, and use Latin superscripts $i \in \{1,2,3\}$ to address the $\R^3$ coordinates. Let $u^\mu$, $f^{\mu\nu}_\lambda$, $g^{\mu\nu}$ be smooth functions on $[0,\infty) \times \R^3$ which are $\Gamma$-periodic in $\R^3$ for some lattice $\Gamma$, and suppose that the $f^{\mu\nu}_\lambda$ and $g^{\mu\nu}$ decay exponentially, i.e.
\begin{equation}\label{edaux} \|f^{\mu\nu}_\lambda\|_{C^k(t)} +  \|g^{\mu\nu}\|_{C^k(t)} \leq C(k) e^{-\epsilon t} \end{equation} for some $\epsilon \in (0,1]$, all $k \in \N_0$, all $t \in [0,\infty)$, where notation such as $C^k(t)$ refers to the relevant function spaces on the slice $\{t\} \times \R^3$.\medskip\

\noindent \emph{Reduced claim:} There exists $C = C(f^{\mu\nu}_\lambda, g^{\mu\nu}) \geq 1$ such that if
\begin{equation}\label{euclie} \frac{\partial u^\mu}{\partial x^\nu} + \frac{\partial u^\nu}{\partial x^\mu} + \sum f^{\mu \nu}_\lambda u^\lambda = g^{\mu\nu},\end{equation}
and if $\|u^\mu\|_{C^{0}(T_0)} + \|\partial_\nu u^\mu\|_{C^{0,\alpha}(T_0)} \leq C_0$ and
$\|\partial_\nu u^\mu\|_{C^{0,\alpha}(t)} \leq 3C_0$ for all $t \in [T_0,T]$, 
where $C_0 \geq 1$ and $T \geq T_0 \geq C$,
then $\|\partial_\nu u^\mu\|_{C^{0,\alpha}(t)} \leq 2C_0$ for all $t \in [T_0,T]$. \medskip\

Thus, assuming (\ref{edaux}), every smooth and spatially $\Gamma$-periodic solution $u^\mu$ to (\ref{euclie}) on $[0,\infty) \times \R^3$ must automatically have uniformly bounded derivatives. From this,
the original claim then follows immediately.

To prove the reduced claim, fix $T \geq T_0 \geq 0$, $t \in [T_0,T]$. By (\ref{euclie}), $\mu = \nu = 0$, 
$$\|\partial_t u^0\|_{C^{0,\alpha}(t)} \leq C_0C(t - T_0 + 1)e^{-\epsilon t} \leq C_0 C e^{-\epsilon T_0}.$$ Next, we differentiate (\ref{euclie}), $\mu = \nu = 0$, with respect to $x^j$.
Using the general fact that
$\frac{d}{dt}\|v\|_{C^{0,\alpha}(t)} \leq \|\frac{\partial v}{\partial t}\|_{C^{0,\alpha}(t)}$
and integrating in $t = x^0$, we obtain $$\|\partial_j u^0\|_{C^{0,\alpha}(t)} \leq C_0(1 + Ce^{-\epsilon T_0}).$$
Plugging this into (\ref{euclie}) with $\mu = i$, $\nu = 0$
yields $$\|\partial_t u^i\|_{C^{0,\alpha}(t)} \leq C_0(1 + Ce^{-\epsilon T_0}).$$
Finally, differentiating (\ref{euclie}) for $\mu = i$, $\nu = 0$ by $x^j$ and integrating in $t$,
$$\|\partial_j u^i\|_{C^{0,\alpha}(t)} \leq C_0(1 + Ce^{-\epsilon T_0}) + \int_{T_0}^t \|\partial_i \partial_j u^0 \|_{C^{0,\alpha}(s)}\, ds.$$ By Schauder theory,
$\|\partial_i \partial_j u^0 \|_{C^{0,\alpha}(s)} \leq C \|\Delta_{\R^3} u^0\|_{C^{0,\alpha}(s)}$, where $C = C(\Gamma,\alpha)$, and
$$\Delta_{\R^3} u^0 = \sum (\partial_i g^{0i} - \frac{1}{2}\partial_t g^{ii}) - \sum (\partial_i[f^{0i}_\lambda u^\lambda]-\frac{1}{2}\partial_t [f^{ii}_\lambda u^\lambda])$$
from differentiating and tracing (\ref{euclie}), which shows that $$\|\Delta_{\R^3} u^0 \|_{C^{0,\alpha}(s)} \leq C_0C(s-T_0+1)e^{-\epsilon s}.$$ Altogether then,
$\|\partial_\nu u^\mu\|_{C^{0,\alpha}(t)} \leq C_0(1 + C e^{-\epsilon T_0})$, where $C = C(f^{\mu\nu}_\lambda,g^{\mu\nu})$.\hfill $\Box$\medskip\

\noindent \emph{Step 4: Integrability.} Assume now that $(M,g)$ is one of the hyperk{\"a}hler ALH spaces obtained by Tian-Yau \cite{ty1} and
that the underlying pencil of cubics is generic. Let $J$ be the unique $g$-parallel orthogonal almost-complex structure on $M$ which 
extends as a complex structure on $X$.
By Steps 2--3,  $\mathcal{E} = \mathcal{D} \oplus \mathcal{L}_{\rm bd}$,
$\mathcal{D} = \mathcal{D}^+ \oplus \mathcal{D}^-$, $\dim_\R \mathcal{D}^+$ $=$ $8 + 4 = 12$,
$\dim_\C \mathcal{D}^- = 8 + 3 = 11$ (in both cases, the subspaces of ED tensors
are of dimension $8$, and the remaining dimensions are as on $\C^2/\Lambda$), 
and $\mathcal{D}$ contains 
an exactly four-dimensional space $\mathcal{L}_{\rm lin}$ of Lie derivatives of $g$,
which asymptotes to the corresponding space on $\C^2/\Lambda$. Since $g$ is ALH$(\ell,\epsilon,\tau)$, we can assume that $\Lambda$ is as in Step 1(ii),
and hence we have a canonical basis $h_i$ for $\mathcal{L}_{\rm lin}$ such that $h_2 \in \mathcal{D}^-$ while ${\rm span}\{h_1,h_3,h_4\}$
projects isomorphically into both $\mathcal{D}^+$ and $\mathcal{D}^-$.

To prove integrability of the infinitesimal Einstein deformations in $\mathcal{D}$, we proceed in three stages: 
(i) Construct a space $\mathcal{R}$
of infinitesimal deformations of $J$ which are integrable and compactifiable, such that the corresponding deformed complex structures
all admit holomorphic volume forms and ALH K{\"a}hler metrics,
and such that $\mathcal{R} = \mathfrak{I}(\mathcal{S})$ 
for a $\C$-linear subspace $\mathcal{S} \subset \mathcal{D}^-$ with $\dim_\C \mathcal{S} = 9$ and $\mathcal{D}^- = \mathcal{S} \oplus \mathcal{L}_{\rm lin}^-$. Thus, in particular,
$\mathcal{E} = \mathcal{D}^+ \oplus \mathcal{S} \oplus \mathcal{L}_{\rm lin} \oplus \mathcal{L}_{\rm bd}$, and the sum of the first two factors does not contain any Lie derivatives.
(ii) Given any $h \in \mathcal{D}^+ \oplus \mathcal{S}$, set up and solve a family of Monge-Amp{\`e}re equations associated to the complex deformation $h^-$ and the K{\"a}hler deformation $h^+$.
This produces a curve $g_t$ of ALH hyperk{\"a}hler metrics with $g_0 = g$, hence a linear map $\mathfrak{L}: \mathcal{D}^+ \oplus \mathcal{S} \to \mathcal{E}$, $\mathfrak{L}(h) := g_t'|_0$. 
(iii) Let $\mathfrak{P}$ denote the projection of $\mathcal{E}$ onto $\mathcal{D}^+ \oplus \mathcal{S}$ along $\mathcal{L}_{\rm lin} \oplus \mathcal{L}_{\rm bd}$.
Then $\mathfrak{P} \circ \mathfrak{L}$ is an isomorphism of $\mathcal{D}^+ \oplus \mathcal{S}$, though perhaps not equal to the identity.

(i) Write $M = X \setminus D$ such that $J$ extends to $X$, $f: (X,J) \to \P^1$ is the elliptic fibration gotten by blowing up the nine distinct base points of 
a generic pencil of cubics, and $D = f^{-1}(p)$ is a smooth elliptic curve of modulus $\tau \in \mathfrak{H}$. Tian-Yau \cite{ty1} construct the $J$-K{\"a}hler form of the ALH$(\ell,\epsilon,\tau)$ hyperk{\"a}hler metric $g$ as
$$\omega = \omega_X + \ell^2 f^* i\partial\bar{\partial}(\chi (\log |z|)^2) + f^*\omega_{\P^1} 
+ i\partial\bar{\partial} u$$
with K{\"a}hler metrics $\omega_X$, $\omega_{\P^1}$, on $X$, $\P^1$, so that $\omega_X$ restricts to a flat K{\"a}hler 
metric of area $\epsilon$ on $D$, where $z$ is a local coordinate around $p$ on $\P^1$, $\chi$ is a suitable cut-off
function,
and $u$ is a smooth bounded potential of finite Dirichlet energy. (Observe that Tian-Yau 
do not establish ALH$(\ell,\epsilon,\tau)$ asymptotics for $g$, but this 
follows from our work in Section 3.4.)
Let $f_t: (X, J_t) \to \P^1$ be a small deformation of $f$ obtained by varying the pencil $f$ in its moduli space, which is smooth of complex dimension $8$ 
near $f$, and let $D_t = f_t^{-1}(p_t)$ for a curve $p_t$ in $\P^1$ with $p_0 = p$.
By classical work of Kodaira-Spencer \cite{ks},
$\omega_X$ fits into a smooth family of $J_t$-K{\"a}hler forms $\omega_{X,t}$ on $X$.
These may not be exactly flat when restricted to $D_t$,
however $\omega_{X,t}|_{D_t} + i\partial_t\bar{\partial}_tv_t$ will again be flat of area $\epsilon$ for a unique smooth function $v_t$ on $D_t$ such that $\int v_t \omega_{X,t} = 0$,
and $v_t$ extends as a smooth family of functions on $X$ with $v_0 \equiv 0$. Then
$$\omega_t := \omega_{X,t} + i\partial_t\bar{\partial_t}v_t + \ell^2 f_t^*i\partial\bar{\partial}(\chi(\log|z|)^2)  + f_t^*\omega_{\P^1}
+ i\partial_t\bar{\partial}_t u$$ 
yields a family of $J_t$-K{\"a}hler metrics $g_t$ on $M_t = X \setminus D_t$. Let $\Phi: \R^+ \times T^3 \to M \setminus K$ be an ALH coordinate system for $g$.
One can construct diffeomorphisms $\Psi_t: M \to M_t$ 
such that
$\Psi_0 = {\rm id}$, $\Psi_t^*g_t$ is smooth in $t$, 
and $\Phi^*\Psi_t^*g_t = 
(dr^2 \oplus \ell^2 d\phi^2 \oplus h_t) + O(t{\rm ED})$ for some fixed topological splitting $T^3 = S^1 \times T^2$,
where $h_t$ is a smooth family of flat Riemannian metrics on $T^2$ of area $\epsilon$ and modulus dictated by $D_t$. In addition,
one can arrange for $\nabla_{h_s} h_t = 0$ for all $s$, $t$. 
In particular, $\Phi^*\Psi_t^*J_t = (J_\ell \oplus K_t)+ O(t{\rm ED})$ after reflecting $\Phi$ if necessary, where
$J_\ell(\partial_r) = \ell^{-1}\partial_\phi$ and $K_t$ is a smooth family of $h_t$-orthogonal almost-complex structures on $T^2$ with $\nabla_{h_s}K_t = 0$ for all $s,t$.
Writing $\Psi_t^*J_t = \mathfrak{C}(I_t)$ with
$\mathfrak{C}$ from (\ref{cayley}), and then $I_t = tI + O(t^2)$, we obtain 
a $(0,1)$-form $I$ with values in $T^{1,0}M$ such that $\bar{\partial}I = 0$ and $\Phi^*I = (\begin{smallmatrix} 0 & 0 \\ 0 & \ast \end{smallmatrix}) + {\rm ED}$,
where the blocks correspond to $\R^+ \times S^1$ and $T^2$, and $\ast$ means a smooth endomorphism field on $T^2$
which is $h_s$-parallel for any $s$. By Corollary \ref{alhsolve}(ii), there exists a unique vector field $X$ $=$ parallel $+$ ED such that
$I^\circ := I + L_X J$ satisfies $\bar{\partial}I^\circ = \bar{\partial}^*I^\circ = 0$. 

This process can be carried out for all deformations $f_t: (X, J_t) \to \P^1$ and
$D_t$ $=$ $f_t^{-1}(p_t)$. We obtain
a complex vector space $\mathcal{R}$ of $T^{1,0}M$-valued $(0,1)$-forms on $M$,
$\dim_\C \mathcal{R} = 9$, such that $\bar{\partial}I = \bar{\partial}^*I = 0$ for all $I \in \mathcal{R}$. 
In addition, $\Phi^*I = (\begin{smallmatrix} 0 & 0 \\ 0 & \ast \end{smallmatrix}) + {\rm ED}$,
and $I$ is tangent 
to a curve $I_t$ with $I_0 = 0$ such that $J_t := \mathfrak{C}(I_t)$
defines a curve of integrable almost-complex structures.
By construction, $(M,J_t)$ admits holomorphic volume forms $\Omega_t$ and complete K{\"a}hler metrics
$\omega_t$, all of which are ALH with respect to $\Phi$, such that $\nabla_{g_s} g_t$, $\nabla_{g_s} \Omega_t$, $\nabla_{g_s} J_t$ are ED for all $s,t$. 
Also, by direct calculation, if $h \in \mathcal{L}_{\rm lin}^-$, then $\Phi^*h = (\begin{smallmatrix} \ast & \ast \\ \ast & 0 \end{smallmatrix}) + {\rm ED}$, 
and so $\mathcal{D}^- = \mathcal{S} \oplus \mathcal{L}_{\rm lin}^-$ with $\mathcal{S} := \mathfrak{I}^{-1}(\mathcal{R})$.

(ii)  Let $h \in \mathcal{D}^+ \oplus \mathcal{S}$ and split $h = h^+ + h^-$ accordingly. 
Put $I := -\frac{1}{2}\mathfrak{I}(h^-) \in \mathcal{R}$ and let $J_t$, $\Omega_t$, $\omega_t$ be as above.
The form $\eta(X,Y) := g(\mathfrak{I}(h^+)X,Y)$ is $(1,1)$-harmonic with respect to $g$ and $J$ from Step 2. 
Let $\eta_t$ denote the $(1,1)$-component of $\eta$ with respect to $J_t$, and consider the following family of Monge-Amp{\`e}re equations:
\begin{equation}\label{mafamily}(\omega_t + t\eta_t + i\partial_t\bar{\partial}_t u_t)^2 = \alpha_t \Omega_t \wedge \bar{\Omega}_t.\end{equation}
Since $\nabla_{g_s} g_t$, $\nabla_{g_s} \Omega_t$, $\nabla_{g_s} J_t$ are ED for all $s,t$, we can choose $\alpha_t$ in such a way 
that $(\omega_t + t\eta_t)^2 - \alpha_t \Omega_t \wedge \bar{\Omega}_t$ is ED for all $t$. 
Then, as indicated by Kovalev \cite{kovalev2}, because
$u_0 \equiv 0$ is a solution for $t  = 0$,
by Corollary \ref{alhsolve}(i) and the implicit function theorem
there exists a smooth family of solutions $u_t$ of linear growth, so $\omega_t + t\eta_t + i\partial_t\bar{\partial}_t u_t$ defines 
an ALH hyperk{\"a}hler metric $g_t$. We obtain a linear map $\mathfrak{L}: \mathcal{D}^+ \oplus \mathcal{S} \to \mathcal{E}$ by setting 
$\mathfrak{L}(h) := g_t'|_0$, and it remains to show that this is essentially the identity.

(iii) Recall $\mathcal{E} = \mathcal{D}^+ \oplus \mathcal{S} \oplus \mathcal{L}_{\rm lin} \oplus \mathcal{L}_{\rm bd}$, and let
$\mathfrak{P}$ denote the projection from $\mathcal{E}$ onto the first two factors along the last two. 
We will prove that $\mathfrak{P} \circ \mathfrak{L}$ is an isomorphism, which is enough for Theorem \ref{alh2}. Accept for the time being that
$ \mathfrak{L}(h)^- = h^-$ and $ \mathfrak{L}(h^+) = h^+$ for all $h \in \mathcal{D}$. 
Then $( \mathfrak{P} \circ  \mathfrak{L})(h) = 0$ for some $h \in \mathcal{D}^+ \oplus \mathcal{S}$ would imply 
$ \mathfrak{L}(h) = L_X g \in \mathcal{L}_{\rm  lin} \oplus \mathcal{L}_{\rm bd}$, so $h^- = (L_X g)^-$ and $h^+ +  \mathfrak{L}(h^-)^+ = (L_X g)^+$. From
the first equation, since all elements of $\mathcal{L}_{\rm bd}$ are ED and the parallel parts of 
non-trivial elements of $\mathcal{S}$ and $\mathcal{L}_{\rm lin}^-$ are all distinct, we conclude $L_X g \in \mathcal{L}_{\rm bd}$, hence $h^- = 0$
from integration by parts, and so $h^+ = 0$ from the second equation and integration by parts. This shows that $ \mathfrak{P} \circ  \mathfrak{L}$ is injective and therefore an isomorphism.

As for the proof of the two properties of $ \mathfrak{L}$ that we used, notice that 
$ \mathfrak{L}(h)^- = h^-$ simply follows from
$J_t^*g_t = g_t$ and $J_t'|_0 = J  \mathfrak{I}(h^-)$. The second property is the only point where we need to exploit that $\eta$ is harmonic.
Let $h = h^+ \in \mathcal{D}^+$. Then (\ref{mafamily}) becomes
$(\omega + t\eta + i\partial\bar{\partial}u_t)^2 = \alpha_t \Omega \wedge \bar{\Omega}$ with $\alpha_t$ such
that $(\omega + t\eta)^2 - \alpha_t\Omega \wedge \bar{\Omega}$ is ${\rm ED}$.
The claim is trivial for $h$ a constant multiple of $g$, so we can assume that
$\eta \wedge \omega = 0$ because $\eta$ is harmonic and so ${\rm tr}\,h$ is constant. This implies $\alpha_t = \alpha_0 + O(t^2)$, hence
$\alpha_t \Omega \wedge \bar{\Omega} = (1 + O(t^2))(\omega + t\eta)^2$,
so $\Delta(u_t'|_0)= 0$ from differentiating the PDE by $t$, and therefore $u_t'|_0 = 0$, $\omega_t'|_0 = \eta$, $g_t'|_0 = h$, since $u_t'|_0$ is of linear growth. \hfill $\Box$\section{Further questions}

One major open question in $4$-manifold geometry is the structure of Riemannian metrics with bounded Ricci tensor, which even in the Einstein case is not very well understood if collapsing is not prohibited \cite{cct}. Any loss of compactness
should be caused by gravitational instantons with little \textquotedblleft internal structure.\textquotedblright

\begin{problem}\label{bubble} Let $f: X \to \P^1$ be an elliptic fibration on K3. Let $\varpi \in H^2(X,\Z)$
be the Poincar{\'e} dual of the fibers and let $\varpi_t$, $t \in [0,1]$, be a smooth path in $H^2(X,\R)$ with
$\varpi_t$ inside 
the K{\"a}hler cone for $t < 1$ and $\varpi_1 = \varpi$. The collapsed limit metric on $\P^1$ 
of the Ricci-flat metrics in $\varpi_t$ as $t \to 1$
was identified in \cite{gw,song-tian,tos}, see also ${\rm (\ref{st})}$, fitting in with the general picture
that limits of compact Einstein $4$-manifolds with bounded Euler 
numbers are Riemannian 
orbifolds away from a finite number of points \cite{nt1}.
Prove that exactly the following bubbles
form:

$\bullet$ For every singular fiber of finite monodromy, an isotrivial ${\rm ALG}$ space with the corresponding dual 
fiber at infinity, cf.~Example \ref{isoex}.

$\bullet$ For ${\rm I}_b$, one Taub-{\rm NUT} at every fixed point of the canonical $S^1$-action \cite{gw, lebrun}.

$\bullet$ For ${\rm I}_b^*$ with $b > 0$, $b$ copies of Taub-{\rm NUT}
as before, 
plus one copy of the crepant resolution of $(\R^3 \times S^1)/\Z_2$, cf.~\cite{bm}, for each tail of the dual Dynkin diagram.
\end{problem}

One motivation for this is the Kummer surface heuristics from the Introduction, 
which would correspond to four isotrivial ALG spaces with an I$_0^*$ fiber both at the core and at infinity.
Also, along the Gross-Wilson path with $24$ I$_1$ singular fibers,
we expect one Taub-NUT to bubble off from each Ooguri-Vafa space that was glued in. 
Finally, recall from Table 4.1 that in the general setting described in this problem, the
limiting metric on $\P^1$ will have a cone point underneath each singular fiber of
finite monodromy, 
whose angle is the same as the asymptotic cone angle of an
ALG space with the dual fiber removed, which suggests a gluing construction.

\bibliographystyle{amsplain}

\end{document}